\documentclass[12pt,reqno]{amsart}
\usepackage{amsthm,amsfonts,amssymb,euscript}

\newcommand{\K}{\ensuremath{\mathcal{K}}}
\newcommand{\R}{\ensuremath{\mathbb{R}}}
\newcommand{\Z}{\ensuremath{\mathbb{Z}}}
\newcommand{\C}{\ensuremath{\mathbb{C}}}

\newcommand{\e}{\ensuremath{\mathbf{e}}}

\newcommand{\eps}{\ensuremath{\varepsilon}}

\newcommand{\al}{\ensuremath{\alpha}}
\newcommand{\be}{\ensuremath{\beta}}
\newcommand{\ga}{\ensuremath{\gamma}}

\newcommand{\txi}{\ensuremath{\tilde{\xi}_1}}
\newcommand{\philin}{\ensuremath{\phi_{lin}}}
\newcommand{\psilin}{\ensuremath{\psi_{lin}}}
\newcommand{\tphi}{\widetilde{\phi}}

\newcommand{\FF}{\ensuremath{\mathcal{F}}}

\newcommand{\DD}{\ensuremath{\mathbf{D}}}

\newtheorem{theorem}{Theorem}[section]
\newtheorem{lemma}[theorem]{Lemma}
\newtheorem{proposition}[theorem]{Proposition}
\newtheorem{corollary}[theorem]{Corollary}
\newtheorem{definition}[theorem]{Definition}

\numberwithin{equation}{section}
\begin{document}

\title[Global Schr\"{o}dinger maps]
{Global Schr\"{o}dinger maps in dimensions $d\geq 2$: small data in the critical Sobolev spaces}
\author{I. Bejenaru}
\address{Texas A\&M University}
\email{bejenaru@math.tamu.edu}
\author{A. D. Ionescu}
\address{University of Wisconsin -- Madison}
\email{ionescu@math.wisc.edu}
\author{C. E. Kenig}
\address{University of Chicago}
\email{cek@math.uchicago.edu}
\author{D. Tataru}
\address{University of California -- Berkeley}
\email{tataru@math.berkeley.edu}
\thanks{I. B. was supported in part by  NSF grant DMS0738442.
 A. I. was supported in part by a Packard Fellowship. C. K. 
 was supported in part by  NSF grant DMS0456583. D. T. was supported in part by  NSF grant DMS0354539}

\begin{abstract}
We consider the Schr\"{o}dinger map initial-value problem
\begin{equation*}
\begin{cases}
&\partial_t\phi=\phi\times\Delta \phi\,\text{ on }\,\mathbb{R}^d\times\mathbb{R};\\
&\phi(0)=\phi_0,
\end{cases}
\end{equation*}
where $\phi:\mathbb{R}^d\times\mathbb{R}\to\mathbb{S}^2\hookrightarrow \R^3$ is a smooth function. In all dimensions $d\geq 2$, we prove that the Schr\"{o}dinger map initial-value problem admits a unique global smooth solution $\phi\in C(\R:H^\infty_Q)$, $Q\in\mathbb{S}^2$, provided that the data $\phi_0\in H^\infty_Q$ is smooth and satisfies the smallness condition $\|\phi_0-Q\|_{\dot{H}^{d/2}}\ll 1$. We prove also that the solution operator extends continuously to the space of data in $\dot H^{d/2}\cap \dot H^{d/2-1}_Q$ with small $\dot{H}^{d/2}$ norm.

\end{abstract}
\maketitle
\tableofcontents

\section{Introduction}\label{INTRO}

In this paper we consider the Schr\"{o}dinger map initial-value
problem
\begin{equation}\label{Sch1}
  \begin{cases}
    &\partial_t\phi=\phi\times\Delta \phi\,\text{ on }\,\mathbb{R}^d\times\mathbb{R};\\
    &\phi(0)=\phi_0,
  \end{cases}
\end{equation}
where $d\geq 2$ and $\phi:\mathbb{R}^d\times\mathbb{R}
\to\mathbb{S}^2\hookrightarrow\mathbb{R}^3$. The Schr\"{o}dinger map
equation has a rich geometric structure and arises in several
different ways. For instance, it arises in ferromagnetism as the
Heisenberg model for the ferromagnetic spin system whose classical
spin $\phi$, which belongs to $\mathbb{S}^2\hookrightarrow \R^3$, is
given by \eqref{Sch1} in dimensions $d=1,2,3$; we refer the reader to
\cite{ChShUh}, \cite{NaStUh}, \cite{PaTo}, and \cite{La} for more
details. In this paper we are concerned with the issue of global
well-posedness of the initial-value problem \eqref{Sch1}, in the case
of data $\phi_0$ which is small in the critical Sobolev spaces
$\dot{H}^{d/2}$, $d\geq 2$ (see \cite{ChShUh} for results in dimension
$d=1$). Our main result is the direct analogue in the setting of
Schr\"{o}dinger maps of the theorem of Tao \cite{Ta2} on global
regularity of wave maps with small critical Sobolev norms. We also
prove continuous dependence of solutions on the initial data in
certain norms, as in \cite{Tat3}.
 
We start with some notation. Let $\Z_+=\{0,1,\ldots\}$. For
$\sigma\in[0,\infty)$ let $H^{\sigma}=H^\sigma(\mathbb{R}^d)$ denote
the usual Sobolev spaces of complex valued functions on
$\mathbb{R}^d$. For $Q \in\mathbb{S}^2$ we define the metric space
\begin{equation}\label{Sch2}
  H^\sigma_Q=\{f:\mathbb{R}^d:\R^3):|f(x)|\equiv 1\text{ a. e.  and }f-Q\in H^\sigma\},
\end{equation}
with the induced distance $d^\sigma_Q(f,g)=\|f-g\|_{H^\sigma}$. For
simplicity of notation, let $\|f\|_{H^\sigma_Q}=d^\sigma_Q(f,Q)$ for
$f\in H^\sigma_Q$. We also define the metric spaces
\begin{equation*}
  H^\infty=\bigcap_{\sigma\in\Z_+}H^\sigma\,\,\,\text{ and }\,\,\,H^\infty_Q=\bigcap_{\sigma\in\Z_+}H^\sigma_Q,
\end{equation*}
with the induced distances.

Similarly, for $T\in(0,\infty)$ and $\sigma,\rho\in\Z_+$ let
$H^{\sigma,\rho}(T)$ denote the Sobolev spaces of complex valued
functions in $\R^d \times [-T.T]$ with norm
\begin{equation*}
  \begin{split}
    \|f\|_{H^{\sigma,\rho}(T)}=\sup_{t\in(-T,T)}\sum_{\rho'=0}^{\rho}\|\partial_t^{\rho'}f(.,t)\|_{H^{\sigma}}
  \end{split}
\end{equation*}
For $Q \in\mathbb{S}^2$ we also define the metric space
\begin{equation*}
  \begin{split}
    H_Q^{\sigma,\rho}(T)&=\{f: \mathbb{R}^d\times(-T,T):\R^3);
    |f(x,t)|\equiv 1\text{ a. e. and }f-Q\in H^{\sigma,\rho}(T)\},
  \end{split}
\end{equation*}
with the distance induced by the $H^{\sigma,\rho}(T)$ norm. Finally,
we define the metric spaces
\begin{equation*}
  H^{\infty,\infty}(T)=\bigcap_{\sigma,\rho\in\Z_+}H^{\sigma,\rho}(T)\,\,\,\text{ and }\,\,\,H_Q^{\infty,\infty}(T)=\bigcap_{\sigma,\rho\in\Z_+}H_Q^{\sigma,\rho}(T),
\end{equation*}
with the induced distances.

For $f\in H^\infty$ we define the homogeneous Sobolev norms
\begin{equation*}
  \begin{split}
    \|f\|_{\dot{H}^{\sigma}}=\|\FF(f)(\xi)\cdot
    |\xi|^\sigma\|^2_{L^2}, \qquad \sigma \geq 0.
  \end{split}
\end{equation*}
Our first main theorem concerns global existence and uniqueness of
solutions of the initial-value problem \eqref{Sch1} for data
$\phi_0\in H^{\infty}_Q$, with $\|\phi_0-Q\|_{\dot{H}^{d/2}}\ll1$.

\begin{theorem}\label{newth0}
  (Global regularity) Assume $d\geq 2$ and $Q\in\mathbb{S}^2$. Then
  there is $\varepsilon_0 (d)>0$ such that for any $\phi_0\in
  H^{\infty}_Q$ with $\|\phi_0-Q\|_{\dot{H}^{d/2}}\leq
  \varepsilon_0(d)$ there is a unique solution $\phi=S_Q(\phi_0)\in
  C(\R:H^{\infty}_Q)$ of the initial-value problem
  \eqref{Sch1}. Moreover
  \begin{equation}\label{am2}
    \sup_{t\in\R}\|\phi(t)-Q\|_{\dot{H}^{d/2}}\leq C\|\phi_0-Q\|_{\dot{H}^{d/2}},
  \end{equation}
  and, for any $T\in[0,\infty)$ and $\sigma\in\Z_+$,
  \begin{equation}
    \sup_{t\in[-T,T]}\|\phi(t)\|_{H^\sigma_Q}\leq C(\sigma,T,\|\phi\|_{H^\sigma_Q}).
    \label{allsigma}\end{equation}
\end{theorem}

Theorem \ref{newth0} was proved in dimensions $d\geq 4$ by the first
three authors in \cite{BeIoKe}. In addition to this global regularity
result, we also prove a uniform global bound on certain smooth norms
and a well-posedness result, see Theorem \ref{Main1} below. If $\sigma
< d/2$ then the completion of $H^\infty$ with respect to the
$\dot{H}^{\sigma}$ norm is a space of distributions which we denote by
$\dot{H}^{\sigma}$. As above, we set
\[
\dot{H}^{\sigma}_Q = \{f:\R^d \to \R^3; f-Q \in \dot H^\sigma, \ \
|f(x)|\equiv 1 \text{ a.e. in } \R^d\},
\]
In the interesting case $\sigma = d/2$ this is no longer the case.
Instead, the completion of $H^\infty$ with respect to the
$\dot{H}^{d/2}$ norm can be identified with a subspace of the quotient
space of distributions modulo constants.
For this and other technical reasons, in this article we do not
consider the most general problem with initial data in $\dot H^{d/2}$
and instead we restrict ourselves to the smaller initial data space
$\dot H^{d/2} \cap \dot H^{d/2-1}_Q$ where the above difficulty does
not arise. More precisely, for $\sigma\geq d/2$ and $\eps>0$ we define
\begin{equation*}
  \mathcal{B}_\eps^\sigma=\{\phi\in\dot{H}^{d/2-1}_Q\cap\dot{H}^\sigma:\|\phi-Q\|_{\dot{H}^{d/2}}\leq\eps\}
\end{equation*}
with the distance induced by the space
$\dot{H}^{d/2-1}_Q\cap\dot{H}^\sigma$. Our second main theorem
concerns global wellposedness of the initial-value problem
\eqref{Sch1} for initial data in $\mathcal{B}^\sigma_\eps$,
$\sigma\geq d/2$, $\eps\ll 1$.
 
\begin{theorem}\label{Main1} (Uniform bounds and well-posedness)
  Assume $d\geq 2$, $Q\in\mathbb{S}^2$ and $\sigma_1\geq d/2$. Then
  there is $\varepsilon_0(d,\sigma_1)\in (0,\eps_0(d)]$ such that for
  any $\phi_0\in H^{\infty}_Q $ with
  $\|\phi-Q\|_{\dot{H}^{d/2}}\leq\eps_0(d,\sigma_1)$ the global
  solution $\phi=S_Q(\phi_0)\in C(\R:H^{\infty}_Q)$ constructed in
  Theorem \ref{newth0} satisfies the uniform bound
  \begin{equation}\label{am3}
    \sup_{t\in\R}\|\phi(t)-Q\|_{{H}^{\sigma}}\leq 
    C_{\sigma} \|\phi_0-Q\|_{{H}^{\sigma}}, \qquad  d/2 \leq \sigma
    \leq \sigma_1.
  \end{equation}
  In addition, for any $\sigma\in[d/2,\sigma_1]$ the operator $S_Q$
  admits a continuous extension
  \[
  S_Q: \mathcal{B}^\sigma_{\eps_0(d,\sigma_1)}\to C(\R, \dot
  H^{\sigma} \cap \dot H^{d/2-1}_Q).
  \]
\end{theorem}

Our analysis gives more information about the global solution $\phi$;
we can prove for instance that $\nabla\phi$ satisfies all the
Strichartz estimates globally in time. The rough solutions obtained in
Theorem \ref{Main1} as uniform limits of smooth solutions can be also
shown to satisfy the equation \eqref{Sch1} in a suitable
distributional sense.  The global bound \eqref{am3} is sometimes
interpreted as the absence of ``weak turbulence''.

We record also two conservation laws for solutions of the
Schr\"{o}dinger map equation \eqref{Sch1}: if $\phi\in C((T_1,T_2):
H^\infty_Q)$ solves the equation $\partial_t\phi=\phi\times\Delta_x
\phi$ on some time interval $(T_1,T_2)$ then the quantities
\begin{equation}\label{conserve}
  E_0(t)=\int_{\R^d}|\phi(t)-Q|^2\,dx\quad\text{ and }\quad E_1(t)=\int_{\R^d}\sum_{m=1}^d|\partial_m\phi(t)|^2\,dx
\end{equation}
are conserved. In particular, with $S_Q(\phi_0)$ as in Theorem
\ref{Main1},
\[
\|S_Q(\phi_0)(t)\|_{H^0_Q}=\|\phi_0\|_{H^0_Q}, \qquad
\|S_Q(\phi_0)(t)\|_{H^1_Q}=\|\phi_0\|_{H^1_Q}, \quad t \in \R
\]

As mentioned earlier, the direct analogue of Theorem \ref{newth0} in
the setting of wave maps is the theorem of Tao \cite{Ta2}. However,
our proof of Theorem \ref{newth0} is closer to that of \cite{ShSt},
\cite{NaStUh3}, \cite{Kri1}, and \cite{Kri2}, in the sense that we
prove a priori bounds on the derivatives of the Schr\"{o}dinger map
$\phi$, in a suitable gauge, rather than the Schr\"{o}dinger map
itself. See also \cite{KlMa}, \cite{KlSe}, \cite{Tat1}, \cite{Tat2},
\cite{Ta1}, \cite{KlRo}, \cite{ShSt}, \cite{NaStUh3}, \cite{Kri1},
\cite{Kri2}, and \cite{Tat3} for other local and global regularity (or
well-posedness) theorems for wave maps. A complete account of the main
ideas in the work on wave maps can be found in the book \cite[Chapter
6]{TaoBook}.

We remark that, while from the geometric and algebraic points of view
there are many similarities between wave maps and Schr\"{o}dinger
maps, there is a fundamental difference from the analytic point of
view. This is mainly due to the fact that it is much more difficult to
handle perturbatively derivatives in the nonlinearity for
Schr\"{o}dinger equations than for wave equations. This reflects the
fact that wave equations have two time derivatives, while
Schr\"{o}dinger equations have only one, with corresponding effect on
the Cauchy data (see \cite[p. 268]{Tao5} for a related
discussion). Thus, for wave equations, at least in high dimensions,
there are large classes of Strichartz estimates which can be used to
control derivative nonlinearities in a perturbative way. This is not
the case for Schr\"{o}dinger equations. To deal with this problem for
Schr\"{o}dinger equations, Kenig, Ponce, and Vega \cite{KePoVe2}
introduced for the first time a method to obtain local well-posedness
for general derivative nonlinearity Schr\"{o}dinger equations. This
method combines ``local smoothing estimates'', ``inhomogeneous local
smoothing estimates'', which give the crucial gain of one derivative,
and ``maximal function estimates''. Further results are in
\cite{KePoVe3} and \cite{KePoVe4}.

The initial-value problem \eqref{Sch1} has been studied extensively
(also in the case in which the sphere $\mathbb{S}^2$ is replaced by
more general targets). It is known that sufficiently smooth solutions
exist locally in time, even for large data (see, for example,
\cite{SuSuBa}, \cite{ChShUh}, \cite{DiWa2}, \cite{Ga} and the
references therein). Such theorems for (local in time) smooth
solutions are proved using delicate geometric variants of the energy
method. For low-regularity data, the initial-value problem
\eqref{Sch1} has been studied indirectly using the ``modified
Schr\"{o}dinger map equations'' and certain enhanced energy methods
(see, for example, \cite{ChShUh}, \cite{NaStUh}, \cite{NaStUh2},
\cite{KeNa}, \cite{Ka}, and \cite{KaKo}), and directly, in the case of
small data, using fixed point arguments in suitable spaces (see
\cite{IoKe2}, \cite{Be}).

The first global well-posedness result for \eqref{Sch1} in critical
spaces (precisely, global well-posedness for small data in the
critical Besov spaces in dimensions $d\geq 3$) was proved by two of
the authors in \cite{IoKe3}, and independently by the first author
\cite{Be2}. This was later improved to global regularity for small
data in the critical Sobolev spaces in dimensions $d\geq 4$ in
\cite{BeIoKe}. In dimension $d=2$, in the case of equivariant data
with energy close to the energy of the equivariant harmonic map, the
existence of global equivariant solutions (and asymptotic stability)
was proved in \cite{GuKaTs}. In the case of radial or equivariant data
of small energy, global well-posedness was proved in \cite{ChShUh}.

The global results in \cite{IoKe3}, \cite{Be2}, and \cite{BeIoKe} use
in a fundamental way the strong ``local smoothing'', ``inhomogeneous
local smoothing'', and ``maximal function'' spaces
\begin{equation}\label{mainspaces}
  L^{\infty,2}_\e\qquad L^{1,2}_\e\qquad L^{2,\infty}_\e.
\end{equation}
See \eqref{Lpqe} for definitions and a longer discussion. These spaces
were introduced earlier in the study of Schr\"{odinger} maps by two of
the authors in \cite{IoKe2}, and replace the corresponding spaces in
\cite{KePoVe2}--\cite{KePoVe4} (where everything was localized to
finite cubes). It is essential to work with the strong spaces in
\eqref{mainspaces} instead of their localized versions in order to be
able to prove global in time results. The spaces in \eqref{mainspaces}
were first used by Linares and Ponce \cite{LiPo} to study the local
well-posedness of the Davey--Stewartson system. Other uses of such
spaces (implicit or explicit) to prove local-wellposedness are in
\cite{KePoVe5}, \cite{IoKe0}, and \cite{CoIoKeSt}. In the case of
global well-posedness, the strong spaces \eqref{mainspaces} were used
for the first time by two of the authors in \cite{IoKe} in the study
of the Benjamin--Ono equation in $L^2$.

As mentioned earlier, Theorem \ref{newth0} was proved in dimensions
$d\geq 4$ by the first three authors in \cite{BeIoKe}. It is likely
that the proof in \cite{BeIoKe} can be extended to dimension $d=3$,
provided one uses some type of ``dynamical separation" to bound
$\mathrm{High}\times\mathrm{High}\to\mathrm{Low}$ frequency
interactions that appear in the connection coefficients $A_m$ of the
Coulomb gauge (as in \cite{Kri1} and \cite{Kri2} in the case of wave
maps). There are, however, two significant difficulties in dimension
$d=2$. The first main difficulty is related to the maximal function
estimate
\begin{equation}\label{max}
  \|e^{it\Delta}\phi\|_{L^{2,\infty}_\e}\lesssim \|\phi\|_{L^2_x}
\end{equation}
for functions $\phi\in L^2(\R^d)$ with $\mathcal{F}(\phi)$ supported
in $\{\xi\in\R^d:|\xi|\in[1/2,2]\}$. This estimate holds in dimensions
$d\geq 3$ (see \cite{IoKe2}) and plays a key role in the global
results of \cite{IoKe3}, \cite{BeIoKe}, and \cite{Be2}, but fails
``logarithmically'' in dimension $d=2$. Because of this logarithmic
failure, in dimension $d=2$ we replace the space $L^{2,\infty}_\e$ in
the left-hand side of \eqref{max} with a sum of Galilean transforms of
it (see the precise definitions in Section~\ref{functionspaces}).  The
idea of using such sums of spaces as substitutes for missing estimates
in low dimensions is due to Tataru \cite{Tat2}, in the setting of wave
maps, where the Lorentz invariance and Strichartz spaces are used
instead of the Galilean invariance and the maximal function
space. These substitutes have played a key role in all the subsequent
work on global wave maps in dimensions $2$ and $3$.

The second main difficulty is related to the choice of a suitable
system of coordinates (or gauge) for perturbative analysis. The use of
gauges for the equation \eqref{Sch1} was pioneered in \cite{ChShUh},
where orthonormal frames were first used in the context of
Schr\"{o}dinger maps. These constructions have been presented by J.
Shatah, his students and collaborators on several occasions (see
\cite{Ga}, \cite{NaShVeZe} and the references therein). Unlike in
dimensions $d\geq 3$, in dimension $d=2$ it appears that one cannot
use the standard Coulomb gauge, even if the analysis of the
Schr\"{o}dinger equation is combined with elliptic techniques such as
the dynamical separation mentioned earlier. We substitute the Coulomb
gauge with Tao's caloric gauge introduced in \cite{Ta3}, see section
\ref{setup} and \cite[Chapter 6]{TaoBook} for a longer discussion on
the various gauges used in the study of wave maps. As explained in
\cite{Ta3}, this caloric gauge leads to better estimates on the
connection coefficients $A_m$ than the Coulomb gauge, which allow us
to close the perturbative part of the argument.

It is important to notice that the main components of the spaces we
use for our perturbative analysis are the strong local smoothing,
inhomogeneous local smoothing, and maximal function spaces in
\eqref{mainspaces}, as well as Galilean transformations of these
spaces in dimension $d=2$ (see Definition \eqref{Ldef}). In
particular, we do not use $X^{s,b}$-type structures that have been
frequently used in the subject. All of our norms are defined in the
physical space, without the use of the Fourier transform (except for
dyadic localizations), and are very simple, see Definitions
\ref{spacesd>2} and \ref{spacesd2}, at least when compared to the
corresponding spaces used in the study of wave maps in dimensions $2$
and $3$. This reflects the cubic nature of the main part of the
nonlinearity; see also \cite{KTa} for another instance of this
phenomenon. The simplicity of these spaces is due, in part, to the
geometric nature and the efficiency of the caloric gauge, compared to
other gauges used in the study of wave maps and Schr\"{o}dinger maps.

Most of our construction is geometric and can be written in covariant
form. There is one exception, however, namely the definition of the
space $H^0_Q$, which depends on the Euclidean distance
$|\phi(x)-Q|$. The supercritical quantity $\|\phi(t)\|_{H^0_Q}=E_0(t)$
is conserved through the Schr\"{o}dinger map flow, see
\eqref{conserve}. It is useful to have control of such a supercritical
quantity in the construction of the caloric gauge in Proposition
\ref{TaoHeat}, particularly in dimension $d=2$, in order to be able to
prove that the orthonormal frame $v,w$ does indeed trivialize as the
heat time $s$ tends to infinity.

We prefer, however, to adopt the extrinsic point of view throughout
the paper: we think of smooth maps $g:D\to \mathbb{S}^2$, where $D$ is
some domain, as maps $g:D\to\mathbb{R}^3$ (thus $(3\times 1)$
matrices) with $|g|\equiv 1$. With this point of view, an orthonormal
frame of $g^\ast T\mathbb{S}^2$ on $D$ is simply a pair of smooth maps
$v,w:D\to \mathbb{S}^2$ such that ${}^tv\cdot g={}^tw\cdot
g={}^tv\cdot w=0$ on $D$. See \cite[Chapter 6]{TaoBook} for a
discussion on the relation between the intrinsic and the extrinsic
points of view, in the setting of wave maps. The extrinsic formalism
we use in this paper was explained to us by T. Tao \cite{Ta4}.

For the sake of completeness we write the proof of the main theorems
in all dimensions $d\geq 2$. We emphasize however that many of the
difficulties are only present in dimension $d=2$. In dimensions $d\geq
3$ the main normed spaces $F_k(T)$, $G_k(T)$, and $N_k(T)$ are
simpler. Also the analysis related to
$\mathrm{High}\times\mathrm{High}\to\mathrm{Low}$ frequency
interactions, which motivates the use of the caloric gauge, is easier.

We thank T. Tao for a discussion on the benefits of the caloric gauge
and for explaining us the elementary extrinsic formalism we use in
this paper.

\section{The differentiated equations and the caloric gauge}
\label{setup}

In this section we start with a smooth solution $\phi$ to the
Schr\"odinger map equation and a smooth orthonormal frame $(v,w)$ in
$T_\phi \mathbb S^2$. Then we construct the fields $\psi_m$ and the
connection coefficients $A_m$, and derive the differentiated
(modified) Schr\"{o}dinger map equations satisfied by these functions
(see \cite{ChShUh}, \cite{NaStUh} where the modified Schr\"{o}dinger
map equations were introduced, and \cite{NaShVeZe} for a detailed
discussion on the connection between the modified Schr\"{o}dinger maps
and the original equation).  Next we introduce the caloric gauge.
This is done by solving first a covariant heat equation which leads to
an extension of smooth Schr\"odinger maps to parabolic time $s \in
[0,\infty)$.  We then construct the orthonormal frame $(v,w)$ by
solving an ordinary differential equation with data prescribed at
infinity in order to construct the orthonormal frame $(v,w)$. This
construction is due to Tao \cite{Ta3}.

The Schr\"{o}dinger map equation leads to the system \eqref{schcov2}
of $d$ scalar Schr\"{o}dinger equations satisfied by the fields
$\psi_m$, $m=1,\ldots,d$, at heat time $s=0$.  The caloric gauge
condition allows us to express the connection coefficients $A_m$ in
terms of the parabolic extensions of the differentiated fields
$\psi_m$, see \eqref{Aform}. Finally, we derive the linearized
Schr\"{o}dinger map equation and we express it in the frame form
\eqref{schlin}.

We begin with a smooth function $\phi: \R^d\times(-T,T) \to \mathbb
S^2$.  Instead of working directly on the equation \eqref{Sch1} for
the function $\phi$ it is convenient to study the equations satisfied
by its derivatives $\partial_m \phi(x,t)$ for $m = 1,d+1$, where
$\partial_{d+1}=\partial_t$. These are tangent vectors to the sphere
at $\phi(x,t)$. Suppose we have a smooth orthonormal frame
$(v(t,x),w(t,x))$ in $T_{\phi(x,t)}\mathbb S^2$. Then we can introduce
the differentiated variables,
\begin{equation}\label{definitionso}
  \psi_m={}^tv\cdot\partial_m {\phi}
+i\ {}^tw\cdot\partial_m {\phi}.
\end{equation}
Thus we can express  $\partial_m \phi$ in the $(v,w)$  frame as 
\begin{equation}\label{basis}
 \partial_m {\phi}=v\Re(\psi_m)+w\Im(\psi_m).
\end{equation}
In order to write the equations for $\psi_m$ we need 
to know how $v$ and $w$ vary as functions of $(x,t)$.
For this we introduce the real coefficients 
\begin{equation}\label{definitions1}
A_m={}^tw\cdot\partial_m v.
\end{equation}
In particular this  allows us to  complement \eqref{basis}
with 
\begin{equation}\label{basis1}
\begin{cases}
&\partial_mv=-\phi\Re(\psi_m)+wA_m;\\
&\partial_mw=-{\phi}\Im(\psi_m)-vA_m.
\end{cases}  
\end{equation}

The variables $\psi_m$ are not independent, instead they satisfy the 
curl type relations
\begin{equation}\label{id1}
(\partial_l+iA_l)\psi_m=(\partial_m+iA_m)\psi_l.
\end{equation}
Thus with  the notation $\DD_m=\partial_m+iA_m$ we can rewrite 
this as 
\begin{equation}\label{id2}
\DD_l \psi_m= \DD_m \psi_l.
\end{equation}
A direct computation using the definition of $A_m$ shows that
\begin{equation}\label{id3}
\partial_lA_m-\partial_mA_l=\Im(\psi_l\overline{\psi_m})=q_{lm}.
\end{equation}
Thus the curvature of the connection is given by
\begin{equation}\label{id4}
\DD_l\DD_m-\DD_m\DD_l=iq_{lm}.
\end{equation}

Assume now that the smooth function $\phi$ satisfies the
Schr\"{o}dinger map equation $\partial_t\phi=\phi\times\Delta_x\phi$.
Then we derive the Schr\"odinger equations for the functions $\psi_m$.
A direct computation, using \eqref{id1}, \eqref{id3},  $\phi\times v=w$, and
$\phi\times w=-v$, shows that
\begin{equation}\label{schcov}
\psi_{d+1}=i\sum_{l=1}^d\DD_l\psi_l.
\end{equation}
Using \eqref{id2} and \eqref{id4}, it follows that for $m=1,\ldots,d$
\begin{equation}
\DD_{d+1}\psi_m=i\sum_{l=1}^d\DD_l\DD_l\psi_m+\sum_{l=1}^dq_{lm}
\psi_l,
\label{Dpsim}\end{equation}
which is equivalent to
\begin{equation}\label{schcov2}
(i\partial_t+\Delta_x)\psi_m=-2i\sum_{l=1}^d \!
A_l\partial_l\psi_m+\big(A_{d+1}+\sum_{l=1}^d (A_l^2-i\partial_lA_l)\big)\psi_m
-i\sum_{l=1}^d\! \psi_l\Im(\overline{\psi}_l\psi_m).
\end{equation}

Consider the system of equations which consists of \eqref{id1},
\eqref{id3} and \eqref{schcov2}.  The solution $\{\psi_m\}$ for the
above system cannot be uniquely determined as it depends on the choice
of the orthonormal frame $(v,w)$. Precisely, it is invariant with
respect to the gauge transformation
\[
\psi_m \to e^{i\theta} \psi_m, \qquad A_m \to A_m + \partial_m \theta.
\]
In order to obtain a well-posed system one needs to make 
a choice which uniquely determines the gauge. Ideally one may hope
that this choice  uniquely determines the $A_m$'s in terms of the 
$\psi_m$'s in a way that makes the nonlinearity in \eqref{schcov2}
perturbative. 

A natural choice for this is the Coulomb gauge, where one adds 
the equation
\[
\sum_{m=1}^d \partial_m A_m = 0
\]
which in view of \eqref{id3} leads to 
\begin{equation}
A_m = \Delta^{-1} \sum_{l=1}^d \partial_l \Im (\bar{\psi_l} \psi_m),
\qquad m=1,\ldots,d.
\label{cou}\end{equation}
This choice works well in high dimension $d \geq 4$, see \cite{BeIoKe}; 
however, in low dimension there are difficulties caused by the
contributions of two high frequencies in $\psi$ to the low frequencies
in $A$.

This is what causes us to look for a different choice of gauge, namely
the caloric gauge. This was proposed in \cite{Ta3} in the context of
the wave map equation, and then as a possible gauge for Schr\"odinger
maps \cite{Ta4}.
  
Precisely, at each time $t$ we solve a covariant heat equation
with $\phi(t)$ as the initial data,
\begin{equation}\label{heat3}
\begin{cases}
&\partial_s {\tphi}=\Delta_x {\tphi}+{\tphi}\cdot\sum_{m=1}^d
|\partial_m{\tphi}|^2\quad\text{ on }[0,\infty)\times\R^d;\\
&{\tphi}(0,t,x)=\phi(t,x).
\end{cases}
\end{equation}
We heuristically remark that as the heat time $s$ approaches infinity,
the solution $\phi(s)$ approaches the equilibrium state $Q$. This is
related to our assumption that the ``mass'' $E_0$ of $\phi_0$ is
finite, and would not necessarily be true otherwise. This would allow
us to arbitrarily pick $(v_\infty,w_\infty)$ at $s = \infty$ as an
arbitrary orthonormal base in $T_Q\mathbb S^2$, independently of $t$
and $x$.  To define the orthonormal frame $(v,w)$ for all $s \geq 0$
we pull back $(v_\infty,w_\infty)$ along the backward heat flow using
parallel transport. This translates into the relation
\begin{equation}
{}^tw\cdot\partial_s v=0
\label{transport}\end{equation}
The existence of a global smooth solution for the caloric equation
\eqref{heat3} and of the corresponding frame $(v,w )$ is proved in
Proposition~\ref{TaoHeat}. In particular for each $F \in
\{\tphi-Q,v-v_\infty,w-w_\infty\}$ the following decay properties are
valid:
\begin{equation}
| \partial_x^\alpha F(s)| \leq c_\alpha \langle s \rangle
^{-(|\alpha|+1)/2}, \qquad s \geq 0
\label{pointdecay}\end{equation} 
  
  Setting $\partial_0 = \partial_s$ we can  define 
the functions $\psi_m$ and $A_m$ for all $s \in [0,\infty)$
and $m = 0, \cdots, d+1$ by 
\begin{equation}\label{definitions}
\begin{cases}
&\psi_m={}^tv\cdot\partial_m\tphi+
i{}^tw\cdot\partial_m\tphi;\\
&A_m={}^tw\cdot\partial_m v.
\end{cases}
\end{equation}
Then the relations \eqref{id1}-\eqref{id4} hold for all $l,m = 0,d+1$.
In addition, the parallel transport relation ${}^tw\cdot\partial_s v=0$ yields
the main gauge condition
\begin{equation}
A_0 = 0.
\end{equation}

As in the case of the Schr\"odinger equation, a direct computation
using the heat equation \eqref{heat3} and \eqref{id1}, \eqref{id3}
shows that
\begin{equation}\label{heatcov}
\psi_0=\sum_{l=1}^d \DD_l \psi_l.
\end{equation}
Thus, using again \eqref{id3}, for any $m=1,\ldots,d+1$
\begin{equation*}
\begin{split}
\partial_0\psi_m&=\DD_m\psi_0=\sum_{l=1}^d\DD_m\DD_l\psi_l=\sum_{l=1}^d\DD_l\DD_m\psi_l+i\sum_{l=1}^dq_{ml}\psi_l\\
&=\sum_{l=1}^d\DD_l\DD_l\psi_m+i\sum_{l=1}^d\Im(\psi_m\overline{\psi_l})\psi_l,
\end{split}
\end{equation*}
which is equivalent to
\begin{equation}\label{heatcov2}
(\partial_s-\Delta_x)\psi_m=2i\sum_{l=1}^dA_l\partial_l\psi_m-\sum_{l=1}^d(A_l^2-i\partial_lA_l)\psi_m+i\sum_{l=1}^d\Im(\psi_m\overline{\psi_l})\psi_l.
\end{equation}

On the other hand from \eqref{id3} we obtain
\[
\partial_s A_m = \Im (\psi_0 \overline{\psi_m}).
\]
Due to \eqref{goodbounds4} and  \eqref{goodbounds5} we can integrate back from 
$s=\infty$ to obtain
\begin{equation}\label{Aform}
\begin{split}
A_m(s)=-\int_{s}^{\infty}\Im(\psi_0\overline{\psi_m})(r)\,dr=-\sum_{l=1}^d\int_s^{\infty}\Im\big(\overline{\psi_m}(\partial_l\psi_l+iA_l\psi_l)\big)(r)\,dr,
\end{split}
\end{equation}
for any $m=1,\ldots,d+1$ and $s\in[0,\infty)$. Then ${A_m}_{|s=0}$ 
represents our choice of the gauge for the Schr\"odinger map equation.
The reason we prefer the caloric gauge to the Coulomb gauge
is the way the high-high frequency interactions are handled.
Indeed, while \eqref{cou} can be conceptually written in the
form
 \[
A \approx \sum_{j < k} 2^{-k} P_j \psi P_k \psi   + 
\sum_{j \leq k} 2^{-j} P_j(P_k \psi P_k \psi),
\] 
substituting the first approximation $\psi(s) \approx e^{s \Delta}
\psi(0)$ in \eqref{Aform} yields the relation
\begin{equation}\label{schem}
A \approx \sum_{j < k} 2^{-k} P_j \psi P_k \psi   + 
\sum_{j \leq k} 2^{-k} P_j(P_k \psi P_k \psi).
\end{equation}
This has a better frequency factor in the high $\times$ high $\to $
low frequency interactions.
 
We consider now linearized Schr\"odinger map equations. This is necessary in order to
establish the continuous dependence of the solutions on the initial
data. The linearized equation along a Schr\"odinger map $\phi$ has 
the form
\begin{equation}
\partial_t \philin = \philin \times \Delta \phi + \phi \times \Delta \philin
\label{eqlin}\end{equation}
where 
\[
{}^t \philin \cdot \phi = 0.
\]
Then we can  express $\philin$ in the $(v,w)$ frame 
by setting
\begin{equation}
\psilin = {}^tv\cdot {\philin}
+i\ {}^tw\cdot {\philin},
\end{equation}
or equivalently
\begin{equation}
 {\philin}=v\Re(\psilin)+ w\Im(\psilin). 
\end{equation}
The field $\psilin$ satisfies the linearized equation
\begin{equation}\label{schlin}
\begin{split}
  &(i\partial_t+\Delta_x)\psilin\\
  &=-2i\sum_{l=1}^dA_l\partial_l\psilin+
\big(A_{d+1}+\sum_{l=1}^d(A_l^2-i\partial_lA_l)\big)\psilin-
i\sum_{l=1}^d\psi_l\Im(\overline{\psi_l}\psilin).
\end{split}
\end{equation}
This can be derived by direct computations as before. Heuristically,
one can also think of a one parameter family of solutions $\phi(h)$
for the Schr\"odinger map equation so that $\phi(0)=\phi$ and
$\partial_h \phi_{|h=0} = \philin$, and extend the frame $(v,w)$ as
$h$ varies. Then we have $\psilin = \psi_h$ and we can use
\eqref{Dpsim} with $m$ replaced by $h$.

\section{Function spaces}\label{functionspaces}

In this section we define our main function spaces and derive some of
their properties. We define first several cutoff functions and
(smooth) projection operators.
 
\begin{definition}\label{etas}
  We fix $\eta_0:\mathbb{R}\to[0,1]$ a smooth even function supported
  in the set $\{\mu\in\mathbb{R}:|\mu|\leq 8/5\}$ and equal to $1$ in
  the set $\{\mu\in\mathbb{R}:|\mu|\leq 5/4\}$. We define
  \begin{equation*}
    \chi_j(\mu)=\eta_0(\mu/2^j)-\eta_0(\mu/2^{j-1}),\quad\chi_{\leq j}=\eta_0(\mu/2^j),\qquad j\in\Z.
  \end{equation*}
  Let $P_k$ denote the operator on $L^\infty(\R^d)$ defined by the
  Fourier multiplier $\xi\to\chi_k(|\xi|)$. For any interval
  $I\subseteq\R$, let $\chi_I=\sum_{j\in I}\chi_j$ and let $P_I$
  denote the operator on $L^\infty(\R^d)$ defined by the Fourier
  multiplier $\xi\to\sum_{k\in I}\chi_k(|\xi|)$. For simplicity of
  notation, we define $P_{\leq k}=P_{(-\infty,k]}$. For any
  $\e\in\mathbb{S}^{d-1}$ and $k\in\Z$ we define the operators
  $P_{k,\e}$ by the Fourier multipliers $\xi\to\chi_k(\xi\cdot\e)$.
\end{definition}

To motivate our choice of spaces, recall the Schr\"{o}dinger
nonlinearities, see \eqref{schcov2}
\begin{equation}\label{Schnonlin}
  L_m=-2i\sum_{l=1}^d A_l\partial_l\psi_m+\big(A_{d+1}+\sum_{l=1}^d (A_l^2-i\partial_lA_l)\big)\psi_m
  -i\sum_{l=1}^d\! \psi_l\Im(\overline{\psi}_l\psi_m).
\end{equation}
We would like to analyze these nonlinearities perturbatively in
suitable spaces. The main difficulty is caused by the magnetic terms
$-2i\sum_{l=1}^d A_l\partial_l\psi_m$. Using \eqref{schem} (for
simplicity consider only the terms corresponding to $k=j$) they can be
written schematically in the form
\begin{equation}\label{schem2}
  \sum_{k,k'\in\Z} 2^{-k} P_k \psi P_k \psi \cdot 2^{k'}P_{k'}\psi.
\end{equation}
The difficulty is to estimate the sum over $k\leq k'-100$ in
\eqref{schem2}. The main ingredient needed to estimate such magnetic
terms in Schr\"{o}dinger problems is the local smoothing phenomenon:
we place the highest frequency factor of the nonlinearity in the local
smoothing space $L^{\infty,2}_\e$ (see definition below) and attempt
to estimate the nonlinearity in the inhomogeneous local smoothing
space $L^{1,2}_\e$. This allows us to barely recover the full
derivative loss of the magnetic nonlinearity. This approach, with
certain local smoothing and inhomogeneous local smoothing spaces
localized to cubes, was first used in \cite{KePoVe2} to study local
well-posedness of Schr\"{o}dinger equations with general derivative
nonlinearities. To prove global results, it is essential to work with
stronger local smoothing/inhomogeneous local smoothing spaces, which
exploit better the geometry of Euclidean spaces. A low-regularity
global result using such stronger local smoothing spaces was proved by
two of the authors in \cite{IoKe}.

This scheme was first used in the setting of Schr\"{o}dinger maps by
two of the authors in \cite{IoKe2} and played a key role in all the
global results in \cite{IoKe3}, \cite{BeIoKe}, and \cite{Be2}. In
order for this scheme to work in the Schr\"{o}dinger map problem we
need to be able to control the $L^{2,\infty}_\e$ norm (in a scale
invariant way) of the low frequency terms.

For a unit vector $\e \in \mathbb S^{d-1}$ we denote by $H_\e$ its
orthogonal complement in $\R^d$ with the induced measure.  We define
the lateral spaces $L^{p,q}_\e$ with norms
\begin{equation}\label{Lpqe}
  \|h\|_{L^{p,q}_{{\e}}}=
  \Big[\int_{\mathbb{R}}\Big[\int_{H_\mathbf{e}\times \mathbb{R}}
  |h(x_1\mathbf{e}+x',t)|^q\,dx'dt\Big]^{p/q}\,dx_1\Big]^{1/p}
\end{equation}
with the usual modifications when $p=\infty$ or $q=\infty$. The key
spaces in this family are the local smoothing space $L^{\infty,2}_\e$
and the inhomogeneous local smoothing space $L^{1,2}_\e$. These are
the only spaces in this family that can be used to analyze
perturbatively magnetic nonlinearities, such as $L_m$. This is in
sharp contrast with the case of wave equations, where large classes of
Strichartz estimates can be used to control magnetic nonlinearities,
at least in high dimensions. To make the transition between the local
smoothing and the inhomogeneous local smoothing spaces we use the
maximal function spaces $L^{2,\infty}_\e$.
  
The following local smoothing/maximal function estimates were proved
by two of the authors in \cite{IoKe2} and \cite{IoKe3}:
 
 \begin{lemma}\label{keyboundsd>2}
   If $f\in L^2(\R^d)$, $k\in\Z$, and $\e\in\mathbb{S}^{d-1}$ then
   \begin{equation}
     \|e^{it \Delta} P_{k,\e} f\|_{L^{\infty,2}_\e} \lesssim  
     2^{ -k/2} \|f\|_{L^2}.
     \label{limaxa}\end{equation}
   In addition, if $d\geq 3$ then
   \begin{equation}
     \|e^{it \Delta} P_k  f\|_{L^{2,\infty}_\e} \lesssim  
     2^{k(d-1)/2} \|f\|_{L^2}.
     \label{limaxb}\end{equation}
 \end{lemma}
 
 It is easy to see that the two bounds \eqref{limaxa} and
 \eqref{limaxb} cooperate in the right way to allow us to estimate the
 expression in \eqref{schem2} in the inhomogeneous local smoothing
 space $L^{1,2}_\e$. The bounds \eqref{limaxa} and \eqref{limaxb} are
 not hard to prove. They depend, however, on delicate global
 properties of the Euclidean geometry.

 Unfortunately, the maximal function bound \eqref{limaxb} fails in
 dimension $d=2$, which causes considerable difficulties. To handle
 this case we need to use the Galilean invariance: if $f$ solves
 $(i\partial_t+\Delta_x)f=0$ in $\R^d\times\R$, then $T_w(f)$ solves
 the same equation, where $T_w$, $w\in\R^d$, is the Galilean operator
 defined below.

\begin{definition}\label{Ldef}
  Assume $d=2$, $p,q\in[1,\infty]$, $\e\in\mathbb{S}^{1}$,
  $\lambda\in\R$ and $W\subseteq\R$ finite. We define the spaces
  $L^{p,q}_{e,\lambda}$ using the norms
  \begin{equation}\label{operatorT}
    \begin{split}
      &\|h\|_{L^{p,q}_{\mathbf{e},\lambda}}=\|T_{\lambda\e}(h)\|_{L^{p,q}_\e}=
      \Big[\int_{\mathbb{R}}\Big[\int_{H_\mathbf{e}\times
        \mathbb{R}}|h((x_1+\lambda t)
      {\e}+x',t)|^q\,dx'dt\Big]^{p/q}\,dx_1\Big]^{1/p};\\
      &T_w(h)(x,t)=e^{-ix\cdot w/2}e^{-it|w|^2/4}h(x+tw,t).
    \end{split}
  \end{equation}
  with the usual modifications when $p=\infty$ or $q=\infty$. Then we
  define the spaces $L^{p,q}_{\e,W}$
  \begin{equation*}
    \begin{split}
      L^{p,q}_{\mathbf{e},W}=\sum_{\lambda\in
        W}L^{p,q}_{\mathbf{e},\lambda},\qquad\|h\|_{L^{p,q}_{\mathbf{e},W}}=\inf_{h=\sum_{\lambda\in
          W}h_\lambda}\sum_{\lambda\in
        W}\|h_\lambda\|_{L^{p,q}_{\mathbf{e},\lambda}}.
    \end{split}
  \end{equation*}
\end{definition}

In what follows we fix some large integer $\K$ and define, for $k \in
\Z$,
\begin{equation*}
  W_k=W_k(\mathcal{K})=\{\lambda\in[-2^k,2^k]:2^{k+2\mathcal{K}}\lambda\in\Z\}.
\end{equation*}

\begin{lemma}
  Let $d=2$. For any $f \in L^2$, $k\in\Z$, and $\e \in \mathbb S^1$
  we have
  \begin{equation}
    \|e^{it \Delta} P_{k,\e}f
    \|_{L^{\infty,2}_{\e,\lambda}}  \lesssim 2^{-k/2} 
    \|f\|_{L^2}, \qquad |\lambda| \leq 2^{k-40}.
    \label{litl}\end{equation}
  In addition, if $T \in (0,2^{2\K}]$ then
  \begin{equation}
    \| 1_{[-T,T]}(t) e^{it \Delta} P_kf\|_{L^{2,\infty}_{\e,W_{k+40}}}  \lesssim 2^{k/2} 
    \|f\|_{L^2}.
    \label{ltil}\end{equation}
  \label{l4s}\end{lemma}
The bound \eqref{litl} is a straightforward consequence of
\eqref{limaxa} via a Galilean transformation. The main novelty is the
estimate \eqref{ltil}, which provides a usable replacement to
\eqref{limaxb} in dimension $d=2$. Indeed, it is easy to see that the
bounds \eqref{litl} and \eqref{ltil} can still be used to estimate the
expression in \eqref{schem2} in the space $L^{1,2}_{\e,W_{k-40}}$,
which is an acceptable inhomogeneous local smoothing space in
dimension $d=2$. The idea of using sums of spaces such as
$L^{2,\infty}_{\e,W_{k+40}}$ as substitutes for missing estimates in
the setting of wave maps is due to Tataru \cite{Tat2}.

Limiting the time $T$ to the interval $(0,2^{2\K}]$ is what allows us
to use the discretization which is given by the $W_k$ sets. One could
also allow $T$ to be arbitrarily large, at the expense of replacing
the discrete sums in the definition of the $
L^{2,\infty}_{\e,W_{k+40}}$ with a continuous counterpart. We do not
pursue this here in order to avoid distracting technicalities.

Once the main terms of the Schr\"{o}dinger nonlinearities
\eqref{Schnonlin} are under control, we can use various
Strichartz-type estimates to estimate the remaining terms. We state
below the Strichartz-type estimates we need in this paper; at this
stage many variations are possible. Let $p_d=(2d+4)/d$ denote the
Strichartz exponent.

\begin{lemma}\label{variousStrichartz}
  If $f \in L^2(\R^d)$  then we have the Strichartz estimates
  \begin{equation*}
    \|e^{it \Delta} f\|_{L^{p_d}_{x,t}} \lesssim \|f\|_{L^2},
  \end{equation*}
  and the maximal function bounds
  \begin{equation*}
    \|e^{it \Delta} P_k f\|_{L^{p_d}_x{ L^\infty_t}} \lesssim 
    2^{kd/(d+2)}\|f\|_{L^2}, \qquad k \in \Z
  \end{equation*}
  In addition, if $2/p+d/q=d/2$ and $2 \leq q < 2d(d-2)$ then
  \begin{equation*}
    \sup_{\e\in\mathbb{S}^1}\|e^{it \Delta} P_k f\|_{L^{p,q}_\e} \lesssim_p 
    2^{k(\frac{d+2}{dp}-\frac12)} \|f\|_{L^2}, \qquad p \leq q,
  \end{equation*}
  respectively
  \begin{equation*}
    \sup_{\e\in\mathbb{S}^1}\|e^{it \Delta} P_{k,\e} f\|_{L^{p,q}_\e} \lesssim  
    2^{k(\frac{d+2}{dp}-\frac12)} \|f\|_{L^2}, \qquad p \geq q.
  \end{equation*}
\end{lemma}

The first bound in Lemma \ref{variousStrichartz} is the original
Strichartz estimate \cite{Str}. The second bound follows by
scaling. The last two bounds, which we call lateral Strichartz
estimates, follow informally by interpolation between the $L^{p_d}$
Strichartz estimate and the local smoothing/maximal function estimates
of Lemma \ref{l4s}. The results
stated in Lemma~\ref{keyboundsd>2}, Lemma~\ref{l4s}, and
Lemma~\ref{variousStrichartz} are summarized and proved later in
Lemma~\ref{linhom} (we prove in fact a slightly stronger version of
\eqref{ltil}, which is needed to prove full inhomogeneous estimates).

We are now ready to define our main dyadic function spaces $F_k(T)$,
$G_k(T)$ and $N_k(T)$. Assume that $T\in\R$. For $k\in\Z$ let
$I_{k}=\{\xi\in\R^d:|\xi|\in[2^{k-1},2^{k+1}]\}$
and 
\[
L^2_{k}(T)=\{f\in L^2(\R^d\times[-T,T]):\mathcal{F}(f)\text{
  is supported in }I_{k}\times\R\},
\]

\begin{definition}\label{spacesd>2}
  Assume $d \geq 3$, $T\in\R$, and $k\in\Z$. Then $F_k(T)$, $G_k(T)$
  and $N_k(T)$ are the Banach spaces of functions in $L^2_{k}(T)$ for
  which the corresponding norms are finite:
  \begin{equation}
    \begin{split}
      \| \phi \|_{F_k(T)} = &\ \|\phi\|_{L^\infty_t L^2_x} +
      \|\phi\|_{L^{p_d}} + 2^{-kd/(d+2)} \|\phi\|_{L^{p_d}_x
        L^\infty_t} + 2^{-k(d-1)/2} \sup_{\e \in \mathbb S^{d-1}}
      \|\phi\|_{L^{2,\infty}_\e},
    \end{split}
    \label{fkdef3}\end{equation}
  \begin{equation}
    \begin{split}
      \| \phi \|_{G_k(T)} = \| \phi \|_{F_k} + 2^{k/2} \sup_{|j-k|
        \leq 20} \sup_{\e \in \mathbb S^{d-1}} \| P_{j,\e}
      \phi\|_{L^{\infty,2}_\e},
    \end{split}
    \label{gkdef3}\end{equation}
  respectively
  \begin{equation}
    \|f\|_{N_k(T)} = \inf_{f=f_1+f_2} \|f_1\|_{L^{p'_d}}+ 2^{-k/2}  
    \sup_{\e \in \mathbb S^{d-1}} \| f_2\|_{L^{1,2}_\e}.
    \label{nkdef3}\end{equation}
\end{definition}

\begin{definition} \label{spacesd2} Assume that $d =2$, $k\in\Z$,
  $\mathcal{K}\in\Z_+$, and $T \in (0,2^{2\K}]$. Recall the definition
  $W_k=\{\lambda\in[-2^k,2^k]:2^{k+2\mathcal{K}}\lambda\in\Z\}$. For
  $\phi\in L^2(\R^d\times[-T,T])$ let
  \begin{equation}
    \| \phi \|_{F_k^0(T)} =\|\phi\|_{L^\infty_t L^2_x} +
    \|\phi\|_{L^{4}} + 2^{-k/2} \|\phi\|_{L^{4}_x
      L^\infty_t} + 2^{-k/2} \sup_{\e \in \mathbb S^{1}}
    \|\phi\|_{L^{2,\infty}_{\e,W_{k+40}}}.
    \label{fkodef2}\end{equation}
  We define $F_k(T)$, $G_k(T)$ and $N_k(T)$ as the normed spaces of
  functions in $L^2_{k}(T)$ for which the corresponding norms are
  finite:
  \begin{equation}
    \| \phi \|_{F_k(T)}=  \inf_{J,m_1,\ldots,m_J\in\Z_+}
    \inf_{f=f_{m_1}+\ldots+f_{m_J}}
    \sum_{j=1}^J2^{m_j} \|f_{m_j}\|_{F_{k+m_j}^0},
    \label{fkdef2}\end{equation}
  \begin{equation}
    \begin{split}
      \| \phi \|_{G_k(T)} = & \ \| \phi \|_{F_k^0} + 2^{-k/6} \sup_{\e
        \in \mathbb S^{1}} \|\phi\|_{L^{3,6}_{\e}}+ 2^{k/6}
      \sup_{|j-k| \leq 20} \sup_{\e \in \mathbb S^{1}} \|P_{j,\e}
      \phi\|_{L^{6,3}_{\e}}
      \\
      &\ + 2^{k/2} \sup_{|j-k| \leq 20} \sup_{\e \in \mathbb S^{1}}
      \sup_{|\lambda| < 2^{k-40}} \| P_{j,\e}
      \phi\|_{L^{\infty,2}_{\e,\lambda}},
    \end{split}
    \label{gkdef2}\end{equation}
  respectively
  \begin{equation}
    \begin{split}
      \|f\|_{N_k(T)} = \inf_{f=f_1+f_2+f_3+f_4}\|f_1\|_{L^{\frac43}} +
      2^{\frac{k}6}\|f_2\|_{L^{\frac32,\frac65}_{\e_1}}+
      2^{\frac{k}6}\|f_3\|_{L^{\frac32,\frac65}_{\e_2}} +
      2^{-\frac{k}2} \sup_{\e \in \mathbb S^{1}} \|
      f_4\|_{L^{1,2}_{\e,W_{k-40}}},
    \end{split}
    \label{nkdef2}\end{equation}
  where $(\e_1,\e_2)$ is the canonical basis in $\R^2$.
\end{definition}

In all dimensions $d\geq 2$ the spaces $N_k(T)$ and $G_k(T)$ and
related by the following linear estimate, which is proved in
Section~\ref{SPACES}.

\begin{proposition}\label{linearmainrep}
  (Main linear estimate) Assume $\mathcal{K}\in\Z_+$,
  $T\in(0,2^{2\mathcal{K}}]$ and $k\in\Z$. Then for each $u_0 \in L^2$
which is frequency localized in $I_k$ and any $h \in N_k(T)$ the solution
$u$ to
\[
(i\partial_t+\Delta_x)u=h, \qquad u(0) = u_0
\]
satisfies
 \begin{equation*}
    \|u\|_{G_k(T)}\lesssim  \|u(0)\|_{L^2_x}+\|h\|_{N_k(T)}
  \end{equation*}
\end{proposition}

We describe now the structure of the normed spaces $G_k(T)$, $N_k(T)$,
and $F_k(T)$. As Proposition \ref{linearmainrep} suggests, we use the
spaces $G_k(T)$ to measure solutions of Schr\"{o}dinger equations. The
main components of the spaces $G_k(T)$ are given by the local
smoothing/maximal function spaces $L^{2,\infty}_{\e,W_{k+40}}$ and
$L^{\infty,2}_{\e,\lambda}$ in dimension $d=2$, respectively
$L^{2,\infty}_{\e}$ and $L^{\infty,2}_{\e}$ in dimensions $d\geq 3$
(compare with Lemma \ref{keyboundsd>2} and Lemma \ref{l4s}). The other
components of the spaces $G_k(T)$ are Strichartz-type spaces, compare
with Lemma \ref{variousStrichartz}. These components are much more
flexible.

As Proposition \ref{linearmainrep} suggests, we use the spaces
$N_k(T)$ to measure nonlinearities of Schr\"{o}dinger equations. The
key components of the spaces $N_k(T)$ are the inhomogeneous local
smoothing spaces $L^{1,2}_{\e,W_{k-40}}$ in dimension $d=2$, and
$L^{1,2}_{\e}$ in dimension $d\geq 3$, which are the only spaces that
can be used to bound the difficult magnetic parts of the
Schr\"{o}dinger nonlinearities. The other components of the spaces
$N_k(T)$ are Strichartz-type spaces, and are chosen in a way that
matches the Strichartz spaces of $G_k(T)$.

We discuss now the spaces $F_k(T)$. It is clear from the definition
that $G_k(T)\hookrightarrow F_k(T)$. The larger spaces $F_k(T)$ have
an important advantage over the spaces $G_k(T)$: for any $k\in\Z$ and
$f\in F_k(T)\cap F_{k+1}(T)$ we have $\|f\|_{F_k(T)}\approx
\|f\|_{F_{k+1}(T)}$. This is easy to see by examining the definitions
and noticing that $\|\phi\|_{F_{k+1}^0(T)}\lesssim
\|\phi\|_{F_k^0(T)}$ for any $\phi\in L^2(\R^d\times[-T,T])$ if
$d=2$. Moreover, if $k,k'\in\Z$, $|k-k'|\leq 20$, $u\in F_{k'}(T)$,
and $v\in L^\infty(\R^d\times(-T,T))$ then
\begin{equation}\label{motiv1}
  \|P_k(uv)\|_{F_k(T)}\lesssim \|u\|_{F_{k'}(T)}\|v\|_{L^\infty_{x,t}}.
\end{equation}
The spaces $G_k(T)$ do not have this important property, proved in
Lemma \ref{building}, mostly because of the local smoothing norms
which require certain frequency localizations.

We use the spaces $G_k(T)$ to measure the fields $\psi_m$,
$m=1,\ldots,d$, at parabolic time $s=0$. This is consistent with the
Schr\"{o}dinger equations \eqref{schcov2} satisfied by these
fields. We use, however, the weaker spaces $F_k(T)$ to measure the
fields $\psi_m(s)$ for $s>0$, as well as the connection coefficients
$A_m(s)$. The fields $\psi_m(s)$ satisfy the covariant heat equations
\eqref{heatcov2}, and we are able to propagate control of these fields
along the heat flow, with suitable parabolic decay, only in the larger
spaces $F_k(T)$. Fortunately, in the perturbative analysis of
Schr\"{o}dinger equations, it is not necessary to control the
connection coefficients $A_m(0)$ in the missing local smoothing norms.

To bound products of functions in $F_k(T)$ we often use a more relaxed
criterion. Precisely, since for $\e \in \mathbb S^1$ we have
\[
\|f\|_{L^{2,\infty}_{\e,W_{k+m_j}}}
\leq\|f\|_{L^{2,\infty}_{\e}}\lesssim
2^{k(d-1)/2}\|f\|_{L^2_xL^\infty_t}
\]
it follows that, in all dimensions $d\geq 2$,
\begin{equation}\label{useboundin}
  \|f\|_{F_k(T)}\lesssim \|f\|_{L^2_xL^\infty_t}+\|f\|_{L^{p_d}}.
\end{equation}
This criterion is often used to estimate bilinear expressions, by
exploiting the $L^{p_d}_xL^\infty_t$ norms in the spaces $F_k(T)$.

We also need to evolve $F_k(T)$ functions along the heat flow.  Since
the $F_k(T)$ norm is translation invariant it immediately follows that
if $h\in F_k(T)$ then
\begin{equation}\label{hebound}
  \|e^{s\Delta_x}h\|_{F_k(T)}\lesssim (1+s2^{2k})^{-20}\|h\|_{F_k(T)},
  \qquad s \geq 0.
\end{equation}

To prove useful bounds on the connection coefficients $A_m$,
$m=1,\ldots,d$, for $k\in\Z$ and $\omega\in[0,1/2]$ we define the
normed spaces $S_k^\omega(T)$ of functions in $L^2_k(T)$ for which
\begin{equation}\label{sk}
  \begin{split}
    \|f\|_{S_k^\omega(T)}=2^{k\omega}(\|f\|_{L^\infty_tL^{2_\omega}_x}+\|f\|_{L^{p_d}_tL^{p_{d,\omega}}_x}+2^{-kd/(d+2)}\|f\|_{L^{p_{d,\omega}}_xL^\infty_t})<\infty,
  \end{split}
\end{equation}
where the exponents $2_\omega$ and $p_{d,\omega}$ are such that
\begin{equation*}
  \frac{1}{2_\omega}-\frac{1}{2}=\frac{1}{p_{d,\omega}}-\frac{1}{p_d}
  =\frac{\omega}{d}.
\end{equation*}
The spaces $S_k^\omega(T)$ are at the same scale as the spaces
$F_k(T)$ and $F_k(T) \hookrightarrow S_k^0(T)$. By Sobolev embeddings
we have
\begin{equation}\label{skin}
  \|f\|_{S^{\omega'}_k(T)}\lesssim \|f\|_{S^{\omega}_k(T)}\quad\text{ if }\omega'\leq\omega.
\end{equation}
Thus the spaces $S_k^\omega(T)$ can be interpreted as refinements of
the Strichartz part of the spaces $F_k(T)$ (which corresponds to
$S_k^0(T)$).  It is important to be able to prove bounds on the
coefficients $A_m$, $m=1,\ldots,d$, in both spaces $F_k(T)$ and
$S_k^{1/2}(T)$. These bounds quantify an essential gain of smoothness
of the coefficients $A_m$ compared to the fields $\psi_m$. This is
proved in Lemma \ref{Aprop} and used in many estimates in sections
\ref{concoef} and \ref{PERTURB}.

\section{Outline of the proof}\label{outline}

This section contains an outline of the proofs of the main
theorems. We observe first that it suffices to construct the solution
$\phi$ on the time interval $(-2^{2\mathcal{K}},2^{2\mathcal{K}})$,
for any given $\mathcal{K}\in\mathbb{Z}_+$, and prove the bounds
\eqref{am2} and \eqref{am3} uniformly in $\mathcal{K}$. We therefore
fix\footnote{The value of $\mathcal{K}$ does not appear in any of the
  effective bounds; it is useful, however, to have
  $\mathcal{K}<\infty$ in some of the continuity arguments and to be
  able to define the sets $W_k(\mathcal{K})$ in Definition
  \ref{spacesd2} as finite sets. Weak bounds, such as
  \eqref{goodbounds2}, may depend implicitly on the value of
  $\mathcal{K}$.} once and for all $\mathcal{K}\gg 1\in\mathbb{Z}_+$
and assume $T\in(0,2^{2\mathcal{K}}]$.

We start with a solution $\phi\in C((-T,T):H^\infty_Q)$ of
\eqref{Sch1} on some time interval $(-T,T)$, where
$T\in(0,2^{2\K}]$. Our main goal is to prove a priori estimates on
\begin{equation*}
  \sup_{t\in(-T,T)}\| \phi(t)\|_{\dot{H}^{d/2}}\quad\text{ and }
  \quad\sup_{t\in(-T,T)}\|\phi(t)\|_{H^{\sigma}_Q},
\end{equation*} 
for $\sigma$ in a fixed interval $\sigma \in [d/2,\sigma_1]$.  We use
the notion of frequency envelopes.  We fix a small parameter $\delta$
(for instance $\delta = 1/(20 d)$ suffices).

\begin{definition}
  A positive sequence $\{b_k\}_{k \in \Z}$ is a frequency envelope if
  it is $l^2$ bounded
  \begin{equation}
    \sum_{k \in \Z} b_k^2 < \infty 
  \end{equation}
  and slowly varying,
  \begin{equation}
    b_k \leq b_j 2^{\delta |k-j|}, \qquad k,j \in \Z.
    \label{alsc}\end{equation}
  An $\epsilon$-frequency envelope $\{b_k\}_{k \in \Z}$ satisfies the
  additional relation
  \begin{equation}
    \sum_{k \in \Z} b_k^2 < \epsilon^2.
  \end{equation}
\end{definition}

Given an $l^2$ bounded nonnegative sequence $\{\al_k\}_{k\in\Z}$ we
often define the frequency envelope
\begin{equation*}
  \al'_k=\sup_{k'\in\Z}\ \al_{k'}2^{-\delta|k-k'|}.
\end{equation*}
Clearly, we have $\alpha_k\leq\alpha'_k$ and $\alpha'_k=\alpha_k$ if
$\{\alpha_k\}_{k\in\Z}$ is already a frequency envelope. In addition,
$\sum_{k\in\Z}[\alpha'_{k}]^2\lesssim \sum_{k\in\Z}\alpha_{k}^2$.

Given $\sigma_1 > d/2$ as in Theorem \ref{Main1}, $T > 0$ and
$\phi \in H^{\infty,\infty}_Q(T)$ we define the frequency
envelopes
\begin{equation}
  \gamma_k(\sigma) = \sup_{k' \in \Z} 2^{-\delta |k-k'|} 2^{\sigma k'} \|P_{k'}
  \phi\|_{L^\infty_t L^2_x}, \qquad \sigma \in [0,\sigma_1],\qquad\delta=1/(20d).
  \label{gammaksi}\end{equation}
We also set $\gamma_k = \gamma_k(d/2)$. Then we have
$\gamma_{k'}(\sigma)\leq 2^{\delta|k-k'|}\gamma_k(\sigma)$ for any
$k,k'\in\Z$, and
\begin{equation}
  \|P_k \phi\|_{L^\infty_t L^2_x} \leq 2^{-\sigma k} \gamma_k(\sigma),
  \qquad \sigma \in [0,\sigma_1].
\end{equation}

\begin{proposition}\label{TaoHeat} (Construction of the caloric gauge)
  Assume that $T\in(0,\infty)$ and $Q\in\mathbb{S}^2$. Let $ \phi\in
  {H}^{\infty,\infty}_Q(T)$ which satisfies the smallness condition
  \begin{equation}
    \sum_{k \in \Z} 2^{kd}    \|P_k \phi\|_{L^\infty_t L^2_x}^2 = \gamma^2 \ll1.
    \label{smallnorm}\end{equation}
  Then there is a unique smooth solution $\tphi\in
  C([0,\infty):{H}^{\infty,\infty}_Q(T))$ of the covariant heat
  equation
  \begin{equation}\label{heat3}
    \begin{cases}
      &\partial_s {\tphi}=\Delta_x {\tphi}+{\tphi}\cdot\sum_{m=1}^d
      |\partial_m{\tphi}|^2\quad\text{ on }[0,\infty)\times\R^d\times(-T,T);\\
      &{\tphi}(0,x,t)=\phi(x,t).
    \end{cases}
  \end{equation}
  In addition, there are smooth functions
  $v,w:[0,\infty)\times\R^d\times(-T,T)\to\mathbb{S}^2$ with the
  properties
  \begin{equation}\label{vwprop}
    {}^tv\cdot\tphi={}^tw\cdot\tphi={}^tv\cdot
    w={}^tw\cdot\partial_s v=0\text{ on
    }[0,\infty)\times\R^d\times(-T,T).
  \end{equation}
  For any $F\in\{\tphi,v,w\}$ we have the bounds
  \begin{equation}\label{goodbounds1}
    \|P_k F(s)\|_{L^\infty_tL^2_x}\lesssim \gamma_k(\sigma) 
    (1+s2^{2k})^{-20}2^{-\sigma k}, \qquad 
    \sigma \in [d/2,\sigma_1]
  \end{equation}
  with $\gamma_k(\sigma)$ defined by \eqref{gammaksi}, and, for any
  $\sigma,\rho\in \Z_+$,
  \begin{equation}\label{goodbounds2}
    \sup_{k\in\Z}\sup_{s\in[0,\infty)}(s+1)^{\sigma/2}2^{\sigma k}\|P_k \partial_t^\rho  F(s)\|_{L^\infty_tL^2_x} <\infty.
  \end{equation}
\end{proposition}
  
The key caloric gauge condition is the last identity in
\eqref{vwprop}, namely ${}^tw\cdot\partial_s v\equiv 0$, which leads
to the identity $A_0\equiv 0$. It is also important that the functions
$\tphi, v,w$ become trivial as $s\to\infty$, in the sense of
\eqref{goodbounds2}. Proposition \ref{TaoHeat} is due to Tao
\cite{Ta3}; we give a complete proof of Proposition \ref{TaoHeat} in
section \ref{HEATPROOF}.

Most of our analysis is done at the level of the fields $\psi_m$ and
the connection coefficients $A_m$. From \eqref{goodbounds1} and
\eqref{basicest5.1} we obtain
\begin{equation}\label{goodbounds3}
  \|P_k \psi_m(s) \|_{L^\infty_t L^2_x}+ 
  \|P_k A_m(s) \|_{L^\infty_t L^2_x}
  \lesssim \gamma_k(\sigma) 
  (1+s2^{2k})^{-20}2^{-(\sigma-1) k},
\end{equation}
for $m=1,\ldots,d$, $\sigma\in[d/2,\sigma_1]$, $s\in[0,\infty)$,
$k\in\Z$.  By \eqref{goodbounds2} we also have
\begin{equation}\label{goodbounds4}
  \sup_{k\in\Z}\sup_{s\in[0,\infty)}(s+1)^{\sigma/2}2^{k\sigma}2^{-k}\big[\|P_k(\partial_t^\rho\psi_m(s))\|_{L^\infty_tL^2_x}+\|P_k(\partial_t^\rho A_m(s))\|_{L^\infty_tL^2_x}\big]<\infty
\end{equation}
for $m=1,\ldots,d$, and
\begin{equation}\label{goodbounds5}
  \sup_{k\in\Z}\sup_{s\in[0,\infty)}(s+1)^{\sigma/2}2^{k\sigma}\big[\|P_k(\partial_t^\rho\psi_{d+1}(s))\|_{L^\infty_tL^2_x}+\|P_k(\partial_t^\rho A_{d+1}(s))\|_{L^\infty_tL^2_x}\big]<\infty.
\end{equation}

Given an initial data $\phi_0 \in H^\infty_Q$ for the Schr\"odinger
map equation, for $\sigma \in [(d-2)/2, \sigma_1-1]$ we introduce the
frequency envelopes
\begin{equation}
  c_k(\sigma) = \sup_{k' \in \Z}  \| P_{k'} \nabla \phi_0\|_{L^2_x}2^{\sigma k'} 2^{-\delta|k-k'|}.
  \label{ckenv}\end{equation}
Then we have the relations, for any $\sigma \in [(d-2)/2, \sigma_1-1]$
and $k\in\Z$,
\begin{equation}
  \|\nabla \phi_0\|_{ \dot H^{\sigma}}^2 \approx 
  \sum_{k\in \Z} c_k^2(\sigma)\qquad\text{ and }\qquad\| P_k \nabla \phi_0\|_{L^2} \leq c_k(\sigma) 2^{-\sigma k}.
  \label{hsbd}\end{equation}
Let $c_k=c_k(d/2-1)$. If $\|\phi_0\|_{\dot H^{d/2}} \leq \eps_0$ then
$\sum_{k \in \Z} c_k^2 \lesssim \eps_0^2$.

The bounds in Proposition~\ref{TaoHeat} in the energy space
$L^\infty_t L^2_x$ are far from sufficient for the study of the
Schr\"odinger equation. We need suitable bounds in the $F_k(T)$
spaces.  In the next proposition we fix some $\sigma_0 \in [d/2-1,
\sigma_1-1]$ and use two frequency envelopes $b_k$ and
$b_k(\sigma_0)$.  The envelope $b_k = b_k((d-2)/2)$ is used to measure
critical regularity and carries a smallness condition. The envelope
$b_k(\sigma_0)$ is always used to measure noncritical regularity. To
these two envelopes we associate the sequences
\[
b_{>k} = \big(\sum_{j \geq k} b_j^2\big)^{1/2}, \qquad
b_{>k}(\sigma_0) = \big(\sum_{j \geq k} b_j b_j(\sigma_0)
2^{(k-j)(\sigma_0-(d-2)/2)}\big)^{1/2}.
\]

\begin{proposition} \label{TaoHeat2} (Heat flow bootstrap estimates)
  Assume that $T\in(0,\infty)$ and $Q\in\mathbb{S}^2$. Given $ \phi\in
  {H}^{\infty,\infty}_Q(T)$ satisfying \eqref{smallnorm} we consider
  $\tphi,v,w$ as in Proposition~\ref{TaoHeat}, and $\psi_m$ and $A_m$
  the associated fields and connection coefficients. Let
  $b_k=b_k(d/2-1)$ be an $\eps$-frequency envelope with small $\eps$,
  and $b_k(\sigma_0)$ be another frequency envelope.

  (a) Suppose that the functions $\{\psi_m\}_{m=1,d}$ satisfy
  \begin{equation}
    \| P_k \psi_m (0) \|_{F_k(T)} \leq  b_k(\sigma)  2^{-\sigma k},
    \qquad 
    \sigma \in \{ (d-2)/2,\sigma_0\}
    \label{psizero}\end{equation}
  as well as the bootstrap condition
  \begin{equation}
    \| P_k \psi_m (s) \|_{F_k(T)} \leq
    \eps^{-1/2} b_k 2^{-k(d-2)/2}(1+s2^{2k})^{-4}.
    \label{psibound}\end{equation}
  Then for $\sigma \in \{ (d-2)/2,\sigma_0\}$ we have
  \begin{equation}
    \| P_k \psi_m (s) \|_{F_k(T)} \lesssim  b_k(\sigma)
    2^{-\sigma k}(1+s2^{2k})^{-4}, \qquad \sigma \in \{ (d-2)/2,\sigma_0\}.
    \label{psifk}\end{equation}
  Also, for $ l,m=1,\cdots,d$ and $\sigma \in \{ (d-2)/2,\sigma_0\}$,
  we have the $F_k(T)$ bounds
  \begin{equation}
    \| P_k ( A_m (s) \psi_l (s)) \|_{F_k(T)} \lesssim b_k(\sigma)
    2^{-(\sigma-1) k}(2^{2k} s)^{-\frac38}
    (1+s2^{2k})^{-2},
    \label{apsifk}\end{equation}
  as well as the $L^{p_d}$ estimate at $s=0$
  \begin{equation}
    \| P_k A_m (0) \|_{L^{p_d}} \lesssim  b_k(\sigma)
    2^{-\sigma k}.
    \label{pkam}\end{equation}

  (b) Assume in addition that
  \begin{equation}
    \| P_k \psi_{d+1}(0)\|_{L^{p_d}} \lesssim  b_k(\sigma) 
    2^{-(\sigma-1) k},  \qquad 
    \sigma \in \{ (d-2)/2,\sigma_0\}.
    \label{psidpu}\end{equation}
  Then for $\sigma \in \{ (d-2)/2,\sigma_0\}$ we have
  \begin{equation}
    \| P_k \psi_{d+1}(s)\|_{L^{p_d}} \lesssim b_k(\sigma)  
    2^{-(\sigma-1) k} (1+2^{2k} s)^{-2},
  \end{equation}
  and the connection coefficient $A_{d+1}$ satisfies the $L^{2}$
  estimate at $s=0$
  \begin{equation}
    \| P_k A_{d+1} (0) \|_{L^{2}} \lesssim  \eps  b_k(\sigma)
    2^{-\sigma k}, \qquad d\geq 3 
    \label{aadunu}\end{equation}
  respectively
  \begin{equation}
    \|A_{d+1}(0)\|_{L^2} \lesssim \eps^2, \qquad 
    \| P_{k} A_{d+1} (0) \|_{L^{2}} \lesssim  
    b_{>k}^2(\sigma)2^{-\sigma k},
    \qquad d = 2. 
    \label{aadunub}\end{equation}
  \label{heatfk}
\end{proposition}
Proposition~\ref{heatfk} is proved in Section~\ref{concoef}. Here we
note the following improvement:

\begin{corollary} 
  The result in Proposition~\ref{heatfk} remains valid as well if the
  bootstrap assumption \eqref{psibound} is dropped.
  \label{heatcor}\end{corollary}

\begin{proof}[Proof of Corollary \ref{heatcor}]
  Define the function
  \[
  \Psi(T') = \sup_{k \in \Z} \sup_{s \geq 0} b_k^{-1}
  2^{k(d-2)/2}(1+s2^{2k})^{4} \| P_k \psi_m (s) \|_{F_k(T')}, \qquad
  0<T'\leq T.
  \]
  It follows easily from \eqref{goodbounds4} with $\rho=1$ that $\Psi$
  is an increasing, continuous function of $T'$. We will prove that
  \begin{equation}\label{alexnew20}
    \lim_{T' \to 0} \Psi(T') \lesssim 1.
  \end{equation}
  Since $\Psi:(0,T]\to[0,\infty)$ is a continuous increasing function
  and $\Psi(T')\leq \eps^{-1/2}$ implies $\Psi(T')\lesssim 1$ (see
  \eqref{psifk}), the corollary follows easily from
  \eqref{alexnew20}. To prove \eqref{alexnew20}, let
  $\psi_{m,0}(s,x)=\psi_m(s,x,0)$. Using \eqref{useboundin} and
  \eqref{goodbounds4}, it suffices to prove that
  \begin{equation}\label{alexnew21}
    \sup_{k \in \Z} \sup_{s \geq 0} b_k^{-1}
    2^{k(d-2)/2}(1+s2^{2k})^{4}
    \| P_k \psi_{m,0} (s) \|_{L^2_x}\lesssim 1.
  \end{equation}
  It follows from \eqref{psizero} that
  \begin{equation}\label{alexnew22}
    2^{k(d-2)/2}\| P_k \psi_{m,0} (0) \|_{L^2_x}\leq b_k.
  \end{equation}
  We need to extend this bound to $s>0$ with suitable parabolic decay.

  Recall the coefficients $c_k=c_k(d/2-1)$ defined in
  \eqref{ckenv}. We apply Proposition \ref{TaoHeat} on sufficiently
  short time intervals $(-T,T)$; using \eqref{goodbounds1}
  \begin{equation}\label{alexnew24}
    2^{dk/2}(1+s2^{2k})^{20}\big[\|P_k(v_0(s))\|_{L^2_x}+\|P_k(w_0(s))\|_{L^2_x}\big]\lesssim c_k,
  \end{equation}
  where $v_0(s,x)=v(s,x,0)$ and $w_0(s,x)=w(s,x,0)$. We use this bound
  at $s=0$, the identity \eqref{basis}, and the bounds
  \eqref{alexnew22} and \eqref{basicest5.1}; it follows that
  \begin{equation*}
    2^{k(d-2)/2}\|P_k\nabla\phi_0\|_{L^2_x}\lesssim b_k.
  \end{equation*}
  Since $\{b_k\}_{k\in\Z}$ is a frequency envelope, it follows that
  $c_k\lesssim b_k$, see definition \eqref{ckenv}. It follows from
  \eqref{goodbounds3} that
  \begin{equation*}
    \|P_k\psi_{m,0}(s)\|_{L^2_x}\lesssim c_k(1+s2^{2k})^{-20}2^{-k(d-2)/2}.
  \end{equation*}
  The bound \eqref{alexnew21} follows since $c_k\lesssim b_k$.
\end{proof}

Next we turn our attention to the Schr\"odinger equations
\eqref{schcov2}. Our main Schr\"odinger bootstrap result is the
following.

\begin{proposition} (Schr\"{o}dinger bootstrap estimates) Assume that
  $T\in(0,2^{2\K}]$ and $Q\in\mathbb{S}^2$.  Let $\{c_k\}_{k \in \Z}$
  be an $\eps_0$-frequency envelope with $\eps_0 \in (0,\eps_0(d)]$,
  and $\{c_k(\sigma_0)\}_{k \in \Z}$ another frequency envelope. Let
  $\phi \in {H}^{\infty,\infty}_Q(T)$ be a solution of the
  Schr\"odinger map equation \eqref{Sch1} whose initial data satisfies
  \begin{equation}
    \| P_k \nabla \phi_0 \|_{L^2_x} \leq  c_k(\sigma) 2^{\sigma k},
    \qquad \sigma \in \{(d-2)/2,\sigma_0\}. 
    \label{schid}\end{equation}
  Assume that $\phi$ satisfies the bootstrap condition
  \begin{equation}
    % \sum_{j \geq k} 2^{(d-2)k}
    \| P_k \nabla \phi \|_{L^\infty_t L^2_x} \leq  \eps_0^{-1/2} 2^{-(d-2)k/2}c_{k}
    \label{bootso}\end{equation}
  and let $(\phi,v,w)$ be the caloric extension of $\phi$ given by
  Proposition~\ref{TaoHeat}, with the corresponding fields $\psi_m$,
  $A_m$.  Suppose also that at the initial parabolic time $s=0$ the
  functions $\{\psi_m\}_{m=1,d}$ satisfy the additional bootstrap
  condition
  \begin{equation}
    \| P_k \psi_m (0) \|_{G_k(T)} \leq \eps_0^{-1/2}  2^{-(d-2)k/2} c_{k}. 
    \label{bootsch}\end{equation}
  Then we have
  \begin{equation}
    \| P_k \psi_m (0) \|_{G_k(T)} \lesssim  c_{k}(\sigma) 2^{-\sigma k}, 
    \qquad \sigma \in \{(d-2)/2,\sigma_0\}.
    \label{bootschout}\end{equation}
  \label{schgk}
\end{proposition}

The above proposition is proved in Section~\ref{PERTURB} by applying
the linear result in Proposition~\ref{linearmainrep} to the equation
\eqref{schcov2}. The right hand side in \eqref{schcov2} is estimated
in the $N_k(T)$ spaces using the bounds in Proposition ~\ref{heatfk}
for the differentiated fields $\psi_m$ and the connection coefficients
$A_m$.  In what follows we show that Proposition~\ref{schgk} implies
Theorem~\ref{newth0} and the bound \eqref{am3}.

\begin{proof}[\bf Proof of Theorem~\ref{newth0} and the bound
  \eqref{am3}.]

  Consider an initial data $\phi_0\in H^\infty_Q$ for the
  Schr\"odinger map equation \eqref{Sch1} which satisfies $\|
  \phi_0\|_{\dot H^{d/2}} \ll 1$. Our starting point is the local
  existence and uniqueness of smooth solutions of the Schr\"{o}dinger
  map equation (see, for example, \cite{Ga}): if $\phi_0\in
  H^\infty_Q$ then there is $T= T(\|\phi_0\|_{H^{2d+20}_Q})>0$ and a
  unique solution $\phi\in C((- T,T):H^\infty_Q)$ of the initial value
  problem \eqref{Sch1}.

 In order to prove  Theorem  \ref{newth0} we 
take $\sigma_1=2d+20$ in what follows. For  Theorem \ref{Main1}, on the other
hand, we allow $\sigma_1$ to be arbitrarily large.   It suffices
  to prove that the solution $\phi\in C((- T,T):H^\infty_Q)$ of the
  initial value problem \eqref{Sch1} satisfies the bounds \eqref{am2}
  and \eqref{am3} with constants which are independent of $T$.  In
  preparation for the proof of the well-posedness part of Theorem
  \ref{Main1}, we prove stronger bounds for the $\phi$ in terms of the
  frequency envelopes of $\phi_0$.

  Define the frequency envelopes $c_k(\sigma)$ as in
  \eqref{ckenv}. Our goal for the rest of this proof is to use
  Proposition~\ref{schgk} in order to prove that
  \begin{equation}
    \sup_{t\in(-T,T)}\| P_k \nabla \phi (0,.,t) \|_{L^2_x} \lesssim  
    c_{k} (\sigma)2^{-\sigma k} , \qquad \sigma \in [(d-2)/2,\sigma_1-1],
    \label{phiboot}\end{equation}
  with implicit constants which depend only on $d$ and $\sigma_1$.  In
  view of \eqref{hsbd}, and \eqref{conserve}, this suffices to
  establish \eqref{am2} and \eqref{am3}. Then \eqref{allsigma} follows
from \eqref{am3} for $\sigma$ up to $2d+20$ and from the result in \cite{Ga}
for larger $\sigma$.

  For $T'\in(0,T]$ let
  \[
  \Psi(T') = \sup_{k \in \Z} c_{k}^{-1} 2^{(d-2) k/2} ( \|P_{k} \nabla
  \phi\|_{L^\infty_{T'}L^2_x}+ \| P_k \psi_m (0) \|_{G_k(T')}), \quad
  0<T'\leq T.
  \]
  The function $\Psi:(0,T]\to[0,\infty)$ is well-defined, increasing
  and continuous, using \eqref{goodbounds4}, the fact that $\phi\in
  C((-T,T):H^\infty_Q)$, and the fact that $c_k$ is a frequency
  envelope. We show now that
  \begin{equation}\label{alexnew30}
    \begin{cases}
      &\text{  if }\Psi(T')\leq \eps_0^{-1/2}\text{ then }\Psi(T')\lesssim 1;\\
      &\lim_{T'\to 0}\Psi(T')\lesssim 1.
    \end{cases}
  \end{equation}
  The limit in the second line of \eqref{alexnew30} follows from the
  definition of the coefficients $c_k$, see \eqref{hsbd}, and
  \eqref{goodbounds3} (we apply Proposition \ref{TaoHeat} on
  sufficiently short intervals $(-T',T')$). Also, using Proposition
  \ref{schgk}, if $\Psi(T')\leq \eps_0^{-1/2}$ then
  \begin{equation*}
    \sup_{k \in \Z} c_{k}^{-1}
    2^{(d-2) k/2}\| P_k \psi_m (0) \|_{G_k(T')}\lesssim 1.
  \end{equation*}
  Assuming $\Psi(T')\leq \eps_0^{-1/2}$, we apply
  Proposition~\ref{TaoHeat} to conclude that
  \begin{equation}
    \| P_k \nabla v(0) \|_{L^\infty_{T'} L^2_x}+ \| P_k \nabla w(0) \|_{L^\infty_{T'} L^2_x} 
    \lesssim \eps_0^{-1/2} 2^{-(d-2)k/2} c_k. 
    \label{pknpb}\end{equation}
  On the other hand from \eqref{bootschout}
  \begin{equation}
    \| P_k  \psi_m(0) \|_{L^\infty_{T'} L^2_x} \lesssim  2^{-\sigma k}
    c_k(\sigma), 
    \qquad    \sigma \in [(d-2)/2,\sigma_1-1].
    \label{pknpa}\end{equation}
  We use the relation, see \eqref{basis}, $\partial_m \phi =
  v\Re(\psi_m)+w\Im(\psi_m)$.  From \eqref{pknpa} with $\sigma =
  (d-2)/2$ and \eqref{pknpb}, using \eqref{basicest5.1}, we obtain
  \begin{equation*}
    \| P_k \nabla \phi \|_{L^\infty_{T'} L^2} \lesssim  2^{-(d-2)k/2} c_k.
    \label{pknp}\end{equation*}
  The implication in the first line of \eqref{alexnew30} follows.

  It follows from \eqref{alexnew30} and the continuity of $\Psi$ that
  $\Psi(T)\lesssim 1$. Thus
  \begin{equation}\label{alexnew31}
    \|P_{k} \nabla \phi\|_{L^\infty_{T}L^2_x}+ \| P_k \psi_m (0) \|_{G_k(T)}\lesssim c_k2^{-k(d-2)/2}.
  \end{equation}
  This suffices to prove the bound \eqref{phiboot} for
  $\sigma=(d-2)/2$. To establish \eqref{phiboot} for a different
  $\sigma$ we denote by $B>0$ the best constant so that
  \begin{equation}
    \| P_k \nabla \phi \|_{L^\infty_{T} L^2_x} \leq B c_{k}(\sigma) 2^{-\sigma k}.
    \label{scor3}\end{equation}
  Such a constant exists because of the smoothness of $\phi$ and the
  fact that $c_k(\sigma)$ is a frequency envelope. Using
  \eqref{alexnew31} and Proposition~\ref{TaoHeat} we have
  \begin{equation}
    \| P_k v(0) \|_{L^\infty_{T} L^2_x}+\| P_k w(0) \|_{L^\infty_{T} L^2_x} \lesssim
    B c_k(\sigma) 2^{-(\sigma+1) k}.
    \label{scor4}\end{equation}
  From \eqref{pknpa} and \eqref{scor4}, by \eqref{basicest5.1}, we
  obtain
  \begin{equation}
    \| P_k \nabla \phi \|_{L^\infty L^2} \lesssim (1+\eps B)
    2^{-\sigma k} c_k(\sigma).
  \end{equation} 
  By the optimality of $B$ in \eqref{scor3} we conclude that $B
  \lesssim 1+\eps B$, which yields $B \lesssim 1$. Thus
  \eqref{phiboot} is proved.
\end{proof}

To define rough solutions and study the dependence of solutions on the
initial data we consider the linearized equation \eqref{schlin} and
prove that it is well-posed in $\dot H^{(d-2)/2}$.

\begin{proposition} 
  Let $\phi_0 \in H^\infty_Q$ be an initial data for the Schr\"odinger
  map equation which satisfies the smallness condition
  $\|\phi_0\|_{\dot H^{d/2}} \ll 1$. Let $\phi$ be the corresponding global
  solution to \eqref{Sch1}, $(\phi,u,v)$ its caloric extension and
  $\psi_m$, $A_m$ as before.  Then for each initial data $\psi_{lin}
  (0) \in H^\infty$ there exists an unique solution $\psi_{lin} \in
  C(\R,H^\infty)$ for \eqref{schlin}, which satisfies the bounds
  \begin{equation}
    \sum_{k} 2^{(d-2) k} \| P_k \psi_{lin} \|_{G_k(T)}^2 \lesssim 
    \|\psi_{lin}(0)\|^2_{\dot H^{\frac{d-2}2}}
    \label{psl}\end{equation}
\end{proposition}

The proof of this result is identical to the proof of
Proposition~\ref{heatfk}. As a consequence of this we obtain the
Lipschitz dependence of solutions to \eqref{Sch1} in terms of the
initial data in a weaker topology:

\begin{proposition}
  Consider two initial data $\phi_0^0$ and $\phi_0^{1}$ in
  $H^\infty_Q$ which satisfy the smallness condition
  $\|\phi_0^h\|_{\dot H^{\frac{d}2}} \ll 1$, $h=0,1$, and let
  $\phi^0$, $\phi^1$ be the corresponding global solutions for
  \eqref{schlin}. Then
  \begin{equation}
    \sum_k 2^{(d-2)k} \| P_k( \phi^0 -\phi^1) \|_{L^\infty L^2}^2 \lesssim
    \| \phi^0_0 -\phi^1_0\|_{\dot H^{\frac{d-2}{2}}}^2
    \label{l2diff}\end{equation}
  \label{pdiff}
\end{proposition}

\begin{proof}
  By \eqref{psl} we have the global in time bound
  \[
  \sum_{k} 2^{(d-2) k} \| P_k \psi_{lin} \|_{L^\infty L^2}^2 \lesssim
  \|\psi_{lin}(0)\|^2_{\dot H^{\frac{d-2}2}}
  \]
  As in the proof of Theorem~\ref{newth0}, this bound easily transfers
  to the functions $\philin$, and we obtain
  \begin{equation}
    \sum_k 2^{(d-2)k} \|P_k \philin\|_{L^\infty L^2}^2 \lesssim
    \|\philin(0)\|_{\dot H^{\frac{d-2}{2}}}^2
    \label{infdiff}\end{equation}
  for all solutions $\philin \in C(\R, H^\infty)$ to \eqref{eqlin}.

  Any two initial data $\phi_0^0$ and $\phi_0^{1}$ in $H^\infty_Q$
  which satisfy the smallness condition $\|\phi_0^h\|_{\dot
    H^{\frac{d}2}} \ll 1$, $h=0,1$ can be joined with a one parameter
  family of initial data as follows:

\begin{lemma} [Proposition 3.13, \cite{Tat2}]
  Consider two functions $\phi_0^0, \phi_0^{1} \in H^\infty_Q$ so that
  $\|\phi_0^h\|_{\dot H^{\frac{d}2}} \ll 1$, $h=0,1$. Then there exists a
  smooth one parameter family of initial data $\{ \phi^h_0\}_{h \in
    [0,1]}\in C^\infty([0,1];H^\infty)$, taking values in $
  H^\infty_Q$, which joins them, so that the smallness condition
  $\|\phi_0^h\|_{\dot H^{\frac{d}2}} \ll 1$, $h\in [0,1]$ is satisfied
  uniformly and
  \begin{equation}
    \int_{0}^{1} \| \partial_h \phi^h_0\|_{\dot H^{\frac{d-2}{2}}} \approx
    \| \phi^0_0 -\phi^1_0\|_{\dot H^{\frac{d-2}{2}}}
    \label{twmref}\end{equation}
\end{lemma}

The corresponding family of global solutions $\phi^h$ is smooth with
respect to $h$, and the functions $\partial_{h}\phi^h $ solve the
linearized equation \eqref{schlin} with $\phi$ replaced by $\phi_h$.
Applying \eqref{infdiff} to $\partial_{h}\phi^h $ we obtain
\[
\sum_{k} 2^{(d-2) k} \| P_k \partial_h \phi^h \|_{L^\infty L^2}^2
\lesssim \|\partial_h \phi^h(0)\|^2_{\dot H^{\frac{d-2}2}}
\]
The estimate \eqref{l2diff} is obtained by integrating with respect to
$h$ due to \eqref{twmref}.
\end{proof}

The above proposition allows us to conclude the proof of
Theorem~\ref{Main1}.

\begin{proof}[\bf Proof of Theorem~\ref{Main1}] The bound \eqref{am3}
  was proved earlier. To show that the map $S_Q$ admits an unique
  continuous extension
  \[
  S_Q: \dot H^{\frac{d}2} \cap \dot H^{\frac{d-2}2}_Q \to C(\R; \dot
  H^{\frac{d}2} \cap \dot H^{\frac{d-2}2}_Q)
  \]
  it suffices to consider a sequence of smooth initial data $\phi^n_0
  \in H^\infty_Q$ which satisfy uniformly the smallness condition $
  \|\phi^n_0\|_{\dot H^{\frac{d}2} } \ll 1 $ and so that $\phi^n_0 \to
  \phi_0$ in $\dot H^{\frac{d}2} \cap \dot H^{\frac{d-2}2}_Q$, and
  show that the corresponding sequence of global solutions is Cauchy
  in the space in $C(\R; \dot H^{\frac{d}2} \cap \dot
  H^{\frac{d-2}2}_Q)$.  By Proposition~\ref{pdiff} it follows the
  $\phi^n$ is Cauchy in $C(\R; \dot H^{\frac{d-2}2}_Q)$,
  \begin{equation}
    \lim_{n,m \to \infty} \| \phi^n -\phi^m\|_{ C(\R; \dot
      H^{\frac{d-2}2})} = 0
    \label{conta}\end{equation}
  Consider frequency envelopes $\{ c^n_k\}$ associated as in
  \eqref{ckenv} to $\phi^n_0$.  Since $\phi^n_0$ is convergent in
  $\dot H^{\frac{d}2}$ we can choose the corresponding envelopes $\{
  c^n_k\}$ to converge in $l^2$. Then we have the uniform summability
  property
  \begin{equation}
    \lim_{k_0 \to \infty} \sup_{n} \sum_{k > k_0} (c_k^n )^2 = 0
    \label{contb}\end{equation}
  Now we use \eqref{phiboot} to estimate
  \[
  \begin{split}
    \| \phi^n -\phi^m\|_{ C(\R; \dot H^{\frac{d}2})} \leq & \| P_{\leq
      k_0} (\phi^n -\phi^m)\|_{ C(\R; \dot H^{\frac{d}2})} + \|
    P_{>k_0} \phi^n \| _{ C(\R; \dot H^{\frac{d}2})}+\| P_{>k_0}
    \phi^m\|_{ C(\R; \dot H^{\frac{d}2})}
    \\
    \lesssim &\ 2^{k_0} \| P_{\leq k_0} (\phi^n -\phi^m)\|_{ C(\R;
      \dot H^{\frac{d-2}2})} + \sum_{k > k_0} (c_k^n )^2 + (c_k^n )^2
  \end{split}
  \]
  Hence using \eqref{conta} we have
  \[
  \limsup_{n,m \to \infty} \| \phi^n -\phi^m\|_{ C(\R; \dot
    H^{\frac{d}2})} \lesssim \sup_n \sum_{k > k_0} (c_k^n )^2
  \]
  Letting $k_0 \to \infty$, by \eqref{contb} we obtain
  \[
  \limsup_{n,m \to \infty} \| \phi^n -\phi^m\|_{ C(\R; \dot
    H^{\frac{d}2})} = 0
  \]
  and the argument is concluded.

  The continuity of the solution operator $S_Q$ in higher Sobolev
  spaces
  \[
  S_Q : \dot H^{\sigma} \cap \dot H^{\frac{d-2}2}_Q \to C(\R; \dot
  H^{\sigma} \cap \dot H^{\frac{d-2}2}_Q), \qquad \frac{d}2 < \sigma
  \leq \sigma_1
  \]
  is obtained in the same manner. The proof of Theorem~\ref{Main1} is
  concluded.
\end{proof}

\section{The heat flow: proof of Proposition~\ref{heatfk}.}\label{concoef}

For $k \in \Z$ we denote 
\begin{equation}
a(k) = \sup_{s \in [0,\infty)} (1+s2^{2k})^{4} \sum_{m=1}^d \| P_k \psi_m(s)\|_{F_k(T)}.
\end{equation}
For $\sigma \in \{(d-2)/2,\sigma_0\}$ we introduce the frequency envelopes
\begin{equation}
a_k(\sigma) = \sup_{j \in \Z} 2^{\sigma j} 2^{-\delta |k-j|} a(j).
\label{aksdef}\end{equation}
These are finite and belong to $l^2$ due to \eqref{goodbounds4} and \eqref{useboundin}. For \eqref{psifk} we need to show that
$a_k(\sigma) \lesssim b_k(\sigma)$. On the other hand from the
bootstrap assumption \eqref{psibound} we know that $a_k=a_k(d/2-1) \leq
\eps^{-1/2} b_k$.  In particular this implies that
\begin{equation}
\sum_{k \in \Z} a_k^2 \leq \eps.
\label{akeps}\end{equation}

We prove first a bilinear estimate.

\begin{lemma}\label{building} Assume that $T\in(0,2^{2\mathcal{K}}]$,
  $f,g\in {H}^{\infty,\infty}(T)$, $P_k f\in F_k(T)\cap
  S_k^\omega(T)$, $P_k g\in F_k(T)$ for some $\omega\in[0,1/2]$ and
  any $k\in\Z$, and
  \begin{equation*}
    \al_k=\sum_{|j-k|\leq 20}\|P_{j}f\|_{F_{j}(T)\cap S_{j}^\omega(T)},\quad\be_k=\sum_{|j-k|\leq 20}\|P_{j}g\|_{F_{j}(T)}.
  \end{equation*}
  Then, for any $k\in\Z$
  \begin{equation}\label{basbo3.3}
      \|P_k(fg)\|_{F_k(T)\cap S^{1/2}_k(T)} \lesssim \sum_{j\leq
        k}2^{\frac{jd}2}(\be_k\al_{j}+\al_k\be_{j})+2^{\frac{kd}2}\sum_{j\geq
        k}2^{(j-k)(\frac{2d}{d+2}-\omega)}\alpha_{j}\be_{j}.
  \end{equation}
\end{lemma}

\begin{proof}[Proof of Lemma \ref{building}] We observe first that if
  $k,j\in\Z$ with $|k-j|\leq 20$ and
  $\omega'\in[0,1/2]$  then
  \begin{equation}\label{basbo1}
    \|P_k(uv)\|_{F_k(T)}\lesssim\|u\|_{F_{j}(T)}\|v\|_{L^\infty_{x,t}}
  \end{equation}
  and
  \begin{equation}\label{basbo1.1}
    \|P_k(uv)\|_{S^{\omega'}_k(T)}\lesssim\|u\|_{F_{j}(T)}2^{k\omega'}
\|v\|_{L^{d/\omega'}_xL^\infty_t}.
  \end{equation}
  both of which follow directly the definitions.  For the second factor on the right in both
  \eqref{basbo1} and \eqref{basbo1.1} we observe that for $v$ 
which is localized at frequency $2^k$ we have by Sobolev embeddings
\begin{equation}
\| v\|_{L^\infty_{x,t}}+2^{k\omega'}\| 
v\|_{L^{d/\omega'}_xL^\infty_t}\leq
  C2^{dk/2}\| v\|_{F_k(T)}
\label{basbo1.2}\end{equation}

  We show now that if $|k_1-k_2|\leq 8$, and 
 $f_{k_1}$, $g_{k_2}$ are localized at frequency $2^{k_1}$,
 respectively $2^{k_2}$ then
  \begin{equation}\label{basbo2}
    \|P_{k}(f_{k_1}g_{k_2})\|_{F_{k}(T)\cap S_k^{1/2}(T)}\leq C2^{kd/2}2^{(k_2-k)(2d/(d+2)-\omega)}\|f_{k_1}\|_{S^\omega_{k_1}(T)}\|g_{k_2}\|_{S^0_{k_2}(T)}.
  \end{equation}
  To prove this we  use \eqref{useboundin} and Sobolev embeddings:
  \begin{equation*}
    \begin{split}
      \|P_{k}(f_{k_1}&g_{k_2})\|_{F_{k}(T)\cap S_k^{1/2}(T)}
      \lesssim  \ \|P_{k}(f_{k_1}g_{k_2})\|_{L^2_xL^\infty_t}+
2^{k/2}\|P_{k}(f_{k_1}g_{k_2})\|_{L^\infty_tL^{2_{1/2}}_x \cap L^{p_d}_tL^{p_{d,1/2}}_x}\\
      \lesssim &\ 2^{kd(1/p_d+1/p_{d,\omega}-1/2)}\|f_{k_1}\|_{L^{p_{d,\omega}}_xL^\infty_t}\|g_{k_2}\|_{L^{p_d}_xL^\infty_t}
 + 2^{kd/2}\|f_{k_1}\|_{L^\infty_tL^2_x}\|g_{k_2}\|_{L^\infty_tL^{2}_x \cap L^{p_d}}\\
      \lesssim & \
      2^{kd/2}(2^{|k_2-k|(2d/(d+2)-\omega)}\|f_{k_1}\|_{S^\omega_{k_1}(T)}\|g_{k_2}\|_{S^0_{k_2}(T)}+\|f_{k_1}\|_{S_{k_1}^0(T)}\|g_{k_2}\|_{S^0_{k_2}(T)}),
    \end{split}
  \end{equation*}

  To prove the estimate \eqref{basbo3.3} we use a bilinear
  Littlewood-Paley decomposition
\begin{equation}
P_{k}(fg) = \sum_{k_1\leq k-4}^{|k_2-k|\leq 4}P_k(P_{k_1}fP_{k_2}g)
+ \sum_{k_2\leq k-4}^{|k_1-k|\leq 4} P_k(P_{k_1}f P_{k_2}g)
+ \sum_{k_1,k_2\geq k-4}^{|k_1-k_2|\leq 8} P_k(P_{k_1}f P_{k_2}g)
\label{bilp}\end{equation}
and bound each of the terms on the right in $F_{k}(T)\cap S_k^{1/2}(T)$.
For the first two we use \eqref{basbo1}, \eqref{basbo1.1} and \eqref{basbo1.2}.
For the third we use   \eqref{basbo2} instead. The bound \eqref{basbo3.3}
follows.

% \begin{equation*}
%     \begin{split}
%       &\|P_{k}(fg)\|_{F_{k}(T)\cap S_k^{1/2}(T)}\lesssim
% \sum_{k_1\leq k-4}^{|k_2-k|\leq 4}\|P_k(P_{k_1}fP_{k_2}g)\|_{F_{k}(T)\cap S_k^{1/2}(T)}\\
%       &+\sum_{k_2\leq k-4}^{|k_1-k|\leq 4}\|P_k(P_{k_1}f\cdot P_{k_2}g)\|_{F_{k}(T)\cap S_k^{1/2}(T)}
%       +\sum_{k_1,k_2\geq k-4}^{|k_1-k_2|\leq 8}\|P_k(P_{k_1}f\cdot P_{k_2}g)\|_{F_{k}(T)\cap S_k^{1/2}(T)}\\
%       &\lesssim \sum_{k_1\leq
%         k}2^{dk_1/2}\alpha_{k_1}\beta_k+\sum_{k_2\leq
%         k}2^{dk_2/2}\beta_{k_2}\alpha_k+2^{dk/2}\sum_{j\geq
%         k}2^{|j-k|(2d/(d+2)-\omega)}\alpha_{j}\beta_{j}
%     \end{split}
%   \end{equation*}
%   as desired.
\end{proof}

We prove now our main estimates on the connection coefficients $A_m$,
$m=1,\ldots,d$.

\begin{lemma} \label{Aprop} (Bounds on $A_m(s)$) For
  any $k\in\Z$, $s\in[0,\infty)$, and $m=1,\ldots,d$
  \begin{equation}\label{claimA3}
    \begin{split}
      \|P_k(A_m(s))\|_{F_k(T)\cap S_k^{1/2}(T)}\lesssim
      2^{-\sigma k} (1+s2^{2k})^{-4}b_{k,s}(\sigma),
    \end{split}
  \end{equation}
  where, if $s\in[2^{2k_0-1},2^{2k_0+1})$, $k_0\in\Z$, then
  \begin{equation}\label{newb}
    b_{k,s}(\sigma)=
    \begin{cases}
      &2^{k+k_0}a_{-k_0}a_k(\sigma)\text{ if }k+k_0\geq 0;\\
      &\sum_{j=k}^{-k_0}a_{j} a_{j}(\sigma)\text{ if }k+k_0\leq 0.
    \end{cases}
  \end{equation}
\end{lemma}

\begin{proof}[Proof of Lemma \ref{Aprop}] We use the identity
  \eqref{Aform}
  \begin{equation}\label{claim12}
    \begin{split}
      A_m(s)&=-\sum_{l=1}^d\int_s^{\infty}\Im\big(\overline{\psi_m}(\partial_l\psi_l+iA_l\psi_l)\big)(r)\,dr.
    \end{split}
  \end{equation}
  To prove \eqref{claimA3}, let $B_1$ denote the smallest number in
  $[1,\infty)$ with the property that for any $s\in[0,\infty)$,
  $k\in\Z$, $m=1,\ldots,d$, and $\sigma\in\{(d-2)/2,\sigma_1-1\}$
  \begin{equation}\label{int21}
    \|P_k(A_m(s))\|_{F_k(T)}\leq B_12^{-\sigma k}(1+s2^{2k})^{-4}b_{k,s}(\sigma).
  \end{equation}

  We observe first that for any
  $f,g\in\{\psi_m,\overline{\psi_m}:m=1,\ldots d\}$, 
  $r\in[2^{2j-2},2^{2j+2}]$, $j\in\Z$, $l=1,\ldots,d$, and
  $\sigma\in\{(d-2)/2,\sigma_1-1\}$ we have the bounds
  \begin{equation}\label{usebound99a}
    \|P_k(f(r) g(r))\|_{F_k(T)\cap S_k^{1/2}(T)}\lesssim 
2^{-\sigma k}(1+2^{2k+2j})^{-4} 2^{-j}a_{-j} a_{\max(k,-j)}(\sigma)
  \end{equation}
  \begin{equation}\label{usebound99}
 \!  \|P_k(f(r) \partial_lg(r))\|_{F_k(T)\cap S_k^{1/2}(T)}\lesssim
2^{-\sigma k}(1+2^{2k+2j})^{-4}2^{-j}a_{-j}(2^{k}a_k(\sigma)+2^{-j}a_{-j}(\sigma))\!\!
  \end{equation}
  To prove this we consider the cases $k+j\leq 0$ and
  $k+j\geq 0$, use the bounds \eqref{psibound} and \eqref{basbo3.3}
  with $\omega=0$, and simplify the resulting expressions using the
  fact that $a_{k}$ is slowly varying.

  We apply $P_k$ to \eqref{claim12} to conclude that for any
  $s\in[0,\infty)$, $k\in\Z$, $m=1,\ldots,d$
  \begin{equation}\label{int22}
    \begin{split}
      \|P_k(A_m(s))\|_{F_k(T)\cap S_k^{1/2}(T)}&\lesssim 
\sum_{\al,\be,\ga=1}^d\int_s^{\infty}\|P_k(\overline{\psi_\al}(r)\partial_\be\psi_\ga(r))\|_{F_k(T)\cap S_k^{1/2}(T)}\,dr\\
      &+\!\! \sum_{\al,\be,\ga=1}^d\int_s^{\infty}\|P_k(\overline{\psi_\al}(r)\psi_\be(r)
      A_\ga(r))\|_{F_k(T)\cap S_k^{1/2}(T)}\,dr
    \end{split}
  \end{equation}
  With $k_0$ as before, using \eqref{usebound99}, the first term in
  the right-hand side of \eqref{int22} is dominated by
  \begin{equation}\label{int23}
    \begin{split}
      &\sum_{\al,\be,\ga=1}^d\sum_{j\geq k_0}\int_{2^{2j-1}}^{2^{2j+1}}\|P_k(\overline{\psi_\al}(r)\partial_\be\psi_\ga(r))\|_{F_k(T)\cap S_k^{1/2}(T)}\,dr \\
      &\lesssim\sum_{j\geq k_0}2^{-\sigma k}(1+2^{2k+2j})^{-4}2^{j}a_{-j}(2^ka_k(\sigma)+2^{-j}a_{-j}(\sigma))\\
      &\lesssim 2^{-\sigma k}\sum_{j\geq k_0}(1+2^{2k+2j})^{-4}(2^{j+k}a_{-j}a_k(\sigma)+a_{-j} a_{-j}(\sigma))\\
      &\lesssim 2^{-\sigma k}(1+s2^{2k})^{-4}b_{k,s}(\sigma),
    \end{split}
  \end{equation}
  where the last inequality follows easily from \eqref{newb} by
  checking the two cases $k+k_0\geq 0$ and $k+k_0\leq 0$.

  We estimate now $\|P_k(\overline{\psi_\al}(r)\psi_\be(r)\cdot
  A_\ga(r))\|_{F_k(T)}$ for $r\in[2^{2j-1},2^{2j+1}]$. It follows from
  \eqref{usebound99a} that for any $k'\in\Z$
  \begin{equation}\label{int7}
    \|P_{k'}(\overline{\psi_\al}(r)\psi_\be(r))\|_{F_k(T)\cap S_k^{1/2}(T)}\lesssim
2^{-\sigma k'}(1+2^{2k'+2j})^{-4}2^{-j}a_{-j} a_{\max(k',-j)}(\sigma).
  \end{equation}
  It follows from \eqref{basbo3.3}, \eqref{int21} and \eqref{akeps} that
  \begin{equation*}
    \|P_k(\overline{\psi_\al}(r)\psi_\be(r)\cdot A_\ga(r))\|_{F_k(T)\cap S_k^{1/2}(T)}
\lesssim B_12^{-\sigma k}\eps2^{-2j}a_{-j} a_{-j}(\sigma)
  \end{equation*}
  if $k+j\leq 0$, and
  \begin{equation*}
    \|P_k(\overline{\psi_\al}(r)\psi_\be(r)\cdot A_\ga(r))\|_{F_k(T)\cap S_k^{1/2}(T)}
\lesssim B_12^{-\sigma k}\eps (1+2^{2k+2j})^{-4}2^{-2j}b_{k,r}(\sigma)
  \end{equation*}
  if $k+j\geq 0$. Thus, with $k_0$ as before, the second term in the
  right-hand side of \eqref{int22} is dominated by
  \begin{equation*}
    \begin{split}
      &\sum_{\al,\be,\ga=1}^d\sum_{j\geq k_0}\int_{2^{2j-1}}^{2^{2j+1}}\|P_k(\overline{\psi_\al}(r)\psi_\be(r)\cdot A_\ga(r))\|_{F_k(T)\cap S_k^{1/2}(T)}\,dr\\
      &\lesssim B_12^{-\sigma k}\eps\sum_{j\geq k_0}(1+2^{2k+2j})^{-4}(\mathbf{1}_-(k+j)a_{-j} a_{-j}(\sigma)+\mathbf{1}_+(k+j)b_{k,2^{2j}}(\sigma))\\
      &\lesssim B_12^{-\sigma
        k}\eps(1+2^{2k+2k_0})^{-4}b_{k,2^{2k_0}}(\sigma).
    \end{split}
  \end{equation*}
  Thus, by \eqref{int22} and \eqref{int23}, for any
  $s\in[0,\infty)$ and
  $\sigma\in\{(d-2)/2,\sigma_0\}$ we obtain
  \begin{equation*}
    \|P_k(A_m(s))\|_{F_k(T)\cap S^{1/2}_k(T)}\lesssim 2^{-\sigma k}(1+s2^{2k})^{-4}b_{k,s}(\sigma)(1+B_1\eps),
  \end{equation*}
  which shows that $B_1\lesssim 1+B_1 \eps $ and further $B_1 \lesssim 1$. This completes the proof of
  \eqref{claimA3}.
\end{proof}

We prove now bounds on nonlinearity of the heat equation
\eqref{heatcov2}
\begin{equation}\label{heatcov22}
  K_m=2i\sum_{l=1}^d\partial_l(A_l\psi_m)-\sum_{l=1}^d(A_l^2+i\partial_lA_l)\psi_m+i\sum_{l=1}^d\Im(\psi_m\overline{\psi_l})\psi_l.
\end{equation}

\begin{lemma}\label{heatcontrol}
  (Control of the heat nonlinearities) For any
  $s\in[0,\infty)$, $k\in\Z$, $m=1,\ldots,d$ and
  $\sigma\in\{(d-2)/2,\sigma_0\}$
  \begin{equation} \label{nonlheat} \Big\|\int_0^s e^{(s-r)\Delta_x}
    P_k(K_m(r)) \,dr \Big\|_{F_k(T)} \lesssim \eps
    (1+s2^{2k})^{-4} 2^{-\sigma k} a_k(\sigma).
  \end{equation}
\end{lemma}

\begin{proof}[Proof of Lemma \ref{heatcontrol}] Assume
  $r\in[2^{2j-2},2^{2j+2}]$ for some $j\in\Z$ and assume that
  \begin{equation}\label{mf1}
    F\in\big\{A_l^2,\partial_l A_l,fg:l=1,\ldots,d,\,\,f,g\in\{\psi_n,\overline{\psi}_n:n=1,\ldots,d\}\big\}.
  \end{equation}
  We show first that for $F$ as in \eqref{mf1} we have
  \begin{equation}\label{mf2}
    \|P_k(F(r))\|_{F_k(T)\cap S^{1/2}_k(T)}\lesssim
2^{-\sigma k}(1+2^{2k+2j})^{-4} c_{k,j}(\sigma)
  \end{equation}
  where
  \begin{equation}\label{mf3}
    c_{k,j}(\sigma)=
    \begin{cases}
      2^{-j}a_{-j}a_{-j}(\sigma)&\text{ if }k+j\leq 0;\\
      2^{2k+j}a_{-j}a_k(\sigma)&\text{ if }k+j\geq 0.
    \end{cases}
  \end{equation}
  If $F$ is of the form $\partial_lA_l$ or $fg$ then the bound
  \eqref{mf2} follows from \eqref{claimA3} and \eqref{newb}, respectively 
  \eqref{usebound99a} (recall that $a_{k}(\sigma)$ is slowly varying). To
  prove this bound for $F=A_l^2$ we use \eqref{claimA3} and Lemma
  \ref{building} with $\omega=0$: if $k+j\leq 0$ then
  \begin{equation*}
    \|P_k(A_l^2(r))\|_{F_k(T)\cap S^{1/2}_k(T)}\lesssim 
\eps 2^{-\sigma k}2^{-j}a_{-j}a_{-j}(\sigma),
  \end{equation*}
  and if $k+j\geq 0$ then
  \begin{equation*}
    \|P_k(A_l^2(r))\|_{F_k(T)\cap S^{1/2}_k(T)}\lesssim 
\eps 2^{-\sigma k}2^{-j}b_{k,2^{2j}}(\sigma).
  \end{equation*}
  These bounds suffice to prove \eqref{mf2}.

  We prove now that, with $r\in[2^{2j-2},2^{2j+2}]$ as before,
  \begin{equation}\label{mf4}
    \|P_k(K_m(r))\|_{F_k(T)}\lesssim\eps 2^{-\sigma k}2^{2k}
(1+2^{2k+2j})^{-4}[a_k(\sigma)+2^{-3(k+j)/2}a_{-j}(\sigma)].
  \end{equation}
  In view of the formula \eqref{heatcov22}, it suffices to prove that
  \begin{equation}\label{mf5}
    \begin{split}
      \|P_k&(F(r)f(r))\|_{F_k(T)}+2^k\|P_k(A_l(r)f(r))\|_{F_k(T)}\\
      &\lesssim \eps 2^{-\sigma
        k}2^{2k}(1+2^{2k+2j})^{-4}[a_k(\sigma)+2^{-3(k+j)/2}a_{-j}(\sigma)],
    \end{split}
  \end{equation}
  for any $F$ as in \eqref{mf1} and
  $f\in\{\psi_n,\overline{\psi}_n:n=1,\ldots,d\}$. In the proof of
  \eqref{mf5} we need to use Lemma \ref{building} with $\omega=1/2$.
  From \eqref{psibound} we have
  \begin{equation}\label{mf9}
    \|P_k(f(r))\|_{F_k(T)}\leq 2^{-\sigma k} a_k(\sigma)(1+2^{2k+2j})^{-4}.
  \end{equation}
  We combine this with  \eqref{mf2}, using  Lemma \ref{building} with $\omega=1/2$,
to obtain
  \begin{equation*}
    \|P_k(F(r)f(r))\|_{F_k(T)}\lesssim 2^{-\sigma k}2^{k-j}2^{-(k+j)/2}
a_{-j}^2a_{-j}(\sigma),
\qquad k+j \leq 0
  \end{equation*}
  \begin{equation*}
    \|P_k(F(r)f(r))\|_{F_k(T)}\lesssim 2^{-\sigma k}(1+2^{2k+2j})^{-4}2^{2k}
a_{-j}^2a_k(\sigma) \qquad k + j > 0
  \end{equation*}
 which imply \eqref{mf5} for the first term.
  By  \eqref{claimA3}, \eqref{mf9}, and Lemma \ref{building} with
  $\omega=1/2$
  \begin{equation*}
    2^k\|P_k(A_l(r)f(r))\|_{F_k(T)}\lesssim 2^{2k}2^{-\sigma k}2^{-(k+j)/2}a_{-j}^2a_{-j}(\sigma)\qquad k+j \leq 0
  \end{equation*}
  \begin{equation*}
    2^k\|P_k(A_l(r)f(r))\|_{F_k(T)}\leq C2^{-\sigma k}(1+2^{2k+2j})^{-4}2^{2k}a_{-j}^2a_k(\sigma)\qquad k + j > 0
  \end{equation*}
   These bounds imply \eqref{mf5} for the second term.

  We use now \eqref{mf4} to prove \eqref{nonlheat}. Assume
  $s\in[2^{2k_0-1},2^{2k_0+1})$ for some $k_0\in\Z$. We use
  \eqref{hebound}. If $k+k_0\leq 0$ then
  \begin{equation*}
    \begin{split}
      \Big\|\int_0^s &e^{(s-r)\Delta_x} P_k(K_m(r)) \,dr \Big\|_{F_k(T)}\lesssim
\sum_{j\leq k_0}\int_{2^{2j-1}}^{2^{2j+1}}\|P_k(K_m(r))\|_{F_k(T)}\,dr\\
      &\lesssim \sum_{j\leq k_0}2^{2j}\eps 2^{-\sigma
        k}2^{2k}2^{-3(k+j)/2}a_{-j}(\sigma)\lesssim \eps
2^{-\sigma k}2^{(k+k_0)/2}a_{-k_0}(\sigma),
    \end{split}
  \end{equation*}
  which suffices. If $k+k_0\geq 0$ then, with $d_{k,j}$ as in the
  right-hand side of \eqref{mf4},
  \begin{equation*}
    \begin{split}
      \Big\|\int_0^s &e^{(s-r)\Delta_x} P_k(K_m(r)) \,dr \Big\|_{F_k(T)}\\
      \leq & \ \int_0^{s/2} \|e^{(s-r)\Delta_x} P_k(K_m(r))\|_{F_k(T)} \,dr+\int_{s/2}^s \|e^{(s-r)\Delta_x} P_k(K_m(r))\|_{F_k(T)} \,dr\\
      \lesssim & \ \sum_{j\leq k_0}2^{-20(k+k_0)}2^{2j}d_{k,j}+2^{-2k}d_{k,k_0}\\
      \lesssim & \ 2^{-20(k+k_0)}\sum_{j\leq k_0}\eps 2^{-\sigma k}2^{2k}(1+2^{2k+2j})^{-4}[2^{2j}a_{k}(\sigma)+2^{j/2}2^{-3k/2}a_{-j}(\sigma)]\\
      & \ +\eps 2^{-\sigma k}(1+2^{2k+2k_0})^{-4}a_k(\sigma)\\
      \lesssim & \ \eps 2^{-\sigma
        k}(1+2^{2k+2k_0})^{-4}a_k(\sigma)
    \end{split}
  \end{equation*}
  which suffices. This completes the proof of the lemma.
\end{proof}

We are now able to prove the $F_k(T)$ bounds \eqref{psifk} in
Proposition~\ref{heatfk}.  In view of \eqref{heatcov2} we have
\begin{equation*}
  P_k(\psi_m(s))=e^{s\Delta_x}P_k(\psi_m(0))+
  \int_0^s e^{(s-r)\Delta_x} P_k(K_m(r)) \,dr.
\end{equation*}
Thus, from Lemma \ref{heatcontrol} and
\eqref{psizero},\eqref{hebound} we obtain
\[
\| P_k \psi_m(s)\|_{F_k(T)} \lesssim 2^{-\sigma k}(1+s2^{2k})^{-4} 
(b_k(\sigma) +\eps a_k(\sigma)), \qquad \sigma \in \{(d-2)/2, \sigma_0\}
\]
Due to the definition of $a_k(\sigma)$ in \eqref{aksdef} this implies
that  $a_k(\sigma) \lesssim b_k(\sigma) +\eps a_k(\sigma)$, and further
$a_k(\sigma) \lesssim b_k(\sigma)$. Then \eqref{psifk} follows.

Next we consider the $F_k$ bound \eqref{apsifk} for the functions
$P_k(A_m(s) \psi_l(s))$. It follows from Lemma \ref{building} (with
$\omega=1/2$), Lemma \ref{Aprop} and \eqref{psibound} that for any
$l,m=1,\ldots,d$, and $r\in[2^{2j-2},2^{2j+2}]$
\begin{equation*}
  \|P_k(A_l(r)\psi_m(r))\|_{F_k(T)}\lesssim 2^{-\sigma
    k}2^k2^{-(k+j)/2}a_{-j}^2a_{-j}(\sigma), \qquad k\leq -j
\end{equation*}
respectively
\begin{equation*}
  \|P_k(A_l(r)\psi_m(r))\|_{F_k(T)}\lesssim 2^{-\sigma
    k}2^{k}(1+2^{2j+2k})^{-4}a_{-j}^2a_{k}(\sigma), \qquad k\geq -j.
\end{equation*}
Then \eqref{apsifk} follows since $a_k$, $a_k(\sigma)$ are slowly varying.

Next we turn our attention to the $L^{p_d}$ bounds in
Proposition~\ref{heatfk}. We start with a general lemma, similar to
Lemma \ref{building}.

\begin{lemma}\label{building2} Assume that $T\in(0,2^{2\mathcal{K}}]$,
  $f,g\in {H}^{\infty,\infty}(T)$, $P_k f\in S_k^\omega(T)$,
  $P_k g\in L^{p_d}_{t,x}$ for some $\omega\in[0,1/2]$ and any $k\in\Z$, and
  \begin{equation*}
    \mu_k=\sum_{|k'-k|\leq 20}\|P_{k'}f\|_{S_{k'}^\omega(T)},\quad\nu_k=\sum_{|k'-k|\leq 20}\|P_{k'}g\|_{L^{p_d}_{t,x}}.
  \end{equation*}
  Then, for any $k\in\Z$
  \begin{equation}\label{bas1}
      \|P_k(fg)\|_{L^{p_d}_{t,x}} \lesssim \sum_{k'\leq
        k}2^{\frac{k'd}2}(\mu_{k'}\nu_k+
      2^{\frac{d}{d+2}(k-k')}\mu_k\nu_{k'})+
      2^{\frac{kd}2}\sum_{k'\geq k}2^{-\omega(k'-k)}\mu_{k'}\nu_{k'}.
  \end{equation}
\end{lemma}

\begin{proof}[Proof  of  Lemma \ref{building2}] 
  We use the same bilinear Littlewood-Paley decomposition \eqref{bilp}
  as in the proof of Lemma \ref{building}, and estimate each term. If
  $|k_2-k| \leq 4$ and $k_1 \leq k-4$ then by the Sobolev embedding
  $\|P_{k}f\|_{L^\infty_{t,x}}\lesssim 2^{dk/2}\mu_k$ we have
\[
\|P_k(P_{k_1}f\cdot P_{k_2}g)\|_{L^{p_d}_{t,x}} \lesssim 
\|P_{k_1}f\|_{L^\infty_{t,x}}\|P_{k_2}g\|_{L^{p_d}_{t,x}} \lesssim 2^{k_1 d/2} \mu_{k_1} \nu_k.
\]
If $|k_1-k| \leq 4$, $k_2 \leq k-4$  then we use the Sobolev embeddings
 $\|P_{k}f\|_{L^\infty_t L^{p_d}_x}\lesssim
  2^{\frac{d}{d+2}k}\mu_k$ and  $\|P_kg\|_{L^{p_d}_tL^\infty_x}\lesssim
  2^{(\frac{d}2 -\frac{d}{d+2}) k}\nu_{k}$ to obtain
\[
\|P_k(P_{k_1}f\cdot P_{k_2}g)\|_{L^{p_d}_{t,x}} \lesssim  \|P_{k_1}f\|_{L^\infty_tL^{p_d}_x}\|P_{k_2}g\|_{L^{p_d}_tL^\infty_x} \lesssim 2^{dk_2/2}
 2^{\frac{d}{d+2}(k-k')}\mu_k\nu_{k_2}.
\]
Finally if $k_1,k_2\geq k-4$ and $|k_1-k_2|\leq 8$ then we similarly have
\[
\|P_k(P_{k_1}f\cdot P_{k_2}g)\|_{L^{p_d}_{t,x}} \lesssim 2^{k(d/2+\omega)}\|P_{k_1}f\|_{L^\infty_tL^{2_\omega}_x}\|P_{k_2}g\|_{L^{p_d}_{t,x}} \lesssim 2^{k(d/2+\omega)}2^{-\omega
        k_1}\mu_{k_1}\nu_{k_1}.
\]

 %  \begin{equation*}
%     \begin{split}
%       \|P_{k}(fg)\|_{L^{p_d}}\lesssim & \sum_{k_1,k_2\geq
%         k-4}^{|k_1-k_2|\leq 8}\|P_k(P_{k_1}f\cdot
%       P_{k_2}g)\|_{L^{p_d}} +\sum_{k_2\leq k-4}^{|k_1-k|\leq
%         4}\|P_k(P_{k_1}f\cdot
%       P_{k_2}g)\|_{L^{p_d}} \\ & + \sum_{k_1\leq k-4}^{|k_2-k|\leq 4}\|P_k(P_{k_1}f\cdot P_{k_2}g)\|_{L^{p_d}}\\
%       \lesssim & \sum_{k_1,k_2\geq k-4}^{|k_1-k_2|\leq 8}2^{k(d/2+\omega)}\|P_{k_1}f\|_{L^\infty_tL^{2_\omega}_x}\|P_{k_2}g\|_{L^{p_d}}\\
%       &+\sum_{k_2\leq k-4}^{|k_1-k|\leq 4}\|P_{k_1}f\|_{L^\infty_tL^{p_d}_x}\|P_{k_2}g\|_{L^{p_d}_tL^\infty_x}+\sum_{k_1\leq k-4}^{|k_2-k|\leq 4}\|P_{k_1}f\|_{L^\infty_{x,t}}\|P_{k_2}g\|_{L^{p_d}}\\
%       \lesssim &\ \sum_{k'\geq k}2^{k(d/2+\omega)}2^{-\omega
%         k'}\mu_{k'}\nu_{k'}+\sum_{k'\leq k}2^{dk'/2}
%       (2^{\frac{d}{d+2}(k-k')}\mu_k\nu_{k'}+\mu_{k'}\nu_k),
%     \end{split}
%   \end{equation*}
%   using the Sobolev imbedding bounds $\|P_{k}f\|_{L^\infty L^p_d}\leq
%   C2^{\frac{d}{d+2}k}\mu_k$, $\|P_{k}f\|_{L^\infty_{x,t}}\leq
%   C2^{dk/2}\mu_k$, and $\|P_kg\|_{L^{p_d}_tL^\infty_x}\leq C
%   2^{(\frac{d}2 -\frac{d}{d+2}) k}\nu_{k}$.
The bound \eqref{bas1} follows by summing up the three cases above.
\end{proof}

We prove now $L^{p_d}_{t,x}$ bounds on the connection coefficients 
$A_m(0)$, $m=1,\ldots,d$.
\begin{lemma}\label{Lpda}
  For any $k\in\Z$, $m=1,\ldots,d$ and $\sigma\in\{(d-2)/2,\sigma_0\}$
  \begin{equation}\label{alpd}
    \begin{split}
 \|P_k(A_m(0))\|_{L^{p_d}_{t,x}}   \lesssim 2^{-\sigma k}b_k(\sigma).
    \end{split}
  \end{equation}
\end{lemma}

\begin{proof}[Proof  of Lemma \ref{Lpda}]
We start from the identity
  \eqref{Aform}
  \begin{equation}\label{mf21}
    A_m(0)=-\sum_{l=1}^d\int_0^{\infty}\Im\big(\overline{\psi_m}\DD_l\psi_l\big)(s)\,ds.
  \end{equation}
where, as before, $\DD_l\psi_l=\partial_l\psi_l+iA_l\psi_l$. From \eqref{psifk}
we have
\[
\|P_k \psi_m(s)\|_{S_k^0} \lesssim  2^{-\sigma k} (1+s2^{2k})^{-4} b_k(\sigma) 
\]
while  from \eqref{psifk}, \eqref{apsifk} we obtain
\[
\|P_k (\DD_l\psi_l(s))\|_{L^{p_d}_{t,x}} \lesssim  2^k 2^{-\sigma k} (s2^{2k})^{-3/8}
(1+s2^{2k})^{-3} b_k(\sigma).
\]
We use now \eqref{bas1} with $\omega=0$ to estimate
\[
\begin{split}
\|P_k(A_m(0))\|_{L^{p_d}_{t,x}} \lesssim &\  \sum_{l=1}^d \int_0^{\infty} \| \overline{\psi_m(s)}\DD_l\psi_l(s)\|_{L^{p_d}_{t,x}} ds \\
\lesssim &\ 2^{-\sigma k} \sum_{k' \leq  k} b_k(\sigma) b_{k'}  2^{{k'}+k}
\int_0^\infty  (s2^{2k})^{-3/8} (1+s2^{2k})^{-3}    ds \\
& \ +  2^{-\sigma k} \sum_{k' \leq  k} b_k(\sigma) b_{k'}  
 2^{2{k'}} 2^{\frac{d}{d+2}(k-{k'})} \int_0^\infty
(s2^{2{k'}})^{-3/8} (1+s2^{2k})^{-4}     ds \\ 
& \ +  \sum_{{k'} \geq  k}  2^{-\sigma {k'}} b_{k'}(\sigma) b_{k'}  2^{(k-{k'})d/2}
 2^{2{k'}} \int_0^\infty
(s2^{2{k'}})^{-3/8} (1+s2^{2{k'}})^{-7}     ds \\
\lesssim &\  2^{-\sigma k}b_k(\sigma) \sum_{{k'} \leq k}  b_{k'} 2^{({k'}-k)/4} +  
\sum_{{k'} \geq  k}  b_{k'}(\sigma) b_{k'}  2^{-\sigma {k'}} 2^{(k-{k'})d/2} \\
\lesssim &\  2^{-\sigma k} b_k b_k(\sigma).
\end{split}
\]
Thus \eqref{alpd} follows.
\end{proof}

This concludes the proof of part (a) of Proposition~\ref{heatfk}.  We
next turn our attention to part (b). We first prove $L^{p_d}$ bounds
on the field $\psi_{d+1}(s)$.

\begin{lemma}\label{L3L6}
  For any $k\in\Z$ and $\sigma\in\{(d-2)/2,\sigma_1-1\}$
  \begin{equation}\label{mf20}
    \begin{split}
   \|P_k(\psi_{d+1}(s))\|_{L^{p_d}_{t,x}}\leq
      C2^k 2^{-\sigma k}b_k(\sigma)(1+s2^{2k})^{-2}
    \end{split}
  \end{equation}
\end{lemma}

\begin{proof}[Proof of Lemma \ref{L3L6}] 
 We use the heat equation \eqref{heatcov2} for $\psi_{d+1}$,
  \begin{equation}\label{mf31}
    \begin{split}
      &(\partial_s-\Delta_x)\psi_{d+1}=K(\psi_{d+1});\\
      &K(\psi)=2i\sum_{l=1}^d\partial_l(A_l\psi)-\sum_{l=1}^d(A_l^2+i\partial_lA_l)\psi +i\sum_{l=1}^d\Im(\psi \overline{\psi_l})\psi_l.
    \end{split}
  \end{equation}
We rewrite this equation in the form
\begin{equation}
\psi_{d+1}(s) = e^{s \Delta} \psi_{d+1}(0) + \int_{0}^s e^{(s-r)
  \Delta} K(\psi_{d+1})(r) dr.
\label{mfint}\end{equation}

Assuming that
\begin{equation} \label{phicond} 
\| P_k \psi(s) \|_{L^{p_d}_{t,x}} \lesssim
  2^{-(\sigma-1)k} b_k(\sigma) (1+s2^{2k})^{-2} 
\end{equation}
we claim the following
\begin{equation} \label{Kphi}
\left\| \int_0^s e^{(s-r)\Delta_x} P_k K(\psi)(r) \,dr \right\|_{L^{p_d}_{t,x}}
\lesssim 
\eps^2 2^{-(\sigma-1)k} b_k(\sigma) (1+s2^{2k})^{-2}.  
\end{equation}
By \eqref{psidpu} the function $e^{s\Delta_x} \psi_{d+1}(0)$ satisfies
\eqref{phicond}.  Then a standard iteration argument shows that the
solution $\psi_{d+1}$ to \eqref{mfint} also satisfies \eqref{phicond}.
We note that by standard $L^\infty$ bounds for the heat equation,
\eqref{mf31} admits an unique bounded solution on each interval
$[0,S]$, with $S > 0$.  Therefore the solution obtained iteratively
must coincide with $\psi_{d+1}$.

It remains to prove our claim.  As in the proof of Lemma
\ref{heatcontrol}, assume that
  \begin{equation*}
    F\in\big\{A_l^2,\partial_l A_l,fg:l=1,\ldots,d,\,\,f,g\in\{\psi_n,\overline{\psi}_n:n=1,\ldots,d\}\big\},
  \end{equation*}
Due to  \eqref{mf2} and \eqref{mf3} we have 
  \begin{equation}\label{mf27}
    \|P_k F(r)\|_{S^{1/2}_k(T)}
\lesssim 2^{- (\sigma-1) k}(1+s2^{2k})^{-2} (s2^{2k})^{-\frac58}
 b_{k}b_{k}(\sigma).
  \end{equation}
Also, by  Proposition \ref{Aprop}, 
  \begin{equation}\label{mf28}
    \|P_k A_l(r)\|_{S^{1/2}_k(T)} 
\lesssim 2^{- \sigma k}(1+s2^{2k})^{-3} (s2^{2k})^{-\frac18}
 b_{k}b_{k}(\sigma).
  \end{equation}

  Using \eqref{bas1} (with $\omega=1/2$), \eqref{mf27}, \eqref{mf28}
  and \eqref{phicond} it follows that
\begin{equation*}\label{mf29}
      \|P_k(F(r)\psi(r))\|_{L^{p_d}_{t,x}}+2^k\|P_k(A_l(r)\psi(r))\|_{L^{p_d}_{t,x}}
      \lesssim 2^{-(\sigma+1) k}
(1+s 2^{2k})^{-2}(s2^{2k})^{-\frac78} b_{k}^2 b_{k}(\sigma) 
  \end{equation*}
  for any $k\in\Z$, $l=1,\ldots,d$, and
  $\sigma\in\{(d-2)/2,\sigma_0\}$. Since $b_k^2 \leq \eps^2$ we 
get
\[
\| P_k K (\phi)\|_{L^{p_d}_{t,x}}
      \lesssim \eps^2 2^{-(\sigma -3) k}
(1+s 2^{2k})^{-2}(s2^{2k})^{-\frac78}  b_{k}(\sigma).
\]
This implies \eqref{Kphi} after
integration with respect to $s$ since
\[
\int_{0}^s (1+(s-r)2^{2k})^{-N} (1+r 2^{2k})^{-2}(r2^{2k})^{-\frac78} dr
\lesssim 2^{-2k} (1+s 2^{2k})^{-2}.
\]
\end{proof}

We conclude the proof of Proposition~\ref{heatfk} with the $L^2$ 
bounds on $P_k A_{d+1}(0)$.

\begin{lemma}
The connection coefficient $A_{d+1}$ satisfies  
 \begin{equation}
\| P_k (A_{d+1}(0)) \|_{L^2_{t,x}}
      \lesssim \eps 2^{-\sigma k} b_k(\sigma), \qquad d \geq 3
\label{adunu} \end{equation}
respectively
 \begin{equation}
\|A_{d+1}(0)\|_{L^2_{t,x}} \lesssim \eps^2,
\qquad
\| P_k (A_{d+1}(0)) \|_{L^2_{t,x}}
      \lesssim  2^{-\sigma k} b_{> k}^2(\sigma), \qquad d =2
\label{adunu2} \end{equation}
\end{lemma}

\begin{proof}[Proof  of  Lemma  \ref{adunu2}]
 To bound $A_{d+1}$ we start from the
  identity \eqref{Aform}
  \begin{equation}\label{mf62}
    \begin{split}
      A_{d+1}(0)=-\sum_{l=1}^d\int_0^{\infty}
\Im\big(\overline{\psi_{d+1}}\DD_l \psi_l\big)(s)\,ds.
    \end{split}
  \end{equation}
For $\psi_{d+1}$ we use the bound \eqref{mf20}. For $\DD_l \psi_l$,
by \eqref{psifk} and \eqref{apsifk},
\begin{equation}
\| \DD_l \psi_l(s)\|_{L^\infty_t L^2_x \cap L^{p_d}_{t,x}} \lesssim 2^{k}
2^{-\sigma k} a_k(\sigma) (s2^{2k})^{-3/8} (1+s2^{2k})^{-2}.
\label{mf21a}\end{equation}
To multiply $L^\infty_t L^2_x \cap L^{p_d}$ and $L^{p_d}$ functions we
will use the following bound: if
  \begin{equation*}
    \mu_k=\sum_{|k'-k|\leq 20}\|P_{k'}f\|_{L^{p_d}_{t,x}
      \cap L^\infty_t L^2_x},\quad
\nu_k=\sum_{|k'-k|\leq 20}\|P_{k'}g\|_{L^{p_d}_{t,x}}
  \end{equation*}
 then, for any $k \in \Z$,
  \begin{equation} \label{resaux1} \| P_k (fg) \|_{L^2_{t,x}}
 \lesssim \sum_{j \leq k} 2^{j(d-2)/2} (\mu_j \nu_k + \mu_k \nu_j) 
+ \sum_{j
      \geq k} 2^{k(d-2)/2}\mu_{j}\nu_{j}.
  \end{equation}
This is easy to prove as in Lemma~\ref{building2} using a bilinear
Littlewood-Paley decomposition and Sobolev embeddings.

  In dimension $d \geq 3$ we estimate using  \eqref{resaux1},
  \eqref{mf20} and \eqref{mf21a}:
\begin{equation*}
\begin{split}
\|P_k A_{d+1}(0)\|_{L^2_{t,x}} \lesssim &\sum_{l=1}^d\int_0^{\infty}
\| P_k \big(\overline{\psi_{d+1}}\DD_l \psi_l\big)(s)\|_{L^2_{t,x}} ds\\
\lesssim &\ 2^{-\sigma k} \sum_{j \leq  k} 2^{j+k} b_k(\sigma) b_{j}  
\int_0^\infty  (s2^{2j})^{-3/8} (1+s2^{2k})^{-2}    ds \\
& \ +  \sum_{{j} \geq  k}  2^{-\sigma {j}} 2^{2j} 2^{(k-{j})(d-2)/2}
b_{j}(\sigma) b_{j}  \int_0^\infty
(s2^{2{j}})^{-3/8} (1+s2^{2{j}})^{-4}     ds \\
\lesssim &\  2^{-\sigma k}b_k(\sigma) \sum_{{j} \leq k}  b_{j} 2^{({j}-k)/4} +  
\sum_{{j} \geq  k}  b_{j}(\sigma) b_{j}  2^{-\sigma {j}} 2^{(k-{j})(d-2)/2} \\
\lesssim &\  2^{-\sigma k} b_k b_k(\sigma).
\end{split}
\end{equation*}
 In dimension $d=2$ the same computation applies, with the only
 difference that the last sum can only be bounded by $b_{>k}^2(\sigma) 2^{-\sigma k}$. This gives the second part of
 \eqref{adunu}. For the first part we replace \eqref{resaux1} by
  \begin{equation} \label{resaux2}
 \| fg \|_{L^2_{t,x}}
 \lesssim \sum_k \mu_k \sum_{j \leq k} \nu_j +
 \sum_k \nu_k \sum_{j \leq k} \mu_j.
  \end{equation}
Then repeating the above computation we obtain
\[
\|A_{d+1}(0)\|_{L^2_{t,x}} \lesssim  \sum_k b_k  \sum_{j \leq k} b_j
2^{({j}-k)/4} \lesssim \sum_k b_k^2.
\]
\end{proof}

\section{Perturbative analysis of the Schr\"{o}dinger
  equation}\label{PERTURB}

In this section we prove Proposition~\ref{schgk}. 
For $k \in \Z$ we denote 
\begin{equation}
b(k) = \sum_{m=1}^d \| P_k \psi_m(0)\|_{G_k(T)}.
\end{equation}
For $\sigma \in \{(d-2)/2,\sigma_0\}$ we introduce the frequency
envelopes
\begin{equation}
b_k(\sigma) = \sup_{j \in \Z} 2^{\sigma j} 2^{-\delta |k-j|} b(j).
\label{bksdef}\end{equation}
These are finite and belong to $l^2$ due to \eqref{goodbounds4} and
Sobolev embeddings. We also have
\begin{equation}
\| P_k \psi_m(0)\|_{G_k(T)} \lesssim 2^{-\sigma k} b_k(\sigma).
\label{bksbd} \end{equation}

For \eqref{bootschout} we need to show that
$b_k(\sigma) \lesssim c_k(\sigma)$. On the other hand from the
bootstrap assumption \eqref{bootsch} we know that $b_k \leq
\eps_0^{-\frac12} c_k$.  In particular
\begin{equation}
\sum_{k \in \Z} b_k^2 \leq \eps_0.
\label{bkeps}\end{equation}
For the connection coefficients $A_m$ we use Proposition~\ref{heatfk}
with $\eps = \eps_0^\frac12$. The assumption \eqref{psizero} follows
from the inclusion $G_k \subset F_k$. We also need to verify that the
assumption \eqref{psidpu} in Proposition~\ref{heatfk} follows from
\eqref{psizero} if $\phi$ solves the Schr\"odinger map equation:

\begin{lemma}\label{alexnew40}
  If $b_k(\sigma)$ are as above then the field $\psi_{d+1}(0)$ satisfies
  the bounds
\begin{equation}
\|P_k \psi_{d+1}(0)\|_{L^{p_d}_{t,x}} \lesssim b_k(\sigma) 2^{-(\sigma-1) k}.
\end{equation}
\end{lemma}

\begin{proof}[Proof of Lemma \ref{alexnew40}] 
We use the identity
\eqref{schcov}
\begin{equation*}
  \psi_{d+1}(0)=i\sum_{l=1}^d(\partial_l\psi_l(0)+iA_l(0)\psi_l(0)).
\end{equation*}
From \eqref{psifk}, \eqref{pkam} and \eqref{goodbounds3} we have
\begin{equation}
  \| P_k  \psi_l(0)\|_{L^\infty_t L^2_x \cap L^{p_d}_{t,x}} + 
  \|P_k A_l(0)\|_{L^\infty_t L^2_x \cap L^{p_d}_{t,x}} \lesssim 
2^{-\sigma k} b_k(\sigma).
  \label{psilal}\end{equation}
The $L^{p_d}_{t,x}$ bound for the first term $\partial_l\psi_l(0)$
immediately follows. The second term $A_l(0)\psi_l(0)$ can also be
estimated by \eqref{psilal} using \eqref{resaux1}, except in dimension
$d=2$. If $d = 2$ then from \eqref{resaux1} and \eqref{psilal} we
still obtain
\[
\| P_k (P_{\leq k+4} A_l(0)  \psi_l(0))\|_{L^2_{t,x}} 
\lesssim 2^{-\sigma k} b_k(\sigma).
\]
However, in order to handle the high-high frequency interactions
we need a stronger bound on $A_l$ which follows from
Lemma \ref{Aprop}, namely
\begin{equation*}
2^{\frac{k}2}\|P_k A_l(0)\|_{L^4_t L^2_x} \lesssim
 \|P_k(A_l(0))\|_{S^{1/2}_k(T)}\lesssim \eps.
\end{equation*}
Combining this with \eqref{psilal} for $A_l$ we easily obtain the
remaining bound
\[
\| P_k (P_{> k+4} A_l(0)  \psi_l(0))\|_{L^2_{t,x}} 
\lesssim 2^{-\sigma k} b_k(\sigma).
\]
\end{proof}

Thus we can apply Proposition~\ref{TaoHeat2} and Corollary~\ref{heatcor}.
For convenience we summarize the two main ingredients 
which are to be
used in the sequel.  On one hand, for $l = 1,\cdots,m$ we have the
bounds
\begin{equation}\label{ro1}
\left\{ \begin{array}{l}
    \|\psi_l(s)\|_{F_k(T)} \lesssim
    2^{-\sigma k}b_k(\sigma)(1+s 2^{2k})^{-4}
\cr \cr
\|P_k \DD_l \psi_l(s)\|_{F_k(T)}\lesssim
  2^k  2^{-\sigma k}b_k(\sigma)(s2^{2k})^{-\frac38}(1+s 2^{2k})^{-2}
  \end{array}\right.
\end{equation}
which follow from \eqref{psifk} and \eqref{apsifk}. On the other
hand,   for each 
\[
F \in \{\psi_m(0) \overline{\psi_l}(0),
A_l^2(0), A_{d+1}(0)\}
\]
we have the bounds
\begin{equation}\label{psa}
\left\{  \begin{array}{ll}
    \| P_k F\|_{L^2_{t,x}}  \lesssim
    \eps_0^\frac12 2^{-\sigma k} b_k(\sigma), & \qquad d \geq 3 \cr \cr
  \| F \|_{L^2_{t,x}} 
  \lesssim \varepsilon_0, \qquad  \| P_k F\|_{L^2_{t,x}} \lesssim
     2^{-\sigma k} b_{>k}^2(\sigma), & \qquad   d = 2
\end{array} \right.
\end{equation}
Also by Sobolev embeddings and Littlewood-Paley theory,
\begin{equation}\label{psasob}
\| F \|_{L^2_t L^d_x}\lesssim [\sum_{k\in\Z}\|P_kF \|_{L^2_t L^d_x}^2]^{1/2}\lesssim[\sum_{k\in\Z}\eps_0b_k^2]^{1/2}\lesssim \eps_0.
\end{equation}
 Here the $A_{d+1}$ bound is from \eqref{aadunu} and
\eqref{aadunub}, while the $\psi_m(0) \overline{\psi_l}(0)$ and the
$A_l^2(0)$ bounds follow from \eqref{psilal} due to \eqref{resaux1}.

For $m=1,\ldots,d$ we denote the nonlinearity of the Schr\"{o}dinger equation
\eqref{schcov2}
\begin{equation}\label{ro2}
  L_m=-2i\sum_{l=1}^dA_l\partial_l\psi_m+\big(A_{d+1}+
\sum_{l=1}^d(A_l^2-i\partial_lA_l)\big)\psi_m-i\sum_{l=1}^d\psi_l\Im(\overline{\psi_l}\psi_m).
\end{equation}
For simplicity of notation, in this section we use sometimes $\psi_m$
for $\psi_m(0)$ and $A_m$ for $A_m(0)$. 

\begin{proposition}\label{SchPer}
  (Control of the Schr\"{o}dinger nonlinearities) For any
  $m=1,\ldots,d$ and $\sigma\in\{(d-2)/2,\sigma_0\}$ we have
  \begin{equation}\label{ro3}
   \|P_k(L_m)\|_{N_k(T)} \lesssim 
\varepsilon_0 2^{-\sigma k} b_k(\sigma).
  \end{equation}
\end{proposition}  

Before proving the above proposition we show how to use it to conclude
the proof of Proposition~\ref{schgk}.  Applying  Proposition
\ref{linearmainrep} for the equations \eqref{schcov2}, by 
\eqref{ro3} and \eqref{schid} we obtain
\begin{equation}
\| P_k \psi_m(0) \|_{G_k(T)} \lesssim 
2^{\sigma k}( c_k(\sigma) + \eps_0 b_k(\sigma)).
\label{ceb}\end{equation}
By the definition of $b_k(\sigma)$ this implies that
\[
 b_k(\sigma) \lesssim c_k(\sigma) + \eps_0 b_k(\sigma).
\]
Hence 
\begin{equation*}
 b_k(\sigma) \lesssim c_k(\sigma)
\end{equation*}
which combined with \eqref{ceb} gives \eqref{bootschout},
concluding the proof of Proposition \ref{schgk}.

The rest of this section is concerned with the proof of Proposition
\ref{SchPer}, which follows from Lemma \ref{easysch} and Lemma~\ref{hardsch}. 
We begin our analysis with some bilinear estimates:

\begin{lemma}
(a) If  $|k_1- k| \leq 80$ and $f\in F_{k_1}(T)$ then
  \begin{equation} 
\| P_k(F f)\|_{N_k(T)} \lesssim \|F\|_{L^2_t L^d_x}\|f\|_{F_{k_1}(T)}.
\label{Fhfl}  \end{equation}
(b) If  $k_1 \leq k-80$ and $f\in F_{k_1}(T)$ then
  \begin{equation} \| P_k(F f) \|_{N_k(T)} \leq
    2^{k_1(d-2)/2} 2^{(k_1-k)/2}\| F\|_{L^2_{t,x}} \| f \|_{F_{k_1}(T)}.
\label{Fhfh}  \end{equation}
(c)  If  $k \leq k_1-80$  and $g\in G_{k_1}(T)$ then
  \begin{equation} \| P_k(F g) \|_{N_k(T)} \leq
    2^{k_1(d-2)/2} 2^{(k-k_1)/6}\| F\|_{L^2_{t,x}} \|g \|_{G_{k_1}(T)}.
\label{Flgh}  \end{equation}
\label{bilnk}\end{lemma}

\begin{proof}[Proof of  Lemma \ref{bilnk}]
Part (a) follows from the definition of $F_k(T)$, $N_k(T)$, and 
\[
\| F f\|_{L^{p'_d}_{t,x}} \lesssim \|F\|_{L^2_t L^d_x}\|f\|_{L^\infty_t L^2_x \cap L^{p_d}_{t,x}}.
\]
Part (b) also follows directly from the definitions, since
  \begin{equation*}
    \begin{split}
      \|P_k(F f)\|_{N_{k}(T)}\lesssim 2^{-k/2}\sup_{\e\in\mathbb{S}^{d-1}}
\|Ff\|_{L^{1,2}_{\e,W_{k-40}}}
      \lesssim
      2^{-k/2}\sup_{\e\in\mathbb{S}^{d-1}}\|f\|_{L^{2,\infty}_{\e,W_{k_1+40}}}
\|F\|_{L^2_{x,t}}.
    \end{split}
  \end{equation*}
Finally for part (c) we use Sobolev embeddings if $d \geq 3$, 
\[
\|P_k(F g)\|_{N_k(T)} \lesssim \|P_k(F g)\|_{L^{p_d'}_{t,x}} \lesssim 2^{k(d-2)/2}
\|F\|_{L^2_{t,x}} \|g\|_{L^\infty_t L^2_x \cap L^{p_d}_{t,x}}. 
\]
If $d=2$ we need to use the lateral Strichartz estimates.  
Using an angular partition of unity in frequency 
we can write
\[
g = g_1 + g_2, \qquad
\|g_1\|_{L^{6,3}_{\e_1}}+\|g_2\|_{L^{6,3}_{\e_2}} \lesssim 
2^{-k_1/6} \|g \|_{G_k(T)}.
\]
Then we have
\[
\begin{split}
\|P_k(Fg)\|_{N_k(T)} \lesssim & \ 2^{k/6} 
\big(\|Fg_1\|_{L^{\frac32,\frac65}_{\e_1}}
+\|Fg_2\|_{L^{\frac32,\frac65}_{\e_2}}\big)
\\
\lesssim & \  2^{k/6} \|F\|_{L^2} \big( \|g_1\|_{L^{6,3}_{\e_1}}
+ \|g_2\|_{L^{6,3}_{\e_2}}\big)
\\
\lesssim & \ 2^{(k-k_1)/6} \|F\|_{L^2} \|g\|_{G_{k_1}(T)}.
\end{split}
\]
\end{proof}

The above lemma suffices in order to estimate the easier component of $L_m$:
\begin{equation}\label{ro5}
  L_{m,1}=\big(A_{d+1}+\sum_{l=1}^dA_l^2\big)\psi_m-
i\sum_{l=1}^d\psi_l\Im(\overline{\psi_l}\psi_m).
\end{equation}

\begin{lemma}\label{easysch}
  Let $F$ satisfy \eqref{psa} and $\psi \in \{\psi_m, \ m = 1,
  \cdots,d\}$. Then
  \begin{equation}\label{ro6}
 \|P_k (\psi F)\|_{N_k(T)} \lesssim \eps_0 2^{-\sigma k}  b_k(\sigma).
  \end{equation}
\end{lemma}

\begin{proof}[Proof of Lemma \ref{easysch}] 
  We use a bilinear Littlewood-Paley decomposition
  \begin{equation*}
    \begin{split}
      P_k (F\psi)=  P_k (P_{<k-80} F P_{[k-4,k+4]} \psi) + 
\!\!\! \sum_{|k_1-k| \leq 4}^{k_2 <k-80}\!\!\! P_k (P_{k_1}F  P_{k_2} \psi)
      +\!\!\!\sum_{k_1,k_2\geq k-80}^{|k_1-k_2|\leq 90}\!\!\! P_k(P_{k_1}F 
      P_{k_2}\psi).
    \end{split}
  \end{equation*}
The first term is estimated using \eqref{Fhfl} and \eqref{psasob}. For the second by \eqref{Fhfh},
\[
\begin{split}
\| P_k(P_{k_1}F \,P_{k_2}\psi)\|_{N_k(T)} 
\lesssim  2^{k_2(d-2)/2} 2^{\frac{k_2-k}2} \|P_{k_1}F\|_{L^2_{t,x}} 
\| P_{k_2}\psi\|_{G_{k_2}(T)} 
 \lesssim  \eps_0 2^{\frac{k_2-k}2} 2^{-\sigma k}   
b_{k}(\sigma). \end{split}
\]
The summation with respect to $k_2 < k-80$ is straightforward.

Finally for the third term we use \eqref{Flgh}
 \[
\|P_k(P_{k_1}F \cdot P_{k_2} \psi)\|_{N_k(T)}  \lesssim 2^{\frac{k-k_2}6} 
2^{(d-2)k_2/2} \|P_{k_1}F\|_{L^2_{t,x}}  \|P_{k_2} \psi \|_{G_{k_2}(T)}.
\]

If $d \geq 3$, then using \eqref{psa} and that
$\sigma \geq \frac12$, the third sum is easily estimated. If $d=2$, one needs to distinguish between two cases. If $0 \leq \sigma \leq 1/12$ then we bound the right hand side by
\[
2^{\frac{k-k_2}6} \eps_0 2^{-\sigma k_2} b_{k_2}(\sigma)
\lesssim \eps_0 2^{-\sigma k} 2^{\frac{k-k_2}{12}}b_{k}(\sigma)
\]
and the summation with respect to $k_1,k_2 \geq k-80$ is straightforward.
If $\sigma \geq 1/12$  then we bound the right hand side by
\[
2^{\frac{k-k_2}6} 2^{-\sigma k} b_{>k}^2(\sigma)  
 b_{k_2} \lesssim 2^{\frac{k-k_2}6} 2^{-\sigma k} 
b_{k}(\sigma)  b_k b_{k_2}
\]
and the summation with respect to $k_1,k_2 \geq k-80$ is again straightforward.
\end{proof}

It remains to estimate the second part of $L_m$, namely
\begin{equation}\label{ro5.1}
  L_{m,2}=-2i\sum_{l=1}^dA_l\partial_l\psi_m-i\sum_{l=1}^d\partial_lA_l\cdot\psi_m=-i\sum_{l=1}^d\partial_l(A_l\psi_m)-i\sum_{l=1}^dA_l\partial_l\psi_m.
\end{equation}
For this we first complement Lemma~\ref{bilnk} with two $L^2$ bilinear
estimates:

\begin{lemma} \label{bilnka}
(a) If  $k_1 \leq k_2$, $f_1\in F_{k_1}(T)$, and $f_2\in F_{k_2}(T)$ then
  \begin{equation} \| f_1 \cdot  f_2 \|_{L^2_{t,x}} \lesssim
    2^{k_1(d-2)/2}  \| f_1 \|_{F_{k_1}(T)} \| f_2\|_{F_{k_2}(T)}.
\label{ff}  \end{equation}
(b)  If  $k_1 \leq k_2$, $f\in F_{k_1}(T)$, and $g\in G_{k_2}(T)$ then
  \begin{equation}  \| f\cdot g \|_{L^2_{t,x}} \lesssim
    2^{k_1(d-2)/2} 2^{(k_1-k_2)/2} \| f \|_{F_{k_1}(T)} \| g
    \|_{G_{k_2}(T)}.
\label{flowghigh}  \end{equation}
\end{lemma}

\begin{proof}[Proof of Lemma \ref{bilnka}]
Part (a) follows by Sobolev embeddings from
\[
\| f_1 \cdot  f_2 \|_{L^2_{t,x}} \lesssim  \| f_1 \|_{L^4_t L^{2d}_x} 
\| f_2\|_{L^\infty_t L^2_x \cap L^{p_d}_{t,x}}.
\]
For part (b), we first observe that, using a smooth partition of $1$
in the frequency space, we may assume that $\mathcal{F}(g)$ is
supported in the set 
\[
\{(\xi,\tau):|\xi|\in[2^{k_2-1},2^{k_2+1}]\text{
  and }\xi\cdot \e_0\geq 2^{k_2-5}\}
\]
 for some vector
$\e_0\in\mathbb{S}^{d-1}$.  Thus
  \begin{equation}
    \|g\|_{L^{\infty,2}_{\e_0,\lambda}}\lesssim 
2^{-k_2/2}\| g\|_{G_{k_2}(T)}\quad \text{ if }|\lambda|\leq 2^{k_2-40}.
 \label{pkgloc} \end{equation}
Then  in dimension $d \geq 3$ we have
\[
\| f\cdot g \|_{L^2_{t,x}} \lesssim \|f\|_{ L^{2,\infty}_\e} 
\|g \|_{L^{\infty,2}_\e} \lesssim
    2^{k_1(d-2)/2} 2^{(k_1-k_2)/2} \| f \|_{F_{k_1}(T)} \| g
    \|_{G_{k_2}(T)}.
\]
The argument is more involved if $d=2$. Given the definition
\eqref{fkdef2} of the $F_k(T)$ space in terms of $F_k^0(T)$, it suffices to
show that the following bounds hold for $F_k^0(T)$:
\begin{equation}
 \| f\cdot g \|_{L^2} \lesssim
 \| f \|_{F_{k_1}^0(T)} \| g
    \|_{G_{k_2}(T)}, \qquad k_1 \geq k_2 - 100
\label{kphigh}\end{equation}
respectively
\begin{equation}
 \| f\cdot g \|_{L^2} \lesssim 2^{(k_1-k_2)/2}
 \| f \|_{F_{k_1}^0(T)} \| g
    \|_{G_{k_2}(T)}, \qquad k_1 < k_2 - 100.
\label{kplow}\end{equation}
The bound \eqref{kphigh} follows easily by estimating both factors in
$L^4_{t,x}$. For \eqref{kplow}, on the other hand, we use the local smoothing/maximal function spaces. Precisely, for $g$ as in \eqref{pkgloc} we have
\[
\|  f\cdot g \|_{L^2_{t,x}} \lesssim \| f\|_{L^{2,\infty}_{\e_0,W_{k_1+40}}} \sup_{|\lambda| < 2^{k_2-40}}
\| P_{k_2} g\|_{L^{\infty,2}_{\e_0,\lambda}} \lesssim  2^{(k_1-k_2)/2} 
\| f \|_{F_{k_1}^0(T)} \| g
    \|_{G_{k_2}(T)},
\]
as desired.
 \end{proof}

 The bilinear estimates in Lemmas~\ref{bilnk},\ref{bilnka} allow us to
 obtain corresponding trilinear estimates.  We denote by
 $C(k,k_1,k_2,k_3)$ the best constant $C$ in the estimate
\begin{equation}
\begin{split}
  &\| P_{k} (P_{k_1} f_1 P_{k_2} f_2 P_{k_3} g) \|_{N_k(T)}\\
  &\lesssim   C
  2^{\frac{d-2}2 (k_1+k_2 +k_3-k)}          \| P_{k_1} f_1\|_{F_{k_1}(T)} 
  \| P_{k_2} f_2\|_{F_{k_2}(T)}    \| P_{k_3} g\|_{G_{k_3}(T)}.
  \end{split}
\label{tris}\end{equation}
Using the $L^\infty_t L^2_x \cap L^{p_d}_{t,x}$ norm for each of the three factors
and the $L^{p'_d}_{t,x}$ norm for the output, by Sobolev embeddings 
one can easily show that
\[
 C(k,k_1,k_2,k_3) \lesssim 1. 
\]
We seek to improve this with certain off-diagonal gains:

\begin{lemma}
The best constant $C=C(k,k_1,k_2,k_3)$  in \eqref{tris}
satisfies the following bounds:
\begin{equation}
C(k,k_1,k_2,k_3) \lesssim \left\{ \begin{array}{ll}  
2^{ (k_1+k_2)/2-k}   \qquad 
& 
k_1, k_2 \leq k -40
\cr
2^{-|k-k_3|/6}  
&
k, k_3 \leq k_1 -40
\cr
2^{-|\Delta k|/6}  
&
\text{otherwise}
\end{array} \right.
\label{ckkkk}\end{equation}
where $\Delta k = \max\{k,k_1,k_2,k_3\}-\min\{k,k_1,k_2,k_3\}$.
\end{lemma}
\begin{proof}
In the case $k_1, k_2 \leq k -40$ we must also have $|k_3-k| \leq 4$.
Then we successively apply \eqref{flowghigh} and \eqref{Fhfh}.

In the case $k, k_3 \leq k_1 -40$ we apply first \eqref{ff} and then 
conclude with \eqref{Flgh} if $k \leq k_3$, respectively \eqref{Fhfh}
if $k > k_3$.

In the remaining case we can assume without any restriction in generality 
that $k_1 \leq k_2$. Then there are two possibilities:

(i) $k_3 \leq  k$ and $|k-k_2| \leq 40$. If $ k_1 \leq k_3$ then we 
use \eqref{ff} for $P_{k_2} f_2 P_{k_3} g$ and then conclude with \eqref{Fhfh}.
If $ k_3<k_1$ then we use \eqref{ff} for $P_{k_1} f_1 P_{k_2} f_2$
and then conclude with \eqref{Fhfh}.

(ii) $k_3 > k$ and $|k_3-k_2| \leq 40$. If $ k_1 \leq k$ then we use
\eqref{flowghigh} for $P_{k_1} f_1 P_{k_3} g$ and then conclude using
only Strichartz norms.  If $k_{min} = k$ then we use \eqref{ff} for
$P_{k_1} f_1 P_{k_2} f_2$ and then conclude with \eqref{Flgh}.

\end{proof}

\begin{lemma}
The following estimate holds:
  \begin{equation}\label{ro6a}
 \|P_k L_{m,2}\|_{N_k(T)} \lesssim \eps_0 2^{-\sigma k}  b_k(\sigma).
  \end{equation}
\label{hardsch}\end{lemma}

\begin{proof}[Proof of Lemma \ref{hardsch}] 
To bound $L_{m,2}$ we use the representation \eqref{Aform} for the
connection coefficients $A_l$. Thus using the short notations $A$,
$\DD \psi$ and $\psi$ for $ A_m$, $\DD_m \psi_m$, and $\psi_m$ with $m =
1,d$ we can write
\[
\begin{split}
\| P_k \partial_x (A \psi)\|_{N_k(T)} \lesssim & \int_{0}^\infty 
\| \partial_x P_k(\psi(s) \DD \psi(s) \psi)\|_{N_k(T)}  
ds 
\\
\lesssim & \sum_{k_1,k_2,k_3} \int_{0}^\infty 
2^k \|P_k(P_{k_1} \psi(s) P_{k_2}(\DD \psi(s)) P_{k_3}\psi)\|_{N_k(T)} ds
\\
\lesssim & \!\! \sum_{k_1,k_2,k_3} \!\! 2^k   C_0
     \|  P_{k_3}\psi \|_{G_{k_3}(T)}
  \int_{0}^\infty \!\!  \|P_{k_1} \psi(s)\|_{F_{k_1}(T)} 
\|P_{k_2}(\DD \psi(s))\|_{F_{k_3}(T)} ds
\end{split}
\]
where $C_0 = C(k,k_1,k_2,k_3)2^{\frac{d-2}2 (k_1+k_2 +k_3-k)}$. For the term
$A \partial_x \psi$ we obtain a similar bound but with 
$2^k$ replaced by $2^{k_3}$ since $\|\partial_x P_{k_3}
\psi\|_{G_{k_3}(T)} \lesssim 2^{k_3} \|P_{k_3}
\psi\|_{G_{k_3}(T)}$.
Thus
\begin{equation*}
\begin{split}
\| P_k L_{m,2}\|_{N_k(T)} &\lesssim \sum_{k_1,k_2,k_3}2^{\max\{k,k_3\}}
 C_0
     \|  P_{k_3}\psi \|_{G_{k_3}(T)}\\
&\times\int_{0}^\infty \!\!  \|P_{k_1} \psi(s)\|_{F_{k_1}\!(T)} 
\|P_{k_2}(\DD \psi(s))\|_{F_{k_3}\!(T)} ds.
\end{split}
\end{equation*}
For the last two factors we use \eqref{ro1}. Thus we need to evaluate
the integrals
\[
I_{k_1k_2} = \int_0^\infty (1+s2^{2k_1})^{-4} 
2^{k_2}(s2^{2k_2})^{-3/8}   (1+s2^{2k_2})^{-2} ds \lesssim 2^{-\max\{k_1,k_2\}}.
\]
Taking this into account, from \eqref{ro1} and \eqref{bksbd}  we obtain
\[
\begin{split}
\| P_k L_{m,2}\|_{N_k(T)} \lesssim &\  2^{\sigma k} \sum_{k_1,k_2,k_3}
C(k,k_1,k_2,k_3) 2^{\max\{k,k_3\}- \max\{k_1,k_2\} }   
  b_{k_{min}} b_{k_{mid}}
b_{k_{max}}(\sigma)
\\ = &\ 2^{\sigma k} (S_1 + S_2 + S_3)
\end{split}
\]
where the sums $S_1$, $S_2$ and $S_3$ correspond to the three
cases in \eqref{ckkkk} and the indices $\{k_{min},k_{mid},k_{max}\}$
represent the increasing rearrangement of $\{k_1,k_2,k_3\}$.  Then
\[
S_1 \lesssim \sum_{k_1,k_2\leq k-4} 2^{-|k_1-k_2|/2}  b_{k_1} b_{k_2}
b_{k}(\sigma) \lesssim \eps_0 b_{k}(\sigma).
\]
In the second case we have full off-diagonal decay
\[
S_2 \lesssim \sum_{k_1 \geq k}^{k_3 \leq k_1} 2^{\max\{k,k_3\}-k_1} 2^{-|k-k_3|/6} 
b_{k_3} b_{k_1}
b_{k_1}(\sigma) \lesssim \eps_0 b_k(\sigma).   
\]
In the third case $\max\{k,k_3\}= \max\{k_1,k_2\}$ therefore we obtain
\[
S_3 \lesssim \sum_{k_1,k_2,k_3} 2^{-|\Delta k|/6}  b_{k_{min}} b_{k_{mid}}
b_{k_{max}}(\sigma) \lesssim \eps_0 b_k(\sigma).
\]
\end{proof}

\section{The main linear estimates}\label{SPACES}

In this section we prove Proposition~\ref{linearmainrep}. We use the
notation of section \ref{functionspaces}, see in particular the
definitions \ref{Ldef}, \ref{spacesd>2} and \ref{spacesd2}.  We define
two more classes of spaces, which are used only in this section. Given
a finite subset $W\subseteq \R$ and $r\in[1,\infty]$ we define the
spaces $\sum^r L^{p,q}_{\e,W}$ and $\bigcap^r L^{p,q}_{\e,W}$ using
the norms
\begin{equation}\label{sumspacesr}
  \| \phi\|^r_{\sum^r 
    L^{p,q}_{\e,W}} = |W|^{r-1}
  \inf_{\phi = \sum_{\lambda \in W} \phi_\lambda} \sum_{\lambda \in W} 
  \|\phi_\lambda\|_{L^{p,q}_{\e,\lambda}}^r
\end{equation}
and
\begin{equation}\label{intspacesr}
  \| \phi\|^r_{\bigcap^r 
    L^{p,q}_{\e,W}} = |W|^{-1} \sum_{\lambda \in W} 
  \|\phi\|_{L^{p,q}_{\e,\lambda}}^r.
\end{equation}
Clearly, $\sum^1 L^{p,q}_{\e,W}=L^{p,q}_{\e,W}$ (compare with
definition \ref{Ldef}) and
\begin{equation}\label{compr}
  \| \phi\|_{\sum^r L^{p,q}_{\e,W}} \leq \| \phi\|_{\sum^{r'} L^{p,q}_{\e,W}}\qquad\text{ if }r\leq r'.
\end{equation}

We first consider the homogeneous equation
\[
(i\partial_t+\Delta_x)u = 0, \qquad u(0) = f \in L^2(\R^d)
\]
which has the solution $u(t) = e^{it \Delta} f$.  For this we have the
following:

\begin{lemma}
  Assume $f \in L^2(\R^d)$, $k\in\mathbb{Z}$, and
  $\e\in\mathbb{S}^{d-1}$. We have:

  (i) Local smoothing estimate
  \begin{equation}\label{locsmobound}
    \sup_{|\lambda|\leq 2^{k-5}}\|e^{it\Delta}P_{k,\e}f\|_{L^{\infty,2}_{\e,\lambda}}\lesssim \|f\|_{L^2}.
  \end{equation}

  (ii) Maximal function estimates:
  \begin{equation}
    \|e^{it \Delta} P_k f\|_{L^{2,\infty}_\e} \lesssim  
    2^{k(d-1)/2} \|f\|_{L^2}, \qquad  d \geq 3,
    \label{latstc} \end{equation}
  and, for  any $\K\in\Z_+$,
  \begin{equation}
    \|1_{[-2^{2\K},2^{2\K}]}(t) 
    e^{it \Delta} P_{k} f\|_{\sum^2 L^{2,\infty}_{\e,W_{k+5}}} \lesssim  
    2^{k/2} \|f\|_{L^2}, \qquad d=2.
    \label{linnew}\end{equation}

  (iii) Strichartz-type estimates:
  \begin{equation}
    \|e^{it \Delta} f\|_{L^{p_d}_{x,t}} \lesssim \|f\|_{L^2},
    \label{linst}\end{equation}
  and
  \begin{equation}
    \|e^{it \Delta} P_k f\|_{L^{p_d}_x{ L^\infty_t}} 
    \lesssim 2^{kd/(d+2)} \|f\|_{L^2}.
    \label{linmax}\end{equation}
  If $d=2$ and $(p,q)\in(2,\infty]\times[2,\infty]$, $1/p+1/q=1/2$,
  then
  \begin{equation}
    \|e^{it \Delta} P_{k,\e} f\|_{L^{p,q}_{\e}} \lesssim  
    2^{k(2/p-1/2)} \|f\|_{L^2}, \qquad p \geq q,
    \label{latsta}\end{equation}
  and
  \begin{equation}
    \|e^{it \Delta} P_k f\|_{L^{p,q}_\e} \lesssim_p  
    2^{k(2/p-1/2)} \|f\|_{L^2}, \qquad p \leq q.
    \label{latstb} \end{equation}
  \label{linhom}\end{lemma}

\begin{proof}[Proof of Lemma \ref{linhom}] The Strichartz estimate
  \eqref{linst} is well-known, see \cite{Str}. As a consequence we
  obtain
  \[
  \|e^{it \Delta} P_k f\|_{L^{p_d}_t L^{p_d}_x} + 2^{-2k}
  \| \partial_t e^{it \Delta} P_k f\|_{L^{p_d}_t L^{p_d}_x} \lesssim
  \|f\|_{L^2}.
  \]
  The maximal Strichartz estimate \eqref{linmax} follows. The local
  smoothing estimate \eqref{locsmobound} is proved, for example, in
  \cite{IoKe2} and \cite{IoKe3} for $\lambda=0$. The general case
  follows using the Galilean transformation $T_{\lambda\e}$, see the
  definition \eqref{operatorT}. The maximal function estimate
  \eqref{latstc} is known, see for example \cite{IoKe2} or
  \cite{IoKe3}. The lateral Strichartz estimate \eqref{latsta} follows
  by interpolation between \eqref{locsmobound} and \eqref{linst}. The
  lateral Strichartz estimate \eqref{latstb} follows using a standard
  $TT^\ast$ argument and the Hardy-Littlewood-Sobolev inequality. We
  prove below the maximal function estimate \eqref{linnew}, which
  represents our main new contribution to the linear theory.  \medskip

  {\bf{Proof of \eqref{linnew}.}} We fix $\e\in\mathbb{S}^{1}$.  By
  rescaling we can assume that $\K = 0$. We may also assume that
  $k\geq 1$, since for $k\leq 0$ one has the stronger bound
  \begin{equation*}
    \|\mathbf{1}_{[-1,1]}(t)e^{it \Delta} P_k f\|_{L^2_xL^\infty_t}\lesssim \|f\|_{L^2}.
  \end{equation*}
  We need to show that
  \begin{equation}\label{keymax}
    \|1_{[-1,1]}(t) e^{it \Delta} P_k f\|_{\sum^2 L^{2,\infty}_{\e,W_{k+5}}} 
    \lesssim 2^{k/2} \|f\|_{L^2}. 
  \end{equation}
   We show first that if $\|g\|_{\bigcap^2L^{2,1}_{\e,W_{k+5}}}\leq 1$ then
  \begin{equation}\label{keymax2}
    \Big|\int_{\R^2\times\R}\overline{g(x,t)}\mathbf{1}_{[-1,1]}(t)(e^{it \Delta} P_k f)(x,t)\,dxdt\Big|\lesssim 2^{k/2} \|f\|_{L^2}.
  \end{equation}
This can be rewritten as 
\[
 \Big|\int_{\R^2\times\R}\overline{(e^{-it \Delta} P_k g(t))(x)}\mathbf{1}_{[-1,1]}(t) 
f(x)\,dtdx\Big|\lesssim 2^{k/2} \|f\|_{L^2}.
\]
or equivalently
\[
\left\| \int_{-1}^1 e^{-it \Delta} P_k g(t)\right\|_{L^2} \lesssim 2^{k/2}
\]
Hence it suffices to show that
  \begin{equation}\label{keymax3}
    \Big|\int_{\R^2\times\R}\int_{\R^2\times\R}g(x,t)\mathbf{1}_{[-1,1]}(t)\overline{g}(y,s)\mathbf{1}_{[-1,1]}(s)K_k(x-y,t-s)\,dxdtdyds\Big|\lesssim 2^{k}
  \end{equation}
  where
  \begin{equation}\label{K_kdef}
    K_k(x,t)=\int_{\R^2}e^{ix\cdot\xi}e^{-it|\xi|^2}\chi_k(|\xi|)^2\,d\xi.
  \end{equation}
  By stationary phase
  \[
  |K_k(t,x)| \lesssim \left\{ \begin{array}{cc} 2^{2k}(1+2^{2k}
      |t|)^{-1} & |x| \leq 2^{k+4} |t|; \cr \cr 2^{2k}(1+2^{k}
      |x|)^{-N} & |x| \geq 2^{k+4} |t|.
    \end{array}\right.
  \]
  The key idea is to foliate $K_k$ in the $\e$ direction with respect
  to (thickened) rays with speed less than $2^{k+5}$. We observe that
  for $t\in[-2,2]$
  \[
  |K_k(t,x)| \lesssim \sum_{\lambda \in W_{k+5}} K_{k,\lambda}(t,x),
  \qquad K_{k,\lambda}(t,x) = (1+2^k|x\cdot \e -\lambda t|)^{-N}.
  \]
  Hence the left hand side of \eqref{keymax3} can be bounded by
  \[
  \begin{split}
    \sum_{\lambda \in W_{k+5}} &\int_{-1}^1\int_{-1}^1
    K_{k,\lambda}(t-s,x-y) |g(y,s)| |g(x,t)| dx dy ds dt
    \\
    \lesssim &\ \sum_{\lambda \in W_{k+5}}
    \|K_{k,\lambda}\|_{L^{1,\infty}_{\e,\lambda}}\|g\|_{L^{2,1}_{\e,\lambda}}
    \|g\|_{L^{2,1}_{\e,\lambda}}\lesssim 2^{-k} \sum_{\lambda \in
      W_{k+5}} \|g\|_{L^{2,1}_{\e,\lambda}}^2 \lesssim
    2^{k}\|g\|_{\bigcap^2 L^{2,1}_{\e,W_{k+5}}}^2,
  \end{split}
  \]
  where we used the fact that $|W_{k+5}| \approx 2^{2k}$.  Thus
  \eqref{keymax2} follows.

  Formally the main bound \eqref{keymax} would follow from \eqref{keymax2} and the
  duality relation
  \[
  \left({\bigcap}^2 L^{2,1}_{\e,W_{k+5}}\right)' = {\sum}^2
  L^{2,\infty}_{\e,W_{k+5}}.
  \]
  However this duality relation is not entirely straightforward so we provide
a direct argument. Let $\chi(t)$ be a function which has the following properties:

a) $\mathcal F\chi$ is smooth and with compact support.

b) $\chi(t) \neq 0$ for $t \in [-1,1]$.
 
Then $\chi$ is rapidly decreasing at infinity, therefore from \eqref{keymax2}
it follows that
 \begin{equation}\label{keymax2chi}
    \Big|\int_{\R^2\times\R}\overline{g(x,t)}\chi(t)
(e^{it \Delta} P_k f)(x,t)\,dxdt\Big|\lesssim 2^{k/2} \|f\|_{L^2}
  \end{equation}
whenever  $\|g\|_{\bigcap^2L^{2,1}_{\e,W_{k+5}}}\leq 1$. We will use this to prove a stronger form of \eqref{keymax}, namely
\begin{equation}\label{keymaxchi}
    \|\chi(t) e^{it \Delta} P_k f\|_{\sum^2 L^{2,\infty}_{\e,W_{k+5}}} 
    \lesssim 2^{k/2} \|f\|_{L^2}. 
\end{equation}
The advantage is that the function $u = \chi(t) e^{it \Delta} P_k f$
has a compactly supported Fourier transform. To obtain \eqref{keymaxchi}
from \eqref{keymax2chi} it suffices to show that
\begin{equation}
\| u\|_{ \sum^2 L^{2,\infty}_{\e,W_{k+5}}} \lesssim 
\sup\ \left\{  \int_{\R^2 \times \R} u(x,t) 
\overline v(x,t) dx dt;\  \|g\|_{\bigcap^2L^{2,1}_{\e,W_{k+5}}}\leq 1 \right\}
\label{duala}\end{equation}

Indeed, suppose that the space-time Fourier transform $\mathcal F u$
is supported inside a ball $B$. We define the normed space
  \begin{equation*}
    X_B=
    \{h\in L^2(\R^2\times\R):\FF(h)\text{ supported in }B 
\text{ and }\|h\|_{A}=\|h\|_{\Sigma^2L^{2,\infty}_{\e,W_{k+5}}}<\infty\}.
  \end{equation*} 
We also use a larger ball $2B$ and the corresponding space $X_{2B}$.

Since $ u\in X_B$, by the Hahn-Banach theorem there is
$\Lambda\in X_B^\ast$ so that $\|u\|_{X_B}=\Lambda(u)$ and
  $\|\Lambda\|_{X_B^\ast}=1$. By the Hahn-Banach theorem we can extend 
$\Lambda$ to $X_{2B}$. On the other hand, for any $h\in X_{2B}$
  \begin{equation*}
    \|h\|_{X_{2B}} \lesssim_B \|h\|_{L^2}.
  \end{equation*}
  Thus, using the Hahn-Banach theorem again, there is a linear
  functional $\Lambda':L^2\to\mathbb{C}$ such that
  $\Lambda'(h)=\Lambda(h)$ for any $h\in A \cap L^2$ and $|\Lambda'(h)|\lesssim_B
 \|h\|_{L^2}$ for any $h\in L^2$. Therefore there is $g\in
  L^2(\R^d\times\R)$ with the property that
  \begin{equation*}
    \Lambda(h)=\int_{\R^2\times\R}\bar{g} 
\cdot h\,dxdt\text{ for any }h\in X_{2B} \cap L^2.
  \end{equation*}
  We consider a space-time multiplier $\chi_{B}(D)$ whose symbol is
  smooth, supported in $2B$ and equals $1$ in $B$. By the choice of $\Lambda$
we have
  \begin{equation*}
    \|u\|_{\sum^2 L^{2,\infty}_{\e,W_{k+5}}} =\Lambda(u)=
    \int_{\R^2\times\R} \bar{g} \cdot u\,dxdt =\int_{\R^2\times\R} 
\overline{\chi_{B}(D)g} \cdot u\,dxdt  
  \end{equation*}
Hence  for \eqref{duala} it suffices to prove
  that
  \begin{equation}\label{keymax8}
    \|\chi_B(D){g}\|_{\bigcap^2L^{2,1}_{\e,W_{k+5}}}\lesssim 1.
  \end{equation}
  Since $\chi_B(D)$ is a bounded operator on $\sum^2
  L^{2,\infty}_{\e,W_{k+5}}$ and $\Lambda$ is bounded in $X_{2B}$ we
  have
\[
\left|\int_{\R^2 \times \R} \overline{\chi_B(D) g} h dx dt \right|
=
\left|\int_{\R^2 \times \R} \overline{ g} \chi_B(D) h dx dt \right|
\lesssim
\|h\|_{\sum^2 L^{2,\infty}_{\e,W_{k+5}}}
\]
for all $h \in \sum^2 L^{2,\infty}_{\e,W_{k+5}} \cap L^2$. In view of the
duality relations $(L^{2,1}_{\e,\lambda})' =  L^{2,\infty}_{\e,\lambda}$
we can optimize the choice of $h$ above to obtain \eqref{keymax8}.
In order to guarantee that $h \in L^2$ one can carry out this 
analysis first in a compact set, and then expand it to infinity.

\end{proof}

We return to the proof of Proposition~\ref{linearmainrep}, which we
restate here for convenience.

\begin{proposition}\label{linearmain}
  (Main linear estimate) Assume $\K\in\Z_+$, $T\in(0,2^{2\K}]$, and
  $k\in\Z$. Then for any $u_0\in L^2$ localized at frequency $2^k$ and
  $h\in N_k(T)$, the solution $u$ to
  \begin{equation}\label{lins}
    (i\partial_t+\Delta_x)u=h,\qquad u(0)=u_0
  \end{equation}
  satisfies
  \begin{equation}\label{linearmainestrep}
    \|u\|_{G_k(T)}\lesssim \|u_0\|_{L^2}+\|h\|_{N_k(T)}.
  \end{equation}
\end{proposition}

The proposition follows from Lemma \ref{alexnew1} and Lemma
\ref{alexnew2} below. If $d=2$ we define $\widetilde{G}_k(T)$ as the
normed spaces of functions in $L^2_k(T)$ for which the norm
\begin{equation*}
  \begin{split}
    \|\phi\|_{\widetilde{G}_k(T)}&=\|\phi\|_{G_k(T)}+2^{-k/2}\sup_{\e\in S^1}\|\phi\|_{\Sigma^2L^{2,\infty}_{\e,W_{k+35}}}+2^{k/6}   \sup_{|j-k|  \leq 25} \sup_{\e \in \mathbb S^{1}}\|P_{j,\e} \phi\|_{L^{6,3}_{\e}}\\
    &+2^{k/2} \sup_{|j-k| \leq 25} \sup_{\e \in \mathbb
      S^{1}}\sup_{|\lambda|\leq
      2^{k-35}}\|P_{j,\e}\phi\|_{L^{\infty,2}_{\e,\lambda}}
  \end{split}
\end{equation*}
is finite. In other words, we replace the norm
$L^{2,\infty}_{\e,W_{k+40}}=\Sigma^1 L^{2,\infty}_{\e,W_{k+40}}$ in
\eqref{fkodef2} with the stronger norm $\Sigma^2
L^{2,\infty}_{\e,W_{k+35}}$, compare with Definition \ref{spacesd2}
and \eqref{compr}, and readjust slightly the ranges of $j$ and
$\lambda$. Similarly, if $d\geq 3$ let
\begin{equation*}
  \|\phi\|_{\widetilde{G}_k(T)}=\|\phi\|_{G_k(T)}+2^{k/2} \sup_{|j-k|  \leq 25} \sup_{\e \in \mathbb S^{d-1}}\|P_{j,\e} \phi \|_{L^{\infty,2}_{\e}}.
\end{equation*}
It is easy to check from the definition that
\begin{equation}\label{bo9}
  \|v\|^2_{\widetilde{G}_k(T)}\leq\sum_{i=0}^{m-1}\|\mathbf{1}_{[t_i,t_{i+1})}(t)\cdot v\|^2_{\widetilde{G}_k(T)}
\end{equation}
for any partition $\{-T=t_0<\ldots<t_m=T\}$ of the interval $[-T,T]$
and any $v\in \widetilde{G}_k(T)$. This property does not hold for the
spaces $G_k(T)$ if $d=2$.

The solution $u$ for \eqref{lins} can be represented as
\[
u(t) = e^{it \Delta } u_0 + \int_{0}^t e^{i(t-s) \Delta } h(s) ds.
\]
As a consequence of the Lemma \ref{linhom}, we immediately obtain
\eqref{linearmainestrep} for the first term. More precisely, for any
$f\in L^2(\R^d)$ localized at frequency $2^k$ we have
\begin{equation}\label{nghom}
  \|e^{it\Delta}f\|_{\widetilde{G}_k(T)}\lesssim \|f\|_{L^2}.
\end{equation}

It remains to make the transition to the full inhomogeneous
problem. We divide the $N_k(T)$ space into two components, $N_k(T) =
N_k^0(T) + N_k^1(T)$ with norms
\begin{equation}
  \begin{split}
    \|f\|_{N_k^0(T)} = &\ \inf_{f=f_1+f_2+f_3}
    \|f_1\|_{L^{\frac{4}{3}}} +
    2^{\frac{k}6}\|f_2\|_{L^{\frac32,\frac65}_{\e_1}} +
    2^{\frac{k}6}\|f_3\|_{L^{\frac32,\frac65}_{\e_2}}
    \\
    \|f\|_{N_k^1(T)}= &\ 2^{-k/2} \sup_{\e \in \mathbb S^{1}} \|
    f\|_{L^{1,2}_{\e,W_{k-40}}}
  \end{split}
  \label{nk1def2}\end{equation}
in dimension $d=2$, and
\begin{equation}\label{bo20}
  \|f\|_{N_k^0(T)}=\|f\|_{L^{p'_d}},\qquad\|f\|_{N_k^1(T)}=2^{-k/2}\sup_{\e \in \mathbb S^{d-1}}\|f\|_{L^{1,2}_\e}
\end{equation}
in dimensions $d\geq 3$. The spaces $N^0_k(T)$ have the property that
if $f\in N^0_k(T)$ and $-T=t_0<t_1<\ldots<t_m=T$ is a partition of the
interval $[-T,T]$ then
\begin{equation}\label{bo1}
  \sum_{i=0}^{m-1}\|f\cdot 1_{[t_{i},t_{i+1})}(t)\|_{N^0_k(T)}^{p_1}\leq\|f\|_{N^0_k(T)}^{p_1}\text{ for some }p_1<2.
\end{equation}
This property is easy to verify for $p_1=p'_d$ if $d\geq 3$ and
$p_1=3/2$ if $d=2$. The spaces $N^1_k(T)$ do not have this
property. From Lemma \ref{linhom}, by duality, we obtain the energy
bound
\begin{equation}
  \left\|\int_{0}^t e^{i(t-s) \Delta } h(s) ds\right\|_{L^\infty_t L^2_x}\lesssim \|h\|_{N^0_k(T)}\quad\text{ for any }h \in N^0_k(T). 
  \label{nginhom}\end{equation}

\begin{lemma}\label{alexnew1}
  Assume that $u\in L^2_k(T)$ satisfies
  \begin{equation*}
    (i\partial_t+\Delta_x)u=h\text{ on }\R^d\times[-T,T],\quad u(0)=0.
  \end{equation*}
  Then
  \begin{equation*}
    \|u\|_{\widetilde{G}_k(T)}\lesssim \|h\|_{N_k^0(T)}.
  \end{equation*}
\end{lemma}

\begin{proof}[Proof of Lemma \ref{alexnew1}] This lemma is an abstract
  consequence of the homogeneous bound \eqref{nghom}, the energy bound
  \eqref{nginhom}, and the bounds \eqref{bo9} and \eqref{bo1}.

  The simplest way to prove this lemma is by using the $U^p_\Delta$
  and $V^p_\Delta$ spaces associated to the Schr\"odinger
  evolution. These spaces have been first introduced in unpublished
  work of the last author, as  substitutes for Bourgain's $X^{s,b}$
  spaces which are better suited for the study of dispersive
  evolutions in critical Sobolev spaces. They have been successfully
  used for instance in \cite{KoTa}, \cite{KoTa1}, \cite{BeTa},
  \cite{HeKo}.

  For convenience we recall their definition. $V^p_\Delta$ is the
  space of right continuous $L^2$ valued functions with bounded
  $p$-variation along the Schr\"odinger flow,
  \[
  \|u\|_{V^p_\Delta}^p = \|u\|_{L^\infty_t L^2_x}^p + \sup_{t_k
    \nearrow} \sum_{k \in \Z} \|u(t_{k+1} -e^{i(t_{k+1} -t_k) \Delta}
  u(t_k)\|_{L^2}^p
  \]
  where the supremum is taken over all increasing sequences $t_k$. On
  the other hand $U^p_\Delta$ is the atomic space generated by a
  family $\mathcal A_p$ of atoms $a$ which have the form
  \[
  a(t) = e^{it\Delta} \sum_{k} 1_{[t_k,t_k+1)} u_k,\qquad\sum_k
  \|u_k\|_{L^2}^p \leq 1,
  \]
  where the sequence $\{t_k\}$ is increasing. Precisely, we have
  \[
  U^p_\Delta = \{ u = \sum c_k a_k: \ \sum_k |c_k| < \infty, \ \ a_k
  \in \mathcal A_p \}.
  \]
  The above sum converges uniformly in $L^2$; it also converges in the
  stronger $V^p_\Delta$ topology. The $U_\Delta^p$ norm is defined by
  \[
  \| u\|_{U^p_\Delta} = \inf \{ \sum_k |c_k|: \ u = \sum c_k a_k, \
  a_k \in \mathcal A_p \}.
  \]
  
  These spaces are related as follows:
  \begin{equation}
    U^p_\Delta \subset V^p_\Delta \subset U^q_\Delta, \qquad 1 \leq p < q < \infty.
    \label{upemb}\end{equation}
  The first inclusion is straightforward, but the second is not. As a
  consequence of the bounds \eqref{nghom} and \eqref{bo9} we have
  \begin{equation}
    \| u\|_{\widetilde{G}_k(T)} \lesssim \|u\|_{U^2_\Delta}
    \label{nghoma}\end{equation}
  for all $u \in U^2_\Delta$ localized at frequency $2^k$. On the
  other hand, as a consequence of the bounds \eqref{nginhom} and
  \eqref{bo1} we have
  \begin{equation}
    \left\| u\right\|_{V^{p_1}_\Delta} 
    \lesssim \|h\|_{N^0_k}, \qquad u(t) = \int_{0}^t e^{i(t-s) \Delta } h(s) ds.
    \label{downemb}\end{equation}
  The lemma follows using \eqref{upemb} and $p_1<2$.

  One could also provide a self-contained proof which does not use the
  spaces $U^p_\Delta$ and $V^p_\Delta$, in the spirit of the
  Christ-Kiselev lemma \cite{ChKi}, or of the proof of the second inclusion 
in \eqref{upemb} (see \cite[Lemma 6.4]{KoTa}).

\end{proof}

We estimate now the contribution of terms in $N_k^1(T)$.

\begin{lemma}\label{alexnew2}
  Assume that $v\in L^2_k(T)$ satisfies
  \begin{equation*}
    (i\partial_t+\Delta_x)u=h\text{ on }\R^d\times[-T,T],\quad u(0)=0.
  \end{equation*}
  Then
  \begin{equation*}
    \|u\|_{G_k(T)}\lesssim \|h\|_{N_k^1(T)}.
  \end{equation*}
\end{lemma}

\begin{proof}[Proof of Lemma \ref{alexnew2}] The spaces $N_k^1(T)$ do
  not satisfy \eqref{bo1}, so the general argument given in Lemma
  \ref{alexnew1} does not apply. We give a direct argument starting
  from the definition of the spaces $L^{1,2}_{\e,\lambda}$.  Using a
  smooth angular partition of unity in frequency we can assume without
  any restriction in generality that $h$ is frequency localized to a
  region $\xi \cdot \e\in[2^{k-2},2^{k+2}]$ for some $\e \in \mathbb
  S^{d-1}$. After a Galilean transformation, we may assume that $h \in
  L^{1,2}_{\e}$ and it suffices to prove the stronger bound
  \begin{equation}
    \| P_k P_{j,\e} u\|_{\widetilde{G}_k(T)} \lesssim 2^{-k/2} \|h\|_{L^{1,2}_{\e}},
    \qquad |j-k| \leq 4.
    \label{lud}\end{equation}
  Suppose $\e = \e_1$.  The solution $u$ can be expressed via the
  fundamental solution $K_0$ for the Schr\"odinger equation as
  \[
  \begin{split}
    u(t,x) = &\ \int_{s < t} \int_{\R^d} K_0(t-s,x-y) h(s,y) dy ds \\
    = &\ \int_{s < t} \int_{\R^d} K_0(t-s,x_1-y_1,x'-y') h(s,y_1,y')
    dy_1dy'ds \\ &\ = \int_\R u_{y_1}(t,x) d y_1
  \end{split}
  \]
  where
  \[
  u_{y_1}(t,x) = \int_{s < t} \int_{\R^{d-1}} K_0(t-s,x_1-y_1,x'-y')
  h(s,y_1,y')dy'ds.
  \]
  It suffices to show that
  \begin{equation}
    \| P_k P_{j,\e_1} u_{y_1}\|_{\widetilde{G}_k(T)} \lesssim 2^{-k/2} \|h(y_1)\|_{L^2}.
    \label{fin}\end{equation}
  This is a consequence of the following:

\begin{lemma} Suppose that $|j-k| \leq
  4$. Then the function $ P_k P_{j,\e_1} u_{y_1}$ can be represented
  as
  \begin{equation}
    \begin{split}
      &P_k P_{j,\e_1} u_{y_1}(t,x) = (P_{<k-40,\e_1} 1_{\{x_1 > y_1\}})\cdot P_k P_{j,\e_1}  e^{it\Delta}  v_0+ w(t,x);\\
      &\|v_0\|_{L^2} + 2^{-k} (\| \Delta w\|_{L^2} + \|\partial_t
      w\|_{L^2}) \lesssim 2^{-k/2} \|h(y_1)\|_{L^2}.
    \end{split}
    \label{fina}\end{equation}
\label{sidel}\end{lemma}
  To see that this implies \eqref{fin}, notice that, since $w$ is also
  localized at frequency $2^k$, the bound for
  $\|w\|_{\widetilde{G}_k(T)}$ is obtained from Sobolev embeddings. We
  observe also that the function $P_{<k-40,\e_1} 1_{\{x_1 < y_1\}}$ is
  bounded while $ v=P_k P_{j,\e_1} e^{it\Delta} v_0$ is an $L^2$
  solution for the homogeneous equation which is localized at
  frequency $2^k$.  By Lemma~\ref{linhom} this allows us to directly
  estimate all components of its $\widetilde{G}_k(T)$ norm except for
  the $L^{\infty,2}_{\e,\lambda}$ and the $L^{6,3}_{\e}$ bounds if
  $d=2$, respectively the $L^{\infty,2}_{\e}$ bound if $d\geq 3$. For
  these we need an additional step: if $|k_1-k| \leq 25$ and $\e \in
  \mathcal S^{d-1}$ we take advantage of the frequency localization of
  the above cutoff function in order to write
  \[
  P_{k_1,\e} [(P_{<k-40,\e_1} 1_{\{x_1 > y_1\}})\cdot v] =
  \sum_{|k_2-k| \leq 30} P_{k_1,\e} [(P_{<k-40,\e_1} 1_{\{x_1 >
    y_1\}}) \cdot P_{k_2,\e}v].
  \]
  This suffices to prove \eqref{fin}, in view of Lemma
  \ref{linhom}. It remains to prove the last lemma.
  \medskip

  \begin{proof}[Proof of Lemma~\ref{sidel}]  By translation invariance we can set
  $y_1=0$ and drop the parameter $y_1$ from the notations. The Fourier
  transform of $ P_k P_{j,\e_1} u$ is
  \[
  (\FF(P_k P_{j,\e_1} u))(\tau,\xi) = \frac{\chi_k(|\xi|)
    \chi_{j}(\xi_1)}{-\tau-|\xi|^2+i0} \hat h(\tau,\xi').
  \]
  On the other hand the Fourier transform of $v = P_k P_{j,\e_1}
  e^{it\Delta} v_0$ equals
  \[
  (\FF v)(\tau,\xi) =\chi_k(|\xi|) \chi_{j}(\xi_1) v_0(\xi)
  \delta_{\tau+|\xi|^2}.
  \]
  Assuming that $\hat v$ is supported in the region $\{|\xi| \approx
  2^k,\ \xi_1 \approx 2^j\}$, after truncation we obtain
  \[
  (\FF ( (P_{<k-40,\e_1} 1_{\{x_1 > y_1\}}) v))(\tau,\xi) = (
  \chi_k(|\xi|) \chi_{j}(\xi_1) v_0(\xi) \delta_{\tau+|\xi|^2}) \ast
  \frac{\chi_{<k-40}(\xi_1)}{\xi_1-i0}.
  \]
  On the $\{\xi_1 > 0\}$ part of the paraboloid we can write
  \[
  \delta_{\tau+|\xi|^2} = \frac{1}{2\sqrt{- \tau-|\xi'|^2}}
  \delta_{\xi_1 -\sqrt{-\tau-|\xi'|^2}}.
  \]
  Hence, with the notation $\txi = \sqrt{-\tau-|\xi'|^2}$, the above
  convolution gives
  \[
  \begin{split}
    (\FF ( (P_{<k-40,\e_1} &1_{\{x_1 > y_1\}}) v))(\tau,\xi) = \
    \chi_k(|(\txi,\xi')|) \chi_{j}(\xi_1) v(\txi,\xi') \frac{1}{2\txi}
    \frac{\chi_{<k-40}(\xi_1-\txi)}{\xi_1 -\txi -i0} \\ = &\ -
    \chi_k(|(\txi,\xi')|) \chi_{j}(\xi_1) v(\txi,\xi') \frac{\xi_1
      +\txi}{2\txi} \frac{\chi_{<k-40}(\xi_1-\txi)}{-\tau-|\xi|^2
      +i0}.
  \end{split}
  \]
  Matching the two expressions on the paraboloid $\tau + |\xi|^2 = 0$
  we see that it is natural to choose $v_0$ of the form
  \[
  v_0(\xi_1,\xi') = \left\{ \begin{array}{cc} \hat h(|\xi|^2,\xi') &
      \xi \in \text{supp} \ \chi_k(|\xi|) \chi_{j}(\xi_1) \cr \cr 0 &
      \text{otherwise}. \end{array} \right.
  \]
  Changing variables we obtain the required bound for $v$, namely
  \[
  \| v_0\|_{L^2} \lesssim 2^{-\frac{k}2} \|h\|_{L^2}.
  \]

  It remains to consider $w$, whose Fourier transform is obtained by
  taking the difference, $\hat w(\tau,\xi) = \hat h(\tau,\xi')
  q(\tau,\xi)$ where
  \[
  \begin{split}
    q(\tau,\xi) = &\ \left( \chi_k(|\xi|) \chi_{j}(\xi_1) -
      \chi_{k}(\txi,\xi') \chi_{j} (\txi) \chi_{<k-40}(\xi_1-\txi)
      \frac{\xi_1 +\txi}{2\txi} \right) \frac{1}{-\tau-|\xi|^2 +i0}.
  \end{split}
  \]
  The expression in the bracket is supported in $\{\xi_1 \approx 2^k,
  |\xi'| \lesssim 2^k\}$ and vanishes on the paraboloid $\{
  \tau+|\xi|^2 = 0\}$, canceling the singularity. Thus we obtain
  \[
  |q(\tau,\xi)| \lesssim (|\tau|+|\xi|^2)^{-1}.
  \]
  The bounds for $w$ follow.
\end{proof}
\end{proof}

\section{Proof of Proposition \ref{TaoHeat}}\label{HEATPROOF}

In this section we prove Proposition \ref{TaoHeat}. Recall that
the frequency envelopes $\gamma_k(\sigma)$ and $\gamma_k$ 
have the properties
\begin{equation}\label{la2}
  \begin{split}
    &\sum_{k\in\Z}\gamma_k^2\approx \gamma^2\ll 1;\\
    &\|P_k\phi\|_{L^\infty_tL^2_x}\leq 2^{-\sigma
      k}\gamma_k(\sigma)\qquad\text{ for any
    }\sigma\in[0,\infty)\text{ and }k\in\Z.
  \end{split}
\end{equation}

We start with two technical lemmas.

\begin{lemma}\label{basicest}
  Assume $f,g\in L^\infty(\R^d\times(-T,T):\C)$, $P_kf,P_kg\in
  L^\infty_tL^2_x$ and define, for any $k\in\Z$,
  \begin{equation}\label{bek}
    \al_k=\sum_{|k'-k|\leq 20}2^{dk'/2}\|P_{k'}f\|_{L^\infty_tL^2_x},\quad\be_k=\sum_{|k'-k|\leq 20}2^{dk'/2}\|P_{k'}g\|_{L^\infty_tL^2_x}.
  \end{equation}
  Then, for any $k\in\Z$,
  \begin{equation}\label{basicest18}
    2^{dk/2}\|P_k(fg)\|_{L^\infty_tL^2_x}\lesssim 
\al_k\sum_{k'\leq k}\be_{k'}+\be_k\sum_{k'\leq k}\al_{k'}+\sum_{k'\geq k}2^{-d|k'-k|}\al_{k'}\be_{k'}.
  \end{equation}
  In addition, if $\|g\|_{L^{\infty}_{x,t}}\leq 1$ then, for any
  $k\in\Z$
  \begin{equation}\label{basicest5.1}
    2^{dk/2}\|P_k(fg)\|_{L^\infty_tL^2_x}\lesssim \al_k+\be_k\sum_{k'\leq k}\al_{k'}+\sum_{k'\geq k}2^{-d|k'-k|}\al_{k'}\be_{k'}. 
  \end{equation}
  Finally, if $\|f\|_{L^\infty_{x,t}}+\|g\|_{L^{\infty}_{x,t}}\leq 1$
  then, for any $k\in\Z$
  \begin{equation}\label{basicest5.2}
    2^{dk/2}\|P_k(fg)\|_{L^\infty_tL^2_x}\lesssim \al_k+\be_k+\sum_{k'\geq k}2^{-d|k'-k|}\al_{k'}\be_{k'}. 
  \end{equation}
\end{lemma}

\begin{proof}[Proof of Lemma \ref{basicest}] We decompose the
  expressions to be estimated into
  $\mathrm{Low}\times\mathrm{High}\to\mathrm{High}$ frequency
  interactions and $\mathrm{High}\times\mathrm{High}\to\mathrm{Low}$
  frequency interactions. We estimate
  $\mathrm{Low}\times\mathrm{High}\to\mathrm{High}$ frequency
  interactions using the $L^\infty_{x,t}$ norm for the low frequency
  factor and the $L^\infty_tL^2_x$ norm for the high frequency factor.
  We estimate $\mathrm{High}\times\mathrm{High}\to\mathrm{Low}$
  frequency interactions using the $L^\infty_tL^2_x$ norms for both
  factors, and Sobolev embedding.

  For example, for \eqref{basicest18} we estimate  
  \begin{equation*}
    \begin{split}
 \|P_k (fg)\|_{L^\infty L^2} \lesssim&
\sum_{|k_2-k|\leq 4}\|P_k(P_{\leq k-5}f\cdot P_{k_2}g)\|_{L^\infty_tL^2_x}
+\negmedspace\sum_{|k_1-k|\leq 4}
\|P_k(P_{k_1}f\cdot P_{\leq  k-5}g)\|_{L^\infty_tL^2_x}
\\
&+\sum_{k_1,k_2\geq k-4,|k_1-k_2|\leq 8}\|P_k(P_{k_1}f\cdot P_{k_2}g)\|_{L^\infty_tL^2_x}\\
      \lesssim &  \sum_{|k_2-k|\leq 4}\|P_{\leq k-5}f\|_{L^\infty_{x,t}}\|P_{k_2}g\|_{L^\infty_tL^2_x}+\sum_{|k_1-k|\leq 4}\|P_{k_1}f\|_{L^\infty_tL^2_x}\|P_{\leq k-5}g\|_{L^\infty_{x,t}}\\
      &+ \sum_{k_1,k_2\geq k-4,|k_1-k_2|\leq 8}2^{dk/2}\|P_{k_1}f\cdot P_{k_2}g\|_{L^\infty_tL^1_x}\\
      \lesssim & \sum_{k_1\leq
        k}\alpha_{k_1}2^{-dk/2}\beta_k+\sum_{k_2\leq
        k}\beta_{k_2}2^{-dk/2}\alpha_k+2^{dk/2}\sum_{k'\geq
        k}2^{-dk'}\alpha_{k'}\beta_{k'},
    \end{split}
  \end{equation*}
  and the bound \eqref{basicest18} follows. To prove
  \eqref{basicest5.1} we use a similar argument, but estimate
  $\|P_{\leq k-5}g\|_{L^\infty_{x,t}}\lesssim 1$ instead of $\|P_{\leq
    k-5}g\|_{L^\infty_{x,t}}\lesssim \sum_{k_2\leq k}\beta_{k_2}$. For
  \eqref{basicest5.2}, we also estimate $\|P_{\leq
    k-5}f\|_{L^\infty_{x,t}}\lesssim 1$ instead of $\|P_{\leq
    k-5}f\|_{L^\infty_{x,t}}\lesssim \sum_{k_2\leq k}\alpha_{k_2}$.
\end{proof}

\begin{lemma}\label{lemmac3}
  Assume $f,g,h\in C^\infty(\R^d\times(-T,T):\mathbb{C})$ and
  $P_kf,P_kg,P_kh\in L^\infty_tL^2_x$ for any $k\in\mathbb{Z}$. Let
  \begin{equation*}
    \mu_k=\sum_{|k'-k|\leq 20}2^{dk'/2}(\|P_{k'}f\|_{L^\infty_tL^2_x}+\|P_{k'}g\|_{L^\infty_tL^2_x}+\|P_{k'}h\|_{L^\infty_tL^2_x}).
  \end{equation*}
  We assume that
  $\|f\|_{L^\infty_{x,t}}+\|g\|_{L^{\infty}_{x,t}}+\|h\|_{L^{\infty}_{x,t}}\leq
  1$ and $\sup_{k\in\Z}\mu_k\leq 1$. Then, for any $k\in\Z$ and
  $l,m=1,\ldots,d$
  \begin{equation}\label{basicest3}
    2^{dk/2}\|P_k(f\partial_lg\partial_mh)\|_{L^\infty_tL^2_x}\lesssim
    2^{2k}\big[\mu_k\sum_{k'\leq k}2^{-|k'-k|}\mu_{k'}+\sum_{k'\geq k}2^{-(d-2)|k'-k|}\mu^2_{k'}\big].
  \end{equation}
\end{lemma}

\begin{proof}[Proof of Lemma \ref{lemmac3}] By symmetry, it suffices
  to estimate
  \begin{equation*}
    \sum_{k_2\in\Z}\sum_{k_1\leq k_2}2^{dk/2}\|P_k(f\cdot P_{k_1}(\partial_lg)\cdot P_{k_2}(\partial_mh))\|_{L^\infty_tL^2_x},
  \end{equation*}
  which is dominated by
  \begin{equation*}
    \begin{split}
      &\sum_{k_2\geq k-4}\sum_{k_1\leq k_2}2^{dk/2}\|P_k(P_{\leq k_2-5}(f)\cdot P_{k_1}(\partial_lg)\cdot P_{k_2}(\partial_mh))\|_{L^\infty_tL^2_x}\\
      &+\sum_{k_2\in\Z}\sum_{k_1\leq k_2}\sum_{k_3\geq k_2+5}2^{dk/2}\|P_k(P_{k_3}(f)\cdot P_{k_1}(\partial_lg)\cdot P_{k_2}(\partial_mh))\|_{L^\infty_tL^2_x}\\
      &+\sum_{k_2\in\Z}\sum_{k_1\leq k_2}\sum_{|k_3-k_2|\leq
        4}2^{dk/2}\|P_k(P_{k_3}(f)\cdot P_{k_1}(\partial_lg)\cdot
      P_{k_2}(\partial_mh))\|_{L^\infty_tL^2_x}=I+II+III.
    \end{split}
  \end{equation*}
  We estimate
  \begin{equation*}
    \begin{split}
      I&\leq \sum_{|k_2-k|\leq 4}\sum_{k_1\leq k_2}2^{dk/2}\|P_k(P_{\leq k_2-5}(f)\cdot P_{k_1}(\partial_lg)\cdot P_{k_2}(\partial_mh))\|_{L^\infty_tL^2_x}\\
      &+\sum_{k_2\geq k+5}\sum_{|k_1-k_2|\leq 4}2^{dk/2}\|P_k(P_{\leq k_2-5}(f)\cdot P_{k_1}(\partial_lg)\cdot P_{k_2}(\partial_mh))\|_{L^\infty_tL^2_x}\\
      &\lesssim \sum_{|k_2-k|\leq 4}\sum_{k_1\leq k_2}2^{dk/2}\|P_{\leq k_2-5}(f)\|_{L^\infty_{x,t}}\|P_{k_1}(\partial_lg)\|_{L^\infty_{x,t}}\|P_{k_2}(\partial_mh)\|_{L^\infty_tL^2_x}\\
      &+ \sum_{k_2\geq k+5}\sum_{|k_1-k_2|\leq 4}2^{dk/2}2^{dk/2}\|P_{\leq k_2-5}(f)\|_{L^\infty_{x,t}}\|P_{k_1}(\partial_lg)\|_{L^\infty_tL^2_x}\|P_{k_2}(\partial_mh)\|_{L^\infty_tL^2_x}\\
      &\lesssim 
\sum_{k_1\leq k}2^{k_1}\mu_{k_1}2^k\mu_k+C\sum_{k_2\geq
        k}2^{-d|k_2-k|}2^{2k_2}\mu_{k_2}^2,
    \end{split}
  \end{equation*}
  which is dominated by the right-hand side of \eqref{basicest3}. We
  estimate
  \begin{equation*}
    \begin{split}
      II&\leq \sum_{|k_3-k|\leq 5}\sum_{k_1,k_2\leq k+5}2^{dk/2}\|P_k(P_{k_3}(f)\cdot P_{k_1}(\partial_lg)\cdot P_{k_2}(\partial_mh))\|_{L^\infty_tL^2_x}\\
      &\lesssim \sum_{|k_3-k|\leq 5}\sum_{k_1,k_2\leq k+5}2^{dk/2}\|P_{k_3}(f)\|_{L^\infty_tL^2_x}\|P_{k_1}(\partial_lg)\|_{L^\infty_{x,t}}\|P_{k_2}(\partial_mh))\|_{L^\infty_{x,t}}\\
      &\lesssim\mu_k\big(\sum_{k_1\leq k}2^{k_1}\mu_{k_1}\big)^2,
    \end{split}
  \end{equation*}
  which is dominated by the right-hand side of \eqref{basicest3} since
  $\sup_{k\in\Z}\mu_{k}\leq 1$. Finally, we estimate
  \begin{equation*}
    \begin{split}
      III&\leq \sum_{k_2\geq k-8}\sum_{k_1\leq k_2}\sum_{|k_3-k_2|\leq 4}2^{dk/2}\|P_k(P_{k_3}(f)\cdot P_{k_1}(\partial_lg)\cdot P_{k_2}(\partial_mh))\|_{L^\infty_tL^2_x}\\
      & \lesssim \sum_{k_2\geq k-8}\sum_{k_1\leq k_2}\sum_{|k_3-k_2|\leq 4}2^{dk}\|P_{k_3}(f)\|_{L^\infty_tL^2_x}\|P_{k_1}(\partial_lg)\|_{L^\infty_{x,t}}\|P_{k_2}(\partial_mh))\|_{L^\infty_tL^2_x}\\
      & \lesssim \sum_{k_2\geq
        k}2^{dk}2^{-dk_2}2^{k_2}\mu_{k_2}^2\sum_{k_1\leq
        k_2}2^{k_1}\mu_{k_1}
    \end{split}
  \end{equation*}
  which is dominated by the right-hand side of \eqref{basicest3} since
  $\sup_{k\in\Z}\mu_{k}\leq 1$. This completes the proof of
  \eqref{basicest3}.
\end{proof}

We construct now the function $\tphi$.

\begin{lemma}\label{heatflow}
  Assume $\phi\in {H}_Q^{\infty,\infty}(T)$ and
  $\{\gamma_k(\sigma)\}_{k\in\Z}$ be as in \eqref{la2}. Then there is
  a unique global solution $\tphi\in
  C([0,\infty):{H}_Q^{\infty,\infty}(T))$ of the
  initial-value problem
  \begin{equation}\label{heat3.1}
    \begin{cases}
      &\partial_s\tphi=\Delta_x\tphi+\tphi\cdot\sum_{m=1}^d|\partial_m\tphi|^2\quad\text{ on }[0,\infty)\times\R^d\times(-T,T);\\
      &\tphi(0)=\phi.
    \end{cases}
  \end{equation}
  In addition, for any $k\in\Z$, $s\in[0,\infty)$, and
  $\sigma\in [0,\sigma_1]$
  \begin{equation}\label{big9}
    \|P_k(\tphi(s))\|_{L^\infty_tL^2_x}\lesssim 
(1+s2^{2k})^{-\sigma_1}2^{-\sigma k}\gamma_k(\sigma),
  \end{equation}
  and, for any $\sigma,\rho\in\Z_+$,
  \begin{equation}\label{big10}
    \sup_{s\in[0,\infty)}(s+1)^{\sigma/2}\|\nabla_x^\sigma\partial_t^\rho(\tphi(s)-Q)\|_{L^\infty_tL^2_x}<\infty.
  \end{equation}
\end{lemma}

\begin{proof}[Proof of Proposition \ref{heatflow}] Let
  $M=\sum_{k\in\Z}(2^{2\sigma_1k}+1)\|P_k(\phi)\|_{L^\infty_tL^2_x}^2$.
  A simple fixed-point argument shows that there is $S=S(M)>0$ and a
  unique smooth solution $\tphi$ for \eqref{heat3.1}, with
  $\tphi-Q\in C([0,S]:{H}^{\infty,\infty}(T))$. In addition, for any
  $\sigma,\rho\in\Z_+$
  \begin{equation}\label{smoothbound}
    \sup_{s\in[0,S]}\|\tphi(s)-Q\|_{{H}^{\sigma,\rho}(T)}
\leq C(M,s,\rho,\|\phi-Q\|_{{H}^{\sigma,\rho}(T)}),
  \end{equation}
  and
  \begin{equation*}
    \partial_s({}^t\tphi\cdot\phi-1)=2{}^t\tphi\cdot\Delta_x\tphi+2\sum_{m=1}^d|\partial_m\tphi|^2+2({}^t\tphi\cdot\phi-1)\sum_{m=1}^d|\partial_m\tphi|^2.
  \end{equation*}
  This shows that $|\tphi|\equiv 1$, thus
  $\tphi\in C([0,S]:{H}_Q^{\infty,\infty}(T))$.
  We prove now a priori estimates on the solution $\tphi$.

  For any $S'\in[0,S]$ we define
  \begin{equation}\label{mainnorm}
    B_1(S')=\sup_{s\in[0,S']}\sup_{k\in\Z}(1+s2^{2k})^{\sigma_1}2^{kd/2}\gamma_k^{-1}\|P_k(\tphi(s))\|_{L^\infty_tL^2_x}.
  \end{equation}
  It is easy to see that $B_1:[0,S]\to(0,\infty)$ is a well-defined
  continuous nondecreasing function and $\lim_{S'\to 0}B_1(S')\leq 1$.

  Using Duhamel's principle, for any $k\in\Z$ and $s\in[0,S]$
  \begin{equation}\label{Duhamel}
    P_k(\tphi(s))=e^{s\Delta}(P_k\phi)+\int_0^s e^{(s-s')\Delta}\Big[P_k\big[\tphi(s')\cdot\sum_{m=1}^d|\partial_m\tphi(s')|^2\big]\Big]\,ds'.
  \end{equation}
Hence for any $k\in\Z$ and
  $s\in[0,S']$
  \begin{equation}\label{big1}
    \begin{split}
      2^{dk/2}&\|P_k(\tphi(s))\|_{L^\infty_tL^2_x}\leq e^{-s2^{2k-2}}2^{dk/2}\|P_k\phi\|_{L^\infty_tL^2_x}\\
      &+\int_0^se^{-(s-s')2^{2k-2}}2^{dk/2}\left\|P_k\big[\tphi(s')\cdot\sum_{m=1}^d|\partial_m\tphi(s')|^2\big]\right\|_{L^\infty_tL^2_x}\,ds'.
    \end{split}
  \end{equation}
  In view of the definition \eqref{mainnorm}, for any $s'\in[0,S']$
  and $k'\in\Z$
  \begin{equation*}
    2^{dk'/2}\|P_{k'}(\tphi(s'))\|_{L^\infty_tL^2_x}\leq B_1(S')(1+s'2^{2k'})^{-\sigma_1}\gamma_{k'}.
  \end{equation*}
  It follows from \eqref{basicest3} (since $d\geq 2$) that, for any
  $k\in\Z$ and $s'\in[0,S']$,
  \begin{equation}\label{big2}
    2^{dk/2}\Big\|P_k\big[\tphi(s')\cdot\sum_{m=1}^d|\partial_m\tphi(s')|^2\big]\Big\|_{L^\infty_tL^2_x}\lesssim
2^{2k}B_1(S')^2\sum_{k'\geq k}\gamma_{k'}^2(1+s'2^{2k'})^{-\sigma_1}.
  \end{equation}
  We substitute this bound into \eqref{big1} and integrate in $s'$. We
  notice that\begin{equation}\label{simple}
    \int_0^se^{-(s-s')\lambda}(1+s'\lambda')^{-\sigma_1}\,ds'\lesssim
    s(1+\lambda s)^{-\sigma_1}(1+\lambda's)^{-1}
  \end{equation}
  for any $s\geq 0$ and $0\leq\lambda\leq\lambda'$. Using \eqref{la2},
   for any $s\in[0,S']$ and $k\in\Z$ we get
  \begin{equation}\label{big3}
    \begin{split}
      &2^{dk/2}\|P_k(\tphi(s))\|_{L^\infty_tL^2_x}(1+s2^{2k})^{\sigma_1}\\
      &\lesssim \gamma_k+B_1(S')^22^{2k}(1+s2^{2k})^{\sigma_1}\int_0^se^{-(s-s')2^{2k}/4}\sum_{k'\geq k}\gamma_{k'}^2(1+{s'}2^{2k'})^{-\sigma_1}\,ds'\\
      &\lesssim \gamma_k+B_1(S')^22^{2k}s\sum_{k'\geq
        k}\gamma_{k'}^2(1+2^{2k'}s)^{-1}\\&\lesssim
      \gamma_k+ B_1(S')^2\gamma\gamma_k,
    \end{split}
  \end{equation} 
which gives
\[
B_1(S') \lesssim 1+ \gamma  B_1(S')^2
\]
Since $B_1$ is continuous and
$B_1(0)\leq 1$, it follows that $B_1(S')\lesssim  1$ for any $S'\in[0,S]$
(provided that $\gamma$ is sufficiently small). Thus 
for any $ k\in\Z$ and $s\in[0,S]$
  \begin{equation}\label{big6}
    2^{dk/2}\|P_k(\tphi(s))\|_{L^\infty_tL^2_x}\leq C(1+s2^{2k})^{-\sigma_1}\gamma_k.
  \end{equation}

  We control now $(1+s2^{2k})^{\sigma_1}2^{\sigma
    k}\|P_k(\tphi(s))\|_{L^\infty_tL^2_x}$,
  $\sigma\in [0,\sigma_1]$. We define
  \begin{equation*}
    B_2(S)=\sup_{\sigma\in\{0,\sigma_1\}}\sup_{s\in[0,S]}\sup_{k\in\Z}(1+s2^{2k})^{\sigma_1}2^{\sigma k}\gamma_k(\sigma)^{-1}\|P_k(\tphi(s))\|_{L^\infty_tL^2_x}.
  \end{equation*}
It is easy to see that $B_2(S)<\infty$. It follows from
  \eqref{Duhamel} that
  \begin{equation}\label{big11}
    \begin{split}
      &\|P_k(\tphi(s))\|_{L^\infty_t L^2_x}2^{\sigma k}(1+s2^{2k})^{\sigma_1}\leq (1+s2^{2k})^{\sigma_1}e^{-s2^{2k-2}}\gamma_k(\sigma)\\
      &+(1+s2^{2k})^{\sigma_1}\int_0^se^{-(s-s')2^{2k-2}}2^{\sigma
        k}\Big\|P_k\big[\tphi(s')\cdot\sum_{m=1}^d|\partial_m\tphi(s')|^2\big]\Big\|_{L^\infty_t
        L^2_x}\,ds'
    \end{split}
  \end{equation}
  for any  $s\in[0,S]$. Using
  the definition of $B_2(S)$ we have
  \begin{equation}\label{big12}
    2^{dk'/2}\|P_{k'}(\tphi(s'))\|_{L^\infty_tL^2_x}\lesssim B_2(S)(1+s'2^{2k'})^{-\sigma_1}2^{dk'/2}2^{-\sigma k'}\gamma_{k'}(\sigma).
  \end{equation}
  We combine this with \eqref{big6}. It follows from \eqref{basicest3}
  that 
  \begin{equation*}
    \begin{split}
      &2^{\sigma
        k}\Big\|P_k\big[\tphi(s')\cdot\sum_{m=1}^d|\partial_m\tphi(s')|^2\big]\Big\|_{L^\infty_tL^2_x}\lesssim
      2^{2k}B_2(S)\sum_{k'\geq
        k}2^{|k'-k|}(1+s'2^{2k'})^{-\sigma_1}\gamma_{k'}\gamma_{k'}(\sigma),
    \end{split}
  \end{equation*}
  for any $k\in\Z$ and $s\in[0,S]$. Using \eqref{simple}
  \begin{equation*}
    \begin{split}
      &(1+s2^{2k})^{\sigma_1}\int_0^se^{-(s-s')2^{2k-2}}2^{\sigma k}\Big\|P_k\big[\tphi(s')\cdot\sum_{m=1}^d|\partial_m\tphi(s')|^2\big]\Big\|_{L^\infty_tL^2_x}\,ds'\\
      &\lesssim sB_2(S)\sum_{k'\geq
        k}2^{k-k'}\gamma_{k'}\gamma_{k'}(\sigma)(1+s2^{2k'})^{-1}\lesssim
      B_2(S)\gamma_k\gamma_k(\sigma).
    \end{split}
  \end{equation*}
  It follows from \eqref{big11} that
  \begin{equation*}
 B_2(s) \lesssim \sup_{k \in \Z} \sup_{s \in [0,S]}
   \gamma_k(\sigma)^{-1}\|P_k(\tphi(s))\|_{L^\infty_tL^2_x}2^{\sigma k}(1+s2^{2k})^{\sigma_1}\lesssim  1+\gamma B_2(S),
  \end{equation*} 
 Since $\gamma$ is small this gives  $B_2(S)\lesssim 1$. Thus 
for any $k\in\Z$, $s\in[0,S]$, $ \sigma\in [0,\sigma_1]$,
  \begin{equation}\label{kj1}
    2^{\sigma k}\|P_k(\tphi(s))\|_{L^\infty_tL^2}\lesssim 
\gamma_k(\sigma)(1+s2^{2k})^{-\sigma_1}.
  \end{equation}
  As a consequence, the solution $\tphi$ can be extended
  globally to a smooth solution $\tphi\in
  C([0,\infty):{H}_Q^{\infty,\infty}(T))$ of \eqref{heat3.1}. The
  bound \eqref{big9} follows from \eqref{big6} and \eqref{kj1}.

  It remains to prove the bound \eqref{big10}, which follows from
  \eqref{big9} (applied for $\sigma=0$ and $\sigma=\sigma_1$) for
  $\rho=0$ and $\sigma\leq\sigma_1$. It follows also from \eqref{big9}
  that for any $s\geq S_0=(M/\gamma)^4\gg 1$,
  $\sigma\in[0,\sigma_1+10]\cap\Z$, and
  $m_1,\ldots,m_\sigma\in\{1,\ldots,d\}$
  \begin{equation}\label{linfty}
    \begin{split}
      \|\partial_{m_1}\ldots\partial_{m_\sigma}(\tphi(s)-Q)\|_{L^\infty_{x,t}}&\lesssim \sum_{k\in\Z}2^{dk/2}2^{k\sigma}\|P_k(\tphi(s)-Q)\|_{L^\infty_tL^2_x}\\
      &\lesssim \sum_{k\in\Z}2^{dk/2}2^{k\sigma}(1+s2^{2k})^{-\sigma_1}M\\
      &\lesssim Ms^{-1/2}s^{-\sigma/2}
    \end{split}
  \end{equation}
  and
  \begin{equation}\label{l2}
      \|\partial_{m_1}\ldots\partial_{m_\sigma}(\tphi(s)-Q)\|_{L^\infty_tL^2_x}\lesssim \big[\sum_{k\in\Z}2^{2k\sigma}\|P_k(\tphi(s)-Q)\|_{L^\infty_tL^2_x}^2\big]^{1/2}\lesssim Ms^{-\sigma/2}.
  \end{equation}
  For $S\geq S_0$ let
  \begin{equation*}
    M_{\rho,\sigma}(S)=M+1+\sum_{\sigma'\leq\sigma}\sum_{\rho'\leq \rho}\sup_{s\in[S_0,S]}s^{\sigma'/2}\|\nabla_x^{\sigma'}\partial_t^{\rho'}(\tphi(s)-Q)\|_{L^\infty_tL^2_x}.
  \end{equation*}
  As in \eqref{linfty},
  \begin{equation}\label{linfty2}
    s^{\sigma'/2}\|\nabla_x^{\sigma'}\partial_t^{\rho'}(\tphi(s)-Q)\|_{L^\infty_{x,t}}\lesssim s^{-1/2}M_{\rho,\sigma}(S)
  \end{equation}
  for any $s\in[S_0,S]$, $\rho'\leq\rho$, and $\sigma'<\sigma-d/2$.

  We prove first \eqref{big10} for $\rho=0$ and $\sigma\geq
  \sigma_1+1$. We use induction over $\sigma$, \eqref{linfty},
  \eqref{linfty2}, and \eqref{heat3.1} to estimate, for $s\geq 2S_0$,
  \begin{equation}\label{estimate}
    \begin{split}
      s^{\sigma/2}\|\nabla_x^\sigma&(\tphi(s))\|_{L^\infty_tL^2_x}\leq s^{\sigma/2}\|e^{(s-S_0)\Delta}(\nabla_x^\sigma \tphi(S_0))\|_{L^\infty_tL^2_x}\\
      &+s^{\sigma/2}\int_{S_0}^s\big\|e^{(s-s')\Delta}[\nabla_x^\sigma(\tphi(s')\cdot \sum_{m=1}^d|\partial_m\tphi(s')|^2)]\big\|_{L^\infty_tL^2_x}\,ds'\\
      \lesssim & \ \|\tphi(S_0)\|_{L^\infty_tL^2_x}+\int_{S_0}^{s/2}\big\|\tphi(s')\cdot \sum_{m=1}^d|\partial_m\tphi(s')|^2)]\big\|_{L^\infty_tL^2_x}\,ds'\\
      &+s^{\sigma/2}\int_{s/2}^s\frac{1}{(s-s')^{1/2}}\big\|\nabla_x^{\sigma-1}(\tphi(s')\cdot \sum_{m=1}^d|\partial_m\tphi(s')|^2)\big\|_{L^\infty_tL^2_x}\,ds'\\
      \lesssim &\ (C_\sigma
      Ms^{-1/2})M_{0,\sigma}(s)+C_\sigma(M_{0,\sigma-1}(s))^3.
    \end{split}
  \end{equation}
  It follows by induction that $\sup_{s\geq
    S_0}M_{0,\sigma}(s)<\infty$ for any $\sigma\in\Z_+$, and
  \eqref{big10} follows for $\rho=0$.

  To prove \eqref{big10} for a pair $(\rho,\sigma)\in\Z_+\times\Z_+$,
  $\rho\geq 1$, we may assume by induction that $\sup_{s\geq
    S_0}M_{\rho-1,\sigma'}(s)<\infty$ for any $\sigma'\in\Z$ and
  $\sup_{s\geq S_0}M_{\rho,\sigma-1}(s)<\infty$. In view of
  \eqref{heat3.1}, the function
  $v_\rho=\partial_t^\rho\tphi$ solves the heat equation
  \begin{equation}\label{v1eq}
    \begin{split}
      &(\partial_s-\Delta_x)v_\rho=v_\rho\cdot\sum_{m=1}^d|\partial_m\tphi|^2+2\tphi\cdot\sum_{m=1}^d{}^t(\partial_m\tphi)\cdot\partial_mv_\rho+E_{\rho-1}\\
      &=\sum_{m=1}^d\big[|\partial_m\tphi|^2I_3-2\partial_m(\tphi\cdot {}^t(\partial_m\tphi))\big]\cdot v_\rho+\sum_{m=1}^d\partial_m[2\tphi\cdot{}^t(\partial_m\tphi)\cdot v_\rho]+E_{\rho-1}\\
      &=P\cdot v_\rho+\sum_{m=1}^d\partial_m(Q_m\cdot
      v_\rho)+E_{\rho-1}.
    \end{split}
  \end{equation}
  It follows from \eqref{linfty} that for any $s\geq S_0$
  \begin{equation}\label{PmQm}
    s^{3/2}\|P\|_{L^\infty_{x,t}}+s\sum_{m=1}^d\|Q_m\|_{L^\infty_{x,t}}
\lesssim M.
  \end{equation}
  In addition, it follows from \eqref{linfty2} and the induction
  hypothesis that
  \begin{equation*}
    \sup_{s\geq S_0}\big[s^{3/2}s^{\sigma'/2}\|\nabla_x^{\sigma'}P\|_{L^\infty_{x,t}}+s^{\sigma'/2+1}\|\nabla_x^{\sigma'}Q_m\|_{L^\infty_{x,t}}+s^{3/2}s^{\sigma'/2}\|\nabla_x^{\sigma'}E_{\rho-1}\|_{L^\infty_tL^2_x}\big]<\infty
  \end{equation*}
  for any $\sigma'\in\Z_+$ and $m=1,\ldots d$. An estimate similar to
  \eqref{estimate} shows that $\sup_{s\geq
    S_0}M_{\rho,\sigma}(s)<\infty$, which completes the proof of the
  lemma.
\end{proof}

We construct now the function $v$.

\begin{lemma}\label{vframe}
  There is a smooth function
  $v:[0,\infty)\times\R^d\times(-T,T)\to\mathbb{S}^2$ such that
  \begin{equation}\label{vprop1} {}^tv\cdot\tphi=
    0\quad\text{ and
    }\quad\partial_sv=[\partial_s\tphi\cdot{}^t\tphi-\tphi\cdot
    {}^t(\partial_s\tphi)]\cdot v
  \end{equation}
  on $[0,\infty)\times\R^d\times(-T,T)$. In addition, for any
  $k\in\Z$, $s\in[0,\infty)$, $\sigma\in[d/2,\sigma_1]$,
  \begin{equation}\label{vprop2}
    2^{\sigma k}\|P_k(v(s))\|_{L^\infty_tL^2_x}\lesssim 
(1+s2^{2k})^{-\sigma_1+1}\gamma_k(\sigma),
  \end{equation}
  and, for any $\sigma,\rho\in\Z_+$,
  \begin{equation}\label{vprop3}
    \sup_{s\in[0,\infty)}\sup_{k\in\Z}(s+1)^{\sigma/2}2^{k\sigma}\|P_k(\partial_t^\rho(v(s))\|_{L^\infty_tL^2_x}<\infty.
  \end{equation}
\end{lemma}

\begin{proof}[Proof of Lemma \ref{vframe}] Let $R$ denote the
  $3\times 3$ matrix
  \begin{equation}\label{Rma}
    R=\partial_s\tphi\cdot{}^t\tphi-\tphi\cdot {}^t(\partial_s\tphi)=\Delta_x\tphi\cdot{}^t\tphi-\tphi\cdot{}^t(\Delta_x\tphi)=\sum_{m=1}^d\partial_m[\partial_m\tphi\cdot{}^t\tphi-\tphi\cdot{}^t(\partial_m\tphi)],
  \end{equation}
  where one of the identities follows from \eqref{heat3.1}. It follows
  from \eqref{big9} and \eqref{basicest5.1} that
  \begin{equation}\label{mainR}
    2^{\sigma k}\|P_k(R(s))\|_{L^\infty_tL^2_x}\lesssim 
2^{2k}(1+s2^{2k})^{-\sigma_1}\gamma_k(\sigma)
  \end{equation}
  for any $k\in\Z$, $s\in[0,\infty)$, $\sigma\in[d/2,\sigma_1]$. It
  follows from \eqref{big10} and \eqref{linfty2} that
  \begin{equation}\label{mainR2}
    \sup_{s\in[0,\infty)}\big[(s+1)^{(\sigma+2)/2}\|\nabla_x^\sigma\partial_t^\rho(R(s))\|_{L^\infty_tL^2_x}+(s+1)^{(\sigma+3)/2}\|\nabla_x^\sigma\partial_t^\rho(R(s))\|_{L^\infty_{x,t}}\big]<\infty
  \end{equation}
  for any $\sigma,\rho\in\Z_+$.

  We prove first the existence of a smooth function
  $v:[0,\infty)\times\R^d\times(-T,T)\to\mathbb{S}^2$ satisfying
  \eqref{vprop1}.  We fix $Q'\in\mathbb{S}^{d-1}$ 
 with the property that ${}^tQ\cdot Q'=0$.
 By  \eqref{mainR2} we have
\[
\int_0^\infty\|R(s)\|_{L^\infty_{x,t}}\,ds<\infty.
\]
This allows us to construct $v:[0,\infty)\times\R^d\times(-T,T)\to
\R^3$ (by a simple fixed point argument) as the unique solution of the
ODE
  \begin{equation}\label{ODE}
    \partial_sv=R(s)\cdot v\text{ and }v(\infty)=Q'
  \end{equation}
  for any $(x,t)\in\R^d\times(-T,T)$.

Since
  $\int_0^\infty\|\nabla_x^\sigma\partial_t^\rho(R(s))\|_{L^\infty_{x,t}}\,ds<\infty$
  for any $\sigma,\rho\in\Z_+$, the function $v$ constructed as a
  solution of \eqref{ODE} is smooth on
  $[0,\infty)\times\R^d\times(-T,T)$ and
  \begin{equation}\label{vs1}
    \sup_{s\in[0,\infty)}(s+1)^{(\sigma+1)/2}\|\nabla_x^\sigma\partial_t^\rho(v(s)-Q')\|_{L^\infty_{x,t}}<\infty
  \end{equation}
  for any $\sigma,\rho\in\Z_+$. Using \eqref{ODE} and
  ${}^t(\partial_s\tphi)\cdot\tphi=0$, it is
  easy to see that
  \begin{equation*}
    \partial_s({}^tv\cdot\tphi)={}^tv\cdot[\tphi\cdot {}^t(\partial_s\tphi)-\partial_s\tphi\cdot{}^t\tphi]\cdot\tphi+{}^tv\cdot\partial_s\tphi=0.
  \end{equation*}
  Since $\lim_{s\to\infty}{}^tv(s)\cdot\tphi(s)=0$, it
  follows that ${}^tv\cdot\tphi\equiv 0$ on
  $[0,\infty)\times\R^d\times(-T,T)$. Thus, using \eqref{ODE} again,
  \begin{equation*}
    \partial_s({}^tv\cdot v)=2{}^tv\cdot [\partial_s\tphi\cdot{}^t\tphi-\tphi\cdot {}^t(\partial_s\tphi)]\cdot v=0.
  \end{equation*}
  Since $\lim_{s\to\infty}{}^tv(s)\cdot v(s)=1$ it follows that
  ${}^tv\cdot v\equiv 1$ on $[0,\infty)\times\R^d\times(-T,T)$. To
  summarize, we constructed a smooth function
  $v:[0,\infty)\times\R^d\times(-T,T)\to\mathbb{S}^2$ that satisfies
  \eqref{vprop1}.

  We prove now \eqref{vprop3}. In view of \eqref{ODE}, we have
  \begin{equation*}
    v(s)-Q'+\int_s^\infty R(s')\cdot Q'\,ds'=-\int_s^\infty R(s')\cdot (v(s')-Q')\,ds'.
  \end{equation*}
  Thus, using \eqref{mainR2} and \eqref{vs1},
  \begin{equation*}
    \sup_{s\in[0,\infty)}(s+1)^{(\sigma+1)/2}\Big\|\nabla_x^\sigma\partial_t^\rho(v(s)-Q')+\int_s^\infty \nabla_x^\sigma\partial_t^\rho(R(s'))\cdot Q'\,ds'\Big\|_{L^\infty_tL^2_x}<\infty
  \end{equation*}
  for any $\sigma,\rho\in\Z_+$. The bound \eqref{vprop3} follows from
  \eqref{mainR2}.

  Finally, we prove \eqref{vprop2}. It follows from \eqref{ODE} that
  \begin{equation}\label{eqv}
    P_k(v(s))=-\int_s^{\infty}P_k(R(s')\cdot v(s'))\,ds'.
  \end{equation}
  For any $S\in[0,\infty)$ let
  \begin{equation*}
    B_3(S)=1+\sup_{\sigma\in[d/2,\sigma_1]}\sup_{s'\in[S,\infty)}\sup_{k\in\Z}\gamma_k(\sigma)^{-1}(1+s'2^{2k})^{\sigma_1-1}2^{\sigma k}\|P_k(v(s'))\|_{L^\infty_tL^2_x}.
  \end{equation*}
  We have $B_3(S)<\infty$ for any $S\in[0,\infty)$, using
  \eqref{vprop3} and
  $\sup_{k\in\Z}\gamma_k(\sigma)^{-1}2^{-\delta_0|k|}<\infty$. Also,
  \begin{equation*}
    2^{\sigma k'}\|P_{k'}(v(s'))\|_{L^\infty_tL^2_x}\leq B_3(S)\gamma_{k'}(\sigma)(1+s'2^{2k'})^{-\sigma_1+1}
  \end{equation*}
  for any $\sigma\in[d/2,\sigma_1]$, $s'\geq S$, and $k'\in\Z$. It
  follows from \eqref{basicest5.1} and \eqref{mainR} that
  \begin{equation*}
    2^{\sigma k}\|P_k(R(s')\cdot v(s'))\|_{L^\infty_tL^2_x}\lesssim 
2^{2k}\big(\gamma_k(\sigma)+\gamma B_3(S)
\sum_{2^{2k}\leq 2^{2k'}\leq 1/s'}\gamma_{k'}(\sigma)\big)
  \end{equation*}
  if $s'2^{2k}\leq 1$, and
  \begin{equation*}
    2^{\sigma k}\|P_k(R(s')\cdot v(s'))\|_{L^\infty_tL^2_x}\lesssim
2^{2k}(s'2^{2k})^{-\sigma_1}\gamma_k(\sigma)(1+\gamma B_3(S))
  \end{equation*}
  if $s'2^{2k}\geq 1$.  Thus, for $s\geq S$ and $k\in\Z$
  \begin{equation*}
    \int_s^{\infty}2^{\sigma k}\|P_k(R(s')\cdot
    v(s'))\|_{L^\infty_tL^2_x}\,ds'\lesssim 
\gamma_k(\sigma)(1+s2^{2k})^{-\sigma_1+1}(1+\gamma B_3(S)).
  \end{equation*}
  It follows from \eqref{eqv} that $B_3(S)\lesssim (1+\gamma B_3(S))$, which gives \eqref{vprop2}.
\end{proof}

We complete now the proof of Proposition \ref{TaoHeat}. We define the
smooth function $w=\tphi\times
v:[0,\infty)\times\R^d\times(-T,T)\to\mathbb{S}^2$. It follows from
\eqref{basicest5.2}, \eqref{big9}, and \eqref{vprop2}
that\begin{equation*} 2^{\sigma k}\|P_k(w(s))\|_{L^\infty_tL^2_x}\lesssim
  (1+s2^{2k})^{-\sigma_1+1}2^{-\sigma k}\gamma_k(\sigma)
\end{equation*}
for any $k\in\Z$, $s\in[0,\infty)$, $\sigma\in[d/2,\sigma_1]$. It
follows easily from \eqref{big10} and \eqref{vprop3} that
\begin{equation*}
  \sup_{s\in[0,\infty)}\sup_{k\in\Z}(s+1)^{\sigma/2}2^{k\sigma}\|P_k(\partial_t^\rho(w(s))\|_{L^\infty_tL^2_x}<\infty
\end{equation*}
for any $\rho,\sigma\in\Z_+$. Finally, the identities \eqref{vwprop}
follow from \eqref{vprop1} and $w=\tphi \times v$.

\end{document}